\documentclass[11pt,a4paper]{article}
\usepackage[utf8]{inputenc}
\DeclareUnicodeCharacter{2212}{-}
\usepackage[T1]{fontenc}
\usepackage[english]{babel}
\usepackage{csquotes}  
\usepackage{hyperref}
\usepackage{graphicx}
\usepackage{lmodern}
\usepackage{geometry}
\usepackage{amsmath}
\usepackage{amsthm}
\usepackage{indentfirst}
\usepackage{tabto}
\usepackage{color}
\usepackage{algorithm}
\usepackage{algorithmic}
\definecolor{bf}{rgb}{0,0,0.6} 
\usepackage{amssymb}
\usepackage{amsmath}
\usepackage{stmaryrd} 
\usepackage{fullpage}
\usepackage{eso-pic}
\usepackage{float}
\usepackage{titlesec}
\usepackage{bigints}
\usepackage{fancyhdr}
\usepackage{mathrsfs}
\usepackage[shortlabels]{enumitem}
\usepackage{scalerel,stackengine}
\usepackage{yhmath}
\usepackage{verbatim}
\usepackage{authblk}
\usepackage{needspace}
\usepackage{graphicx}
\usepackage{subcaption}
\usepackage{bbold}
\usepackage{pagesel}
\usepackage[section]{placeins}

\usepackage[disable]{todonotes}
\newcommand{\todoJules}[1]{\todo[color=green!30,line]{#1}} 

\setuptodonotes{line, color=blue!30}
\numberwithin{equation}{section}

\usepackage[maxnames=99,maxcitenames=3,backend=bibtex]{biblatex}
\bibliography{bibliography}
\AtEveryBibitem{\clearfield{month}\clearfield{url}\clearfield{issn}\clearfield{isbn}\clearfield{address}}
\DeclareFieldFormat[misc]{title}{\mkbibquote{#1}}
\DeclareBibliographyDriver{book}{%
	\usebibmacro{bibindex}%
	\usebibmacro{begentry}%
	\usebibmacro{author/editor+others/translator+others}%
	\setunit{\printdelim{nametitledelim}}\newblock
	\usebibmacro{maintitle+title}%
	\newunit
	\printlist{language}%
	\newunit\newblock
	\usebibmacro{byauthor}%
	\newunit\newblock
	\usebibmacro{byeditor+others}%
	\newunit\newblock
	\printfield{edition}%
	\newunit\newblock
	\usebibmacro{series+number}%
	\newunit
	\iffieldundef{maintitle}
		{\printfield{volume}%
		 \printfield{part}}
		{}%
	\newunit
	\printfield{volumes}%
	\newunit\newblock
	\printfield{note}%
	\newunit\newblock
	\usebibmacro{publisher+location+date}%
	\newunit\newblock
	\usebibmacro{chapter+pages}%
	\newunit
	\printfield{pagetotal}%
	\newunit\newblock
	\iftoggle{bbx:isbn}
		{\printfield{isbn}}
		{}%
	\newunit\newblock
	\usebibmacro{doi+eprint+url}%
	\newunit\newblock
	\usebibmacro{addendum+pubstate}%
	\setunit{\bibpagerefpunct}\newblock
	\usebibmacro{pageref}%
	\newunit\newblock
	\iftoggle{bbx:related}
		{\usebibmacro{related:init}%
		 \usebibmacro{related}}
		{}%
	\usebibmacro{finentry}}
\DeclareBibliographyDriver{inproceedings}{%
	\usebibmacro{bibindex}%
	\usebibmacro{begentry}%
	\usebibmacro{author/translator+others}%
	\setunit{\printdelim{nametitledelim}}\newblock
	\usebibmacro{title}%
	\newunit
	\printlist{language}%
	\newunit\newblock
	\usebibmacro{byauthor}%
	\newunit\newblock
	\usebibmacro{in:}%
	\usebibmacro{maintitle+booktitle}%
	\newunit\newblock
	\usebibmacro{event+venue+date}%
	\newunit\newblock
	\usebibmacro{byeditor+others}%
	\newunit\newblock
	\usebibmacro{series+number}%
	\newunit
	\iffieldundef{maintitle}
		{\printfield{volume}%
		 \printfield{part}}
		{}%
	\newunit
	\printfield{volumes}%
	\newunit\newblock
	\printfield{note}%
	\newunit\newblock
	\printlist{organization}%
	\newunit
	\usebibmacro{publisher+location+date}%
	\newunit\newblock
	\usebibmacro{chapter+pages}%
	\newunit\newblock
	\iftoggle{bbx:isbn}
		{\printfield{isbn}}
		{}%
	\newunit\newblock
	\usebibmacro{doi+eprint+url}%
	\newunit\newblock
	\usebibmacro{addendum+pubstate}%
	\setunit{\bibpagerefpunct}\newblock
	\usebibmacro{pageref}%
	\newunit\newblock
	\iftoggle{bbx:related}
		{\usebibmacro{related:init}%
		 \usebibmacro{related}}
		{}%
	\usebibmacro{finentry}}
\newcommand{\jules}[1]{\textcolor{red}{#1}}

\newcommand{\dd}{\mathrm{d}} 

\hypersetup{
	colorlinks=true,
	citecolor=green,
	linkcolor=blue,
}

\newtheorem{dft}{Definition}[section]
\newtheorem{te}[dft]{Theorem}
\newtheorem{rem}[dft]{Remark}
\newtheorem{prop}[dft]{Proposition}

\newtheorem{lemm}[dft]{Lemma}

\title{Estimating an initial telomere length distribution from the Laplace transform of its senescence times distribution}
\author{Jules Olayé$^{*}$}
\date{}
\begin{document}
	\maketitle
	\let\thefootnote\relax\footnotetext{$^{*}$  Université de Toulouse, CNRS, Institut de Mathématiques de Toulouse, 31062 Toulouse, France.}
	\let\thefootnote\relax\footnotetext{\hspace{2.525mm}Mail: \href{mailto:jules.olaye@math.univ-toulouse.fr}{jules.olaye@math.univ-toulouse.fr}}
	\begin{abstract}
	This work follows from a previous study on the estimation of an initial distribution of telomere length from a senescence times  distribution done in~\cite{olaye_transport_2026}. In this previous study, we have presented an estimation method based on the fact that our telomere shortening model can be approximated by a transport equation. This method has encouraging results, but fails to provide a good estimation when the variability of the initial  telomere length distribution is too small. We improve here this method by approximating our model with an advection-diffusion equation, which  allows us to better take into account the randomness of the shortening values. We show that under this approximation, there exists a simple link between the Laplace transform of the initial telomere length distribution and that of the senescence times distribution. Then, by using a numerical method for inverting Laplace transforms called \textit{Gaver-Stehfest algorithm}, we exploit this link to construct a new estimator.
	\end{abstract}
	\noindent\textit{\small{Keywords: Inverse problem, advection-diffusion equation, Laplace transform, Gaver-Stehfest algorithm, telomeres}}

	\section{Introduction}
	\paragraph{Motivation and presentation of the problem.} Telomeres are DNA sequences located at the end of chromosomes that shorten during cell divisions. When one of the telomeres of a cell becomes critically short, which is around $27$ base pairs for the yeast~\cite{rat_mathematical_2025}, this cell enters a state called~\textit{senescence}, where it stops to divide~\cite{abdallah_2009,bourgeron_2015,Martin2021,Xu2013}. The study of this phenomenon has gained in interest over the last few years, due to its link with the apparition of cancer~\cite{robinson_telomerase_2022}. This has led to the development of mathematical models representing the phenomenon, studied with both theoretical and numerical approaches~\cite{arino_mathematical_1995,benetos_stochastic_2025,benetos_branching_2025,bourgeron_2015,olaye_long-time_2026,rat_mathematical_2025}. Due to the fact that a cell becomes senescent when it has a too short telomere, the initial distribution of telomere lengths and the distribution of senescence times are deeply linked. We are interested here in better understanding this link, by solving the inverse problem which consists in estimating the initial telomere length distribution from a senescence times distribution in a single-telomere~model.
	
	
	
	\paragraph{Context of the study and previous results.} This work is the direct continuation of the article~\cite{olaye_transport_2026}. In this previous study, we have developed a telomere shortening model representing cells with one telomere. This model corresponds to a system of integro-differential equations representing the evolution of the telomere length distribution and the senescence times distribution, and is given here in~\eqref{eq:PDE_model_telomeres_one_telo_chaplaplace}. Then, we have proposed a first method to solve the inverse problem presented in the previous paragraph, based on a model approximation by a transport equation, see \hbox{\cite[Eq.~$3.5$]{olaye_transport_2026}}. The interest of this model approximation is that it allows us to retrieve a well-posed inverse problem, see~\cite[Remark~$2.2$]{olaye_transport_2026}. The main limitation of the estimation method presented in this previous article is that it does not work very well when the initial telomere length distribution has a small variability, see~\cite[Section~$5.1.2$]{olaye_transport_2026}. In the current work, our aim is to improve our estimation results on such initial distributions.
	
	
	The estimation strategy presented in~\cite{olaye_transport_2026} is inspired by the following works in protein depolymerisation~\hbox{\cite{armiento_estimation_2016,doumic_asymptotic_2026}}. In these two articles, the context is similar: a trait progressively degrades to~$0$, and the authors aim to retrieve the initial distribution from the density of the trait at $0$. In~\cite{armiento_estimation_2016}, the inverse problem has been solved thanks to a model approximation by a transport equation, obtained through a scaling limit. This type of approximation is called~\textit{first-order approximation}, and is the type of approximation used in~\cite{olaye_transport_2026}. In order to obtain better estimation results, this approximation has been refined in~\cite{doumic_asymptotic_2026} by approximating their model by a transport-diffusion equation. This type of approximation is called~\textit{second-order approximation}, and its main idea is to add a diffusion term to take into account the fluctuations around the scaling limit related to the central limit theorem. After having obtained this model approximation, the authors established observability results using Carleman inequalities, allowing them to justify an estimation method based on Tikhonov regularisation. They then developed a Kalman observer for the numerical reconstruction of the initial distribution.
	

	

	
	\paragraph{Estimation strategy.} The main reason that the estimator developed in~\cite{olaye_transport_2026} does not work when the initial distribution has a small variability is that the approximated model used to construct it does not take into account well the variability of the shortening values at each division. To manage this issue, in the current work, we develop an estimation strategy that better takes into account this variability. This strategy is inspired by the one followed in~\cite{olaye_transport_2026}, and consists in the three following steps. First, we refine the first-order model approximation obtained in~\cite{olaye_transport_2026}, by obtaining a second-order model approximation, as done in~\cite{doumic_asymptotic_2026}. Then, we show that there exists a simple link between the Laplace transform of the senescence times distribution and that of the initial telomere length distribution. Finally, we use a numerical method for inverting Laplace transforms called Gaver-Stehfest algorithm~\cite{gaver_observing_1966,stehfest_algorithm_1970,stehfest_remark_1970} to retrieve the initial distribution of telomere length.
	
	\paragraph{Contributions.} Our first contribution is to provide precise bounds for the second-order approximation. As mentioned before, the first-order and second-order approximations are obtained by a scaling limit of our model. In~\cite{doumic_asymptotic_2026}, the authors proved on a similar discrete model that, denoting by $N > 0$ the scaling parameter used to obtain the approximation, the error in \hbox{$L^2$-norm} between their model and a transport-diffusion model is at most of order $1/N^{3/2}$. They also obtained a bound that increases linearly with time. We improve here this result by obtaining, on more restrictive assumptions, a bound of order $1/N^2$ for the pointwise approximation error, that can be extended to all the \hbox{$L^p$-norm}, where $p>0$. We also prove that the bound on the error decreases exponentially fast after a certain time. The bounds we have obtained on the second-order approximation error allow us to provide a first theoretical justification that telomere shortening models can be approximated by transport-diffusion equations. Indeed, this approximation has already been studied and justified from a modelling perspective in previous works~\cite{portillo_influence_2023,wattis_mathematical_2020}, but no theoretical results related to it have been obtained.

	Our second contribution is to present and rigorously analyse on simulations a new method for solving our inverse problem based on a Laplace transform inversion. This new method improves the estimation results presented in~\cite{olaye_transport_2026} when the initial distribution has a small coefficient of variation. By presenting this new method, we also complete what is done in~\cite{doumic_asymptotic_2026} by providing an alternative method for estimating the initial distribution of advection-diffusion~equations on $\mathbb{R}_+$ from the observation of its value at the boundary of the space. Compared to~\cite{doumic_asymptotic_2026}, the advantage of our new method is that it is based on a simple link between the initial distribution and the trait density at the boundary. As a result, our new method is easier to understand and better highlights the fact that there is a deep link between the initial distribution and the distribution of the trait at the boundary of the space on such models. For the moment, we have however not obtained a theoretical result on the estimation error done with this method, due to difficulties presented in Section~\ref{subsubsect:estimation_gaver_stehfest_noise_free}. The latter corresponds to a work in progress.

	
	
	\paragraph{Difficulty for the model approximation.} The main difficulty in obtaining a bound of order~$1/N^2$ is that the third derivative in space of the telomere length density in the approximated model explodes when $N \to +\infty$. This derivative is the main term that describes the error between the original model and the approximated model due to a Taylor expansion, see~\eqref{eq:model_approximation_intermediate_first}. Thus, it is necessary to control carefully the explosion of this derivative as $N\rightarrow +\infty$ to obtain a qualitative bound on the approximation error. Compared to~\cite{doumic_asymptotic_2026}, we take into account here the evolution of the impact of the third derivative with time. This allows us to improve the bound on the model approximation. To do this, we first rewrite the approximation error as the integral of several sub-errors, as done in~\cite{olaye_transport_2026}. Then, we show that these sub-errors dissipate over time by using a maximum principle. What changes compared to~\cite{olaye_transport_2026} is that we cannot use the method of characteristics to obtain our bounds when we prove the dissipation of the sub-errors. We use here instead that the second derivative in space of our approximated model can be explicitly written thanks to the heat kernel. This explicit representation allows us to simplify the use of a maximum principle to bound the second derivative in space of the approximation by the eigenfunctions of the operator associated to its equation. It also allows us to exploit heat kernel estimates when bounding the third derivative in space of the approximated model, see~Sections~\ref{subsect:auxiliary_statements_derivatives} and~\ref{subsect:proof_auxiliary_statements}.
	

	\paragraph{Plan of the article.} The paper is organised as follows. First, we present the model, its approximation, and the main results of the paper in Section~\ref{sect:model_approximation_result}.  Then, we prove that our model can be approximated by an advection-diffusion equation in Section~\ref{sect:proof_model_approximation}. Thereafter, we present our estimation method based on the Gaver-Stehfest algorithm in Section~\ref{sect:gaver_stehfest}. Finally, in Section~\ref{sect:discussion}, we discuss our results and the perspectives of this work. The appendices~\ref{appendix:proof_link_laplace}-\ref{sect:round_off_errors} are devoted to the proofs of certain statements given during the paper, and to additional information about the way we have done our simulations.

	\section{The model, its approximation, and the main results}\label{sect:model_approximation_result}
	
	This section is devoted to the presentation of the notations and of the results of this work. First, in Section~\ref{subsect:general_model}, we present our model for telomere shortening and its scaled version. Then, in Section~\ref{subsect:approximated_models}, we present the second-order approximation of this model. Finally, in Section~\ref{subsect:main_result_chaplaplace}, we present the assumptions and the main results of the paper. 
	\subsection{The model}\label{subsect:general_model}
	
	The model we use in this work is the same as the one presented in~\cite[Eq. $2.1$]{olaye_transport_2026}. It represents the evolution of the telomere length density $n$ and the senescence times density~$n_{\partial}$ when several cell lineages with a single telomere are tracked. We assume that cells divide at a rate $\tilde{b} > 0$, and that at each cell division, the telomere of the dividing cell is shortened by a random value. The distribution of the shortening values is represented by a density $\tilde{g} : [0,\tilde{\delta}] \rightarrow \mathbb{R}$, where $\tilde{\delta} >0$ is the maximum shortening value, which has a finite third moment. We also assume that cells go to senescence when their telomere length goes below a critical value $0$. Then, our model is given by the following system of integro-differential equations (see \cite[Section $2.1$]{olaye_transport_2026} for more information)

	\begin{equation}\label{eq:PDE_model_telomeres_one_telo_chaplaplace}
		\begin{cases}
			\partial_t n(t,x) = \tilde{b}\int_0^{\tilde{\delta}}n\left(t,x + v\right)\tilde{g}(v) \mathrm{d}v - \tilde{b}n(t,x), & \forall t\geq0,\,x\geq0,\\
			n_{\partial}(t) = \tilde{b}\int_0^{\tilde{\delta}} n(t,y)(1-\tilde{G}(y))\mathrm{d}y ,& \forall t\geq0, \\ 
			n(0,x) = n_0(x), & \forall x\geq 0,
		\end{cases}
	\end{equation}
	where $n_0 \in L^{1}\left(\mathbb{R}_+\right)$ is the initial distribution of telomere length, and $\tilde{G}(y) := \int_0^y \tilde{g}(s) \dd s$ for all~$y\in[0,\tilde{\delta}]$. In all this work, we assume that $n_0$ is non-negative and verifies $\int_{x\in\mathbb{R}} n_0(x) \dd x = 1$. Our aim here is to estimate $n_0$ from the senescence times distribution $n_{\partial}$.
	
	As in~\cite{olaye_transport_2026}, we need to consider the scaled version of this model to obtain a model approximation. To do so, we introduce a scaling parameter $N > 0$, and assume that there exist $\delta > 0$, $g : [0,\delta] \rightarrow \mathbb{R}_+$ a probability density function and $b >0$ such that
	$$
	\tilde{\delta} = \frac{\delta}{N}, \hspace{8mm}\forall x \in [0,\tilde{\delta}]:\,\tilde{g}(x) = Ng(Nx),\hspace{4mm}\text{ and }\hspace{4mm}\tilde{b} = bN. 
	$$
	This assumption, when $N$ is large, means that the telomere shortening value is small at each division, and that the cell division rate is large. It is biologically relevant, and we refer to~\cite[Section $2.3$]{olaye_transport_2026} for more information. Then, by proceeding as in~\cite[Section $3.1$]{olaye_transport_2026}, we obtain the following scaled version of the model~\eqref{eq:PDE_model_telomeres_one_telo_chaplaplace}, depending on the scaling parameter $N > 0$,
	\begin{equation}\label{eq:scaled_model}
		\begin{cases}
			\partial_t n^{(N)}(t,x) = bN\int_0^{\delta}\left[n^{(N)}\left(t,x + \frac{v}{N}\right)- n^{(N)}(t,x)\right]g\left(v\right) \dd v , & \forall t\geq0,\,x\geq0,\\
			n_{\partial}^{(N)}(t) = b\int_0^{\delta} n^{(N)}\left(t,\frac{v}{N}\right)\left(1-G(v)\right)\dd v ,& \forall t\geq0, \\ 
			n^{(N)}(0,x) = n_0\left(x\right), & \forall x\geq0,
		\end{cases}
	\end{equation}
	where $G(y) := \int_0^y g(s) \dd s$ for all~$y\in[0,\delta]$. The above model is the one for which we find an approximation, and on which we solve the inverse problem. Specifically, we do these by studying the behaviour of $n^{(N)}$ and $n_{\partial}^{(N)}$ when $N \rightarrow \infty$. 
	
	We conclude this section by giving the following result, which deals with the uniqueness of a solution 
	to the equation verified by $n^{(N)}$ in~\eqref{eq:scaled_model}, where $N>0$. Its proof relies on a classical fixed point argument, and is provided in~\cite[Section~$3.$A$.1$]{olaye_thesis_2025}.
	\begin{prop}[Well-posedness of~\eqref{eq:scaled_model}]\label{prop:uniqueness_solution}
	For all $N > 0$, there exists a unique non-negative solution in~$C\left(\mathbb{R}_+;\,L^1\left(\mathbb{R}_+\right)\right)$ to
		\begin{equation}\label{eq:scaled_model_first_third_lines}
			\begin{cases}
				\partial_t n^{(N)}(t,x) = bN\int_0^{\delta}\left[n^{(N)}\left(t,x + \frac{v}{N}\right)- n^{(N)}(t,x)\right]g\left(v\right) \dd v , & \forall t\geq0,\,x\geq0,\\
				n^{(N)}(0,x) = n_0\left(x\right), & \forall x\geq0.
			\end{cases}
		\end{equation}
	\end{prop}
	\begin{rem}
	The uniqueness of a solution to the full system given in~\eqref{eq:scaled_model} is a question that remains open actually. The reason is that we are not able to prove that the second line of~\eqref{eq:scaled_model} has a unique solution in $L^\infty\left(\mathbb{R}_+\right)$ or $L^1\left(\mathbb{R}_+\right)$. \todoJules{Phrase sur non-injectivité ou pas ?}
	\end{rem}
	\begin{rem}\label{rem:mass_conservation}
		In view of \cite[Proposition~A.$1$]{olaye_transport_2026}~and~\cite[Remark~$2.1$]{olaye_transport_2026}, $n^{(N)}$ and $n_{\partial}^{(N)}$ are non-negative for all $N>0$, and verify for all $t\geq0$
		$$
		\int_0^{\infty} n^{(N)}(t,x) \dd x + \int_0^{t} n_{\partial}^{(N)}(s) \dd s = 1.
		$$
	\end{rem}
	\subsection{The approximated model}\label{subsect:approximated_models}

	We now provide the model approximation that allows us to solve this inverse problem. To do so, let us first introduce the following notation
	\begin{equation}\label{eq:definition_moments_laplace}
		\forall i\in\{1,2,3\}: \hspace{2mm} m_i := \int_0^{\delta} v^i g(v) \dd v.
	\end{equation}
	It has been proved in~\cite[Proposition~$3.2$]{olaye_transport_2026} that the limit of~\eqref{eq:scaled_model} when $N\rightarrow +\infty$ is a transport equation with an absorbing state at $x = 0$, and that the rate of convergence towards this approximation is of order $\frac{1}{N}$. Then, to obtain a better estimation, we need to take into account in our new approximated model the telomere length variation of order $\frac{1}{N}$ at each division. By doing a second-order Taylor expansion in the first line of~\eqref{eq:scaled_model}, one can easily obtain that for all $t\geq0$ and~$x\geq0$,
	\begin{equation}\label{eq:model_approximation_intermediate_first}
		\begin{aligned}
			\partial_t n^{(N)}(t,x) &\underset{N\rightarrow+\infty}{\approx} bN\int_0^{\delta}\left[\frac{v}{N}\partial_x n^{(N)}(t,x) + \frac{v^2}{2N^2}\partial_x^2 n^{(1)}(t,x)\right]g(v) \,\mathrm{d}v \\ 
			&=  bm_1\partial_x n^{(N)}(t,x) + \frac{bm_2}{2N}\partial_x^2 n^{(N)}(t,x).
		\end{aligned}
	\end{equation}
	In addition, by doing now a Taylor expansion in the second line of~\eqref{eq:scaled_model} and using the following equalities $\int_0^{\delta} (1- G(v)) \dd v = m_1$ and $\int_{0}^{\delta} v(1-G(v)) \dd v = \frac{m_2}{2}$ (obtained by integration by part), one can easily get that for all $t\geq0$,
	\begin{equation}\label{eq:model_approximation_intermediate_second}
		\begin{aligned}
			n_{\partial}^{(N)}(t) &\underset{N\rightarrow+\infty}{\approx} b\int_0^{\delta} \left[n^{(N)}(t,0) + \frac{v}{N}\partial_x n^{(N)}(t,0)\right](1-G(v))\,\mathrm{d}v \\ 
			&= bm_1 n^{(N)}\left(t,0\right) + \frac{bm_2}{2N}\partial_x n^{(N)}\left(t,0\right).
		\end{aligned}
	\end{equation}
	Then, in view of~\eqref{eq:model_approximation_intermediate_first} and~\eqref{eq:model_approximation_intermediate_second}, we have the conjecture that when $N$ is large, the system~\eqref{eq:scaled_model} can be approximated by the following 
	\begin{equation}\label{eq:approximation_transport_diffusion}
		\begin{cases}
			\partial_t u^{(N)}(t,x) = bm_1\partial_{x} u^{(N)}(t,x) + \frac{bm_2}{2N}\partial_{x}^2 u^{(N)}(t,x), & \forall t\geq0,\,x\geq0,\\
			\partial_tu^{(N)}(t,0) = bm_1\partial_{x} u^{(N)}(t,0), & \forall t\geq0,\\
			u^{(N)}_{\partial}(t) = bm_1u^{(N)}(t,0) + \frac{bm_2}{2N}\partial_xu^{(N)}(t,0), & \forall t\geq0,\\ 
			u^{(N)}(0,x) = n_0(x), & \forall x\geq0.
		\end{cases}
	\end{equation}
	We refer to the proof of~\cite[Proposition~$2.2$]{doumic_asymptotic_2026}  to see why~\eqref{eq:approximation_transport_diffusion} is well-posed and why we have~\hbox{$u^{(N)}\in\mathcal{C}^0\left(\mathbb{R}_+,H^2\left(\mathbb{R}_+\right)\right)\cap \mathcal{C}^1\left(\mathbb{R}_+,L^2\left(\mathbb{R}_+\right)\right)$} when $n_0 \in H^2\left(\mathbb{R}_+\right)$. The boundary condition we take in~our approximated model, see the second line of~\eqref{eq:approximation_transport_diffusion}, is the same as in~\cite{doumic_asymptotic_2026}. In view of the first line of~\eqref{eq:approximation_transport_diffusion}, this boundary condition is equivalent to having~$\frac{bm_2}{2N}\partial_{x}^2 u^{(N)}(t,0) = 0$ for all~\hbox{$t\geq0$}, so to having a second derivative of $0$ at the boundary. We use this condition because a Dirichlet or a Neumann boundary condition do not allow us to control the third derivative of~$u^{(N)}$~(see~Lemma~\ref{eq:approx_eigen_with_mu_sigma}), which is crucial here to control the remainder of the Taylor expansion~(see~\eqref{eq:model_approximation_intermediate_first}). Qualitatively, this is because having a second derivative of $0$ allows us to control the too abrupt variations of $u^{(N)}$ near $0$ that can result in a bad approximation.

	\subsection{Notations and main theoretical results}\label{subsect:main_result_chaplaplace}
	
	In this work, we have two main theoretical results. The first one corresponds to the qualitative bounds we have obtained on the model approximation. The second one is the link between the Laplace transforms of $n_0$ and $u_{\partial}^{(N)}$ we have established, for all $N>0$, that allows us to construct estimators of $n_0$. To present these results, we need to introduce some notations. 
	
	\paragraph{Notations.}  First, we denote for all $f : \mathbb{R}_+ \rightarrow \mathbb{R}$ and $p\in \mathbb{C}$ such that the integral below converges
	\begin{equation}\label{eq:definition_laplace_transform}
		\mathcal{L}(f)(p) := \int_0^{\infty} e^{-px}f(x) \dd x,
	\end{equation}
	which corresponds to the Laplace transform of $f$. Then, we denote for all $f : \mathbb{R}_+ \rightarrow \mathbb{R}$
	\begin{equation}\label{eq:abscissa_convergence}
	\mathcal{R}(f) := \inf\left(\left\{r\in\mathbb{R}\,|\,\int_0^{\infty} e^{-rx}\left|f(x)\right| \dd x < +\infty \right\}\right),
\end{equation}
	which corresponds to the abscissa of convergence of $\mathcal{L}(f)$. Notably, for all $p\in \mathbb{C}$ such that $\text{Re}(p) > \mathcal{R}(f)$, the integral presented in~\eqref{eq:definition_laplace_transform} converges. Thereafter, for all $N > 0$ and $\lambda >0$, we define the three following constants
	\begin{equation}\label{eq:approximation_eigenvalues_chaplaplace}
		\lambda_N := \lambda\left(1 - \frac{\lambda m_2}{2m_1N}\right), \hspace{2.2mm}C_N := \frac{b}{2}\left[1 - \mathcal{L}(g)\left(\frac{2m_1\lambda_N}{m_2\lambda } \right)\right], \hspace{2.2mm}\beta_N := \begin{cases}
			NC_N, & \text{ if } N\leq \frac{bm_1\lambda_N}{C_N} ,\\ 
			bm_1\lambda_N , & \text{ if } N > \frac{bm_1\lambda_N}{C_N}. 
		\end{cases}
	\end{equation}
	Finally, we consider  for all $N>0$ the function $q_N : \mathbb{C} \rightarrow \mathbb{C}$ defined for all $p\in\mathbb{C}$ as 
	\begin{equation}\label{eq:definition_qN}
		q_N(p) := bm_1p + \frac{bm_2}{2N}p^2,
	\end{equation}
	and the set
	\begin{equation}\label{eq:definition_P}
		\mathcal{P}_N := \left\{p\in \mathbb{C}\,|\,Re(p) > \mathcal{R}(n_0),\,Re(q_N(p)) > \max\left(\mathcal{R}\left(u_{\partial}^{(N)} \right), - (bm_1)^2\frac{2N}{bm_2}\right)\right\}.
	\end{equation}
\paragraph{Main theoretical results.} The first main theoretical result of the paper is the following. It allows us to justify that we can approximate~\eqref{eq:PDE_model_telomeres_one_telo_chaplaplace} by~\eqref{eq:approximation_transport_diffusion}, and is proved in Section~\ref{sect:proof_model_approximation}.

\begin{te}[Bounds on approximation errors]\label{te:model_approximation}
	Assume that $n_0\in H^3\left(\mathbb{R}_+\right)$, and that there exist $C_{\lambda} > 0$,\,$D_{\lambda} > 0$, $\lambda >0$ and $\lambda' >0$, such that for all $x\geq0$:
	\begin{equation}\label{eq:assumptions_main_result_chaplaplace}
		\left|n'''_0(x)\right| \leq C_{\lambda}\exp\left(-\lambda x\right), \hspace{1.5mm}\text{ and }\hspace{1.5mm} \left|n''_0(x)\right| \leq D_{\lambda}\exp\left(-\lambda x\right)\left(1 - \exp\left(-\lambda' x\right)\right).
	\end{equation}
	Then, there exist $c_1>0$, $c_2 > 0$, $c_3 > 0$ and $c_4>0$ such that the following statements hold.
	\begin{enumerate}[(a)]
		\item For all $N > \frac{\lambda m_2}{2m_1}$, $t\geq 0$, $x\geq 0$, we have\vspace{-0.5mm} \label{te:model_approximation_first}
		\begin{equation}\label{eq:main_result_first_statement}
			\begin{aligned}
				\left|n^{(N)}(t,x)-u^{(N)}(t,x)\right| &\leq  \frac{c_1t}{N^2}\exp\left[-bm_1\lambda_N t  -\lambda x\right]\\
				&+  \frac{1}{C_N}\left(\frac{c_2}{N^3} + \frac{c_3}{N^2}\right)\exp\left[-\beta_Nt -\frac{2Nm_1}{m_2} \frac{\lambda_N}{\lambda}x\right].
			\end{aligned}
		\end{equation}
		\item For all $N > \frac{m_2}{2m_1}\left(2\lambda + \lambda'\right)$, $t\geq0$, we have \label{te:model_approximation_second}\vspace{-0.5mm}
		\begin{equation}\label{eq:main_result_second_statement}
			\left|n_{\partial}^{(N)}(t) - u_{\partial}^{(N)}(t)\right| \leq \frac{bm_1c_1t + c_4}{N^2}\exp\left[-bm_1\lambda_N t\right] +  \frac{bm_1}{C_N}\left(\frac{c_2}{N^3} + \frac{c_3}{N^2}\right)\exp\left[-\beta_Nt\right].
		\end{equation}
	\end{enumerate}
\end{te}
\noindent This theorem is very similar to~\cite[Proposition~$3.2$]{olaye_transport_2026}. The assumptions are also on the derivatives of $n_0$, but this time, the orders of the derivatives are greater. We have an error that first increases with time, and then decreases exponentially fast. This is related to the fact that we take into account the accumulation of the errors, and the fact that they dissipate. We also have that the error decreases exponentially fast with the telomere length. This is related to the exponential dissipation of the distribution of telomere length at each time.

The second statement we present corresponds to the link between the Laplace transform of~$n_0$ and $u_{\partial}^{(N)}$ we have obtained, for all $N > 0$. This statement is very important in our work, because this is from this link that we construct our estimators. Its proof is however not very long, and does not require any specific argument that we would like to emphasise in this work. Hence, we postpone it to Appendix~\ref{appendix:proof_link_laplace}.
\begin{prop}[Link between the Laplace transforms]\label{prop:link_laplace_transforms}
	Assume that the assumptions of~Theorem~\ref{te:model_approximation} hold and that $n_0(0) = 0$. Then, the following equality holds for all $N > 0$ and~$p\in \mathcal{P}_N$
	\begin{equation}\label{eq:link_laplace_transforms}
		\mathcal{L}\left(n_0\right)(p) = \left(1 + p \frac{bm_1}{(bm_1)^2\frac{2N}{bm_2} + q_N(p)}\right)\mathcal{L}\left(u_{\partial}^{(N)}\right)\left(q_N(p)\right),
	\end{equation}
	where $q_N$ is defined in~\eqref{eq:definition_qN}.
\end{prop}
\noindent What is interesting to observe is that we have obtained a link between the Laplace transforms in the sets $\left(\mathcal{P}_N\right)_{N>0}$, but not in the whole set $\left\{p\in\mathbb{C}\,|\,\text{Re}(p) > \mathcal{R}\left(n_0\right)\right\}$ on which $\mathcal{L}\left(n_0\right)$ is defined. The latter plays an important role in the method for inverting Laplace transforms we choose to estimate $n_0$ in Section~\ref{subsubsect:estimation_gaver_stehfest_noise_free}, since methods using information of the Laplace transform on the complex plane will be more difficult to use. 



	\section{Proof of the model approximation}\label{sect:proof_model_approximation}
	This section is devoted to the proof of Theorem~\ref{te:model_approximation}. Since the proof of this theorem is inspired by the one of~\cite[Proposition $3.2$]{olaye_transport_2026}, we first present in Section~\ref{subsect:recall_old_statements} the auxiliary results obtained in this previous paper, slightly readapted to work in our setting. Then, in~Section~\ref{subsect:auxiliary_statements_derivatives}, we give  the new auxiliary statements we need, that are related to the second and third derivatives in space of the approximated model. Thereafter, we prove these auxiliary statements in Section~\ref{subsect:proof_auxiliary_statements}. Finally, in Section~\ref{subsect:proof_full_approximation}, we prove Theorem~\ref{te:model_approximation}.

	

	
	\subsection{Adaptation of old auxiliary statements}\label{subsect:recall_old_statements}
	
	The statements we present here are mainly useful for proving Theorem~\ref{te:model_approximation}-\ref{te:model_approximation_first}. The reason is that the proof of Theorem~\ref{te:model_approximation}-\ref{te:model_approximation_second} is quite short once Theorem~\ref{te:model_approximation}-\ref{te:model_approximation_first} has been obtained. In the same way as in the proof of~\cite[Proposition $3.2$-$(a)$]{olaye_transport_2026}, the first argument to obtain \hbox{Theorem~\ref{te:model_approximation}-\ref{te:model_approximation_first}} is to prove that there exists $F\in L^1\left(\mathbb{R}_+\right)$ such that $\overline{u}^{(N)} := n^{(N)} - u^{(N)}$ verifies an equation of the form 
	\begin{equation}\label{eq:PDE_to_develop_general_chaplaplace}
		\begin{cases}
			\partial_t\overline{u}^{(N)}(t,x) = bN\int_{v\in\mathbb{R}_+^d}\left[\overline{u}^{(N)}\left(t,x + \frac{v}{N}\right)- \overline{u}^{(N)}(t,x)\right] g\left(\mathrm{d}v\right) +F(t,x),& \forall t\geq0,\,x\geq0, \\ 
			\overline{u}^{(N)}(0,.) \equiv 0.
		\end{cases}
	\end{equation}
	Then, the second argument is to bound the solution of~this equation. Statements allowing us to bound equations of the form given in~\eqref{eq:PDE_to_develop_general_chaplaplace} are given in~\cite[Eq. $3.19$]{olaye_transport_2026},~\cite[Corollary~A.$2$]{olaye_transport_2026} and~\cite[Proposition~A.$4$]{olaye_transport_2026}. Specifically, these statements have been combined in~\hbox{\cite[Section~$3.3$]{olaye_transport_2026}} to prove \hbox{\cite[Lemma~$3.5$]{olaye_transport_2026}}, which is the key lemma of~\cite{olaye_transport_2026}. These statements have been written in a framework in which the approximated models are transport equations, and cannot be directly reused in this work. Hence, we present here slight adaptations of these statements, in order to apply them to our framework. 
	
	The first statement we present is the following. It provides a representation for the solutions to~\eqref{eq:PDE_to_develop_general_chaplaplace}, by writing them as the sum of the evolutions of all the source terms. It corresponds to a slightly modified version of~\cite[Eq. $3.19$]{olaye_transport_2026}, for a more general function $H_N$. Proving it requires to do the same steps as those done to obtain~\cite[Eq. $3.19$]{olaye_transport_2026}. Hence, we do not give its proof.
	\begin{lemm}[Alternative representation of a solution  to~\eqref{eq:approximation_transport_diffusion}]\label{lemm:alternative_representation}
		Let $N > 0$. Assume that there exists $F \in L^1\left(\mathbb{R}_+^2\right)$ such that~\eqref{eq:PDE_to_develop_general_chaplaplace} holds, and $C,\,\alpha,\,\beta > 0$ such that for all~\hbox{$(t,x)\in\mathbb{R}_+^2$}
		\begin{equation}\label{eq:bound_source_term}
			\left|F(t,x)\right| \leq C\exp\left(-\alpha t - \beta x\right).
		\end{equation}
		Then, denoting $\Phi_N :  L^1\left(\mathbb{R}_+\right) \mapsto C\left(\mathbb{R}_+, L^1\left(\mathbb{R}_+\right)\right)$, the operator such that for all $f_0\in L^1\left(\mathbb{R}\right)$, $\Phi_N(f_0)$ is a solution to~\eqref{eq:scaled_model_first_third_lines} with $n_0 = f_0$, we have for all~$(t,x)\in\mathbb{R}_+^2$ 
		$$
		\overline{u}^{(N)}(t,x) = \int_0^t \Phi_N\left(F(s,.)\right)\left(t-s,x\right) \dd s.
		$$
	\end{lemm}
	\noindent The second statement we provide corresponds to a bound we have on the operator $\Phi_N$. It is a consequence of the two statements \cite[Corollary A.$2$]{olaye_transport_2026} and \cite[Proposition A.$4$]{olaye_transport_2026}.
	\begin{lemm}[Maximum principle]\label{lemm:maximum_principle}
		Let us consider $f_0\in L^1\left(\mathbb{R}_+\right)$. Assume that there exist four constants~\hbox{$C_1,\,C_2,\,C_3,\,C_4\geq0$} and $\omega_1,\,\omega_2 \geq 0$ such that for all $y\geq 0$
		\begin{equation}\label{eq:assumption_maximum_principle}
		\left|f_0(y)\right| \leq C_1\exp\left(-\omega_1 y\right) + C_2\exp\left(-\omega_2 y\right). 
		\end{equation}
		Then, it holds for all~$(t,x)\in\mathbb{R}_+^2$ 
		\begin{equation}\label{eq:maximum_principle}
			\begin{aligned}
				\left|\Phi_N(f_0)(t,x)\right| &\leq C_1\exp\left[-\left(bm_1\omega_1 - \frac{bm_2}{2N}\left(\omega_1\right)^2\right)t - \omega_1 x\right] \\
				&+ C_2\exp\left[-bN\left(1 - \mathcal{L}(g)\left(\frac{\omega_2}{N}\right)\right)t - \omega_2 x\right].
			\end{aligned}
		\end{equation}
	\end{lemm}
	\begin{proof}
		We begin by obtaining an intermediate inequality. First, we apply~\cite[Corollary A.$2$]{olaye_transport_2026} for $v_0 = f_0$ and $w_0 = C_1\exp\left(-\omega_1 \text{Id}\right) + C_2\exp\left(-\omega_2 \text{Id}\right)$. Then, we use~\cite[Proposition A.$4$]{olaye_transport_2026} to compute the values of the function~$w$ stated in~\cite[Corollary A.$2$]{olaye_transport_2026}. We obtain that for all~$(t,x)\in\mathbb{R}_+^2$
		\begin{equation}\label{eq:proof_maximum_principle_intermediate_first}
			\begin{aligned}
				\left|\Phi_N(f_0)(t,x)\right| &\leq C_1\exp\left[-bN\left(1 - \mathcal{L}(g)\left(\frac{\omega_1}{N}\right)\right)t - \omega_1 x\right] \\
				&+  C_2\exp\left[-bN\left(1 - \mathcal{L}(g)\left(\frac{\omega_2}{N}\right)\right)t - \omega_2 x\right].
			\end{aligned}
		\end{equation}
		Then, as by the inequality $1 - e^{-x} \geq x - \frac{x^2}{2}$ for all $x\geq 0$ and the definition of $m_1$ and $m_2$ given in~\eqref{eq:definition_moments_laplace} it holds
		\begin{equation}\label{eq:proof_maximum_principle_intermediate_second}
			1 - \mathcal{L}(g)\left(\frac{\omega_1}{N}\right) = \int_0^{+\infty} \left[1 - e^{-\frac{\omega_1}{N}u}\right]g(u) \dd u \geq  \frac{m_1\omega_1}{N} -\frac{m_2\left(\omega_1\right)^2}{2N^2},
		\end{equation}
		we can conclude the proof by plugging~\eqref{eq:proof_maximum_principle_intermediate_second} in~\eqref{eq:proof_maximum_principle_intermediate_first}.
	\end{proof}
	
	\subsection{Auxiliary statements related to the derivatives of \texorpdfstring{$u^{(N)}$}{u\^\{(N)\}}}\label{subsect:auxiliary_statements_derivatives}

	In the same way as in~\cite{olaye_transport_2026}, we need exponential bounds on the derivatives of our approximated model to apply the statements presented in the previous section. These exponential estimates allow us for instance to verify~\eqref{eq:bound_source_term} and~\eqref{eq:assumption_maximum_principle} when we use Lemmas~\ref{lemm:alternative_representation} and~\ref{lemm:maximum_principle}, and are also used in the proof of Theorem~\ref{te:model_approximation}-\ref{te:model_approximation_second}, see Section~\ref{subsubsect:proof_approximation_cemetery}. In~\cite{olaye_transport_2026}, the bounds obtained are on the first and second derivatives of the approximated model. Here, we do an approximation on a higher order, so we need bounds on the second and third derivatives of $u^{(N)}$.
	
	The method used in~\cite{olaye_transport_2026} to obtain the bounds is mainly based on the method of characteristics. We cannot proceed in the same way here because~\eqref{eq:approximation_transport_diffusion} does not correspond to a transport equation. Therefore, we use a different approach based on the explicit representation of the second spatial derivative of $u^{(N)}$ through the heat kernel. This representation allows us to simplify the use of a maximum principle to bound $\partial_{x}^2 u^{(N)}$ by the eigenfunctions of the operator $bm_1\partial_x + \frac{bm_2}{2N}\partial_x^2$ on $H^2\left(\mathbb{R}_+\right)$ with Dirichlet boundary condition, which is the operator associated to the equation verified by $\partial_{x}^2 u^{(N)}$. It also allows us to use classical heat kernel estimates. We first present statements that provide the expression of $\partial_{x}^2 u^{(N)}$ for all~$N>0$. Then, we present statements providing exponential bounds on the derivatives of~$u^{(N)}$.
	
	\paragraph{Expression of the second derivative.}  We denote for all $N>0$ the function $v^{(N)} := \partial_{x}^2u^{(N)}$. To find an explicit representation of this function, we need to find the partial differential equation verified by this function. To obtain it, first notice that the first and second lines of~\eqref{eq:approximation_transport_diffusion} imply that $\frac{bm_2}{2N}\partial_{x}^2u^{(N)}(t,0) = 0$. Then, differentiating two times in the first and last lines of~\eqref{eq:approximation_transport_diffusion} yields that~$v^{(N)}$ is a solution of the following equation
	\begin{equation}\label{eq:PDE_second_derivative_advection_diffusion}
		\begin{cases}
			\partial_t v^{(N)}(t,x) = bm_1\partial_{x} v^{(N)}(t,x) + \frac{bm_2}{2N}\partial_{x}^2 v^{(N)}(t,x), & \forall t\geq0,\,x\geq0,\\
			v^{(N)}(t,0) = 0, & \forall t\geq0,\\
			v^{(N)}(0,x) = n''_0(x), & \forall x\geq0.
		\end{cases}
	\end{equation}
	This equation has a unique solution in $C^0\left(\mathbb{R}_+,L^2\left(\mathbb{R}_+\right)\right)$ by the following proposition, proved in Appendix~\ref{appendix:proof_uniqueness_dirichlet}. We also refer to~\cite{cornilleau_controllability_2012,doumic_asymptotic_2026,gueye_singular_2016} for the proofs of the well-posedness of very similar~systems. 
	\begin{prop}[Existence and uniqueness of~\eqref{eq:PDE_second_derivative_advection_diffusion}]\label{prop:uniqueness_dirichlet}
	Assume that $n''_0\in L^2\left(\mathbb{R}_+\right)$. Then, there exists a unique mild solution to~\eqref{eq:PDE_second_derivative_advection_diffusion} in the space $C^0\left(\mathbb{R}_+,L^2\left(\mathbb{R}_+\right)\right)$. In addition, if \hbox{$n''_0\in H_0^1\left(\mathbb{R}_+\right)\cap H^2\left(\mathbb{R}_+\right)$}, then there exists a unique strict solution to~\eqref{eq:PDE_second_derivative_advection_diffusion} in the space 
		$$
		C^0\left(\mathbb{R}_+,H_0^1\left(\mathbb{R}_+\right)\cap H^2\left(\mathbb{R}_+\right)\right)\cap C^1\left(\mathbb{R}_+,L^2\left(\mathbb{R}_+\right)\right).
		$$
	\end{prop}
	\noindent Therefore, we only have to find an explicit solution to~\eqref{eq:PDE_second_derivative_advection_diffusion}, and we will have an explicit representation for $v^{(N)}$. The following proposition provides this explicit solution, and thus the expression of $v^{(N)}$. It is proved in Section~\ref{subsubsect:proof_second_derivative_explicit}.

	
	\begin{prop}[Explicit solution to~\eqref{eq:PDE_second_derivative_advection_diffusion}]\label{prop:second_derivative_explicit}
		Assume that $n''_0\in L^2\left(\mathbb{R}_+\right)$. We denote for all $N>0$, $(t,x)\in\mathbb{R}_+^*\times\mathbb{R}_+$,
		\begin{equation}\label{eq:definition_fundamental_advection_diffusion}
			\Psi_{N}(t,x) := \left(\frac{N}{bm_2}\frac{1}{2\pi t}\right)^{\frac{1}{2}}\exp\left[-\frac{N}{bm_2}\frac{\left(x+bm_1 t\right)^2}{2t}\right]. 
		\end{equation}
		Then, for all $N>0$, $(t,x)\in\mathbb{R}_+^2$, we have
		$$
		\begin{aligned}
			v^{(N)}(t,x) = \begin{cases}
				n_0''(x), & \text{ if }t = 0,\\
				\left(n_0''*\Psi_{N}(t,.)\right)(x) - \exp\left[-\frac{2Nm_1}{m_2}x\right]\left(n_0''*\Psi_{N}(t,.)\right)(-x), & \text{ otherwise,}
			\end{cases} 
		\end{aligned}
		$$
		where for all $(f,g) \in L^1\left(\mathbb{R}_+\right)\times L^1\left(\mathbb{R}\right)$, $z\in\mathbb{R}$, we write 
		\begin{equation}\label{eq:definition_convolution}
		\left(f*g\right)(z) := \int_{0}^{+\infty} f(y)g(z-y) \dd y = \int_{-\infty}^{z} f(z-y')g(y') \dd y'.
		\end{equation}
	\end{prop}
	\begin{rem}
	As for all $(f,g) \in L^1\left(\mathbb{R}_+\right)\times L^1\left(\mathbb{R}\right)$ we have
	$$
	\int_{z\in\mathbb{R}}\left|\left(f*g\right)(z)\right| \dd z \leq  \int_{y\in\mathbb{R}_+}\int_{z\in\mathbb{R}} \left|f(y)\right|\left|g(z-y)\right|  \dd z\dd y  = \left|\left|f\right|\right|_{L^1\left(\mathbb{R}_+\right)}\left|\left|g\right|\right|_{L^1\left(\mathbb{R}\right)} < +\infty,
	$$
	the convolution product $\left(f*g\right)(z)$ is finite for almost all $z\in\mathbb{R}$, and there is no problem of divergence of the integral in~\eqref{eq:definition_convolution}.
	\end{rem}
	
	\paragraph{Bounds on the derivatives.} Now that we have the explicit representation of $v^{(N)}$, we can obtain the bounds on the second and third derivatives of $u^{(N)}$. The first statement we present provides bounds on the second derivative. Its proof, given in Section~\ref{subsubsect:proof_bound_second_derivative}, is based on the fact that the operator $bm_1 \partial_x + \frac{bm_2}{2N}\partial_{x}^2$ on $H^2\left(\mathbb{R}_+\right)$ with Dirichlet boundary condition has explicit eigenfunctions. 
	\begin{lemm}[Bound on the second derivative]\label{lemm:bound_second_derivative}
		Assume that~\eqref{eq:assumptions_main_result_chaplaplace} holds. Then, we have for all $N > \frac{m_2}{2m_1}\left(2\lambda + \lambda'\right)$ and $(t,x)\in\mathbb{R}_+^2$
		\begin{equation}\label{eq:bound_second_derivative}
			\left|\partial_{x}^2 u^{(N)}(t,x)\right| \leq D_{\lambda}\exp\left[-bm_1\lambda_Nt\right]\left[\exp\left(- \lambda x\right)- \exp\left(-\frac{2Nm_1}{m_2}\frac{\lambda_N}{\lambda}x\right)\right].
		\end{equation}
	\end{lemm}
	\noindent The second statement we present provides bounds for the third derivative. Its proof is presented in Section~\ref{subsubsect:proof_bound_third_derivative}, and relies on classical heat kernel estimates.
	\begin{lemm}[Bound on the third derivative]\label{lemm:bound_third_derivative}
		Assume that~\eqref{eq:assumptions_main_result_chaplaplace} holds. Then,  we have for all~$N > 0$ and $(t,x)\in\mathbb{R}_+^2$
		\begin{equation}\label{eq:bound_third_derivative}
		\left|\partial_{x}^3 u^{(N)}(t,x)\right| \leq C_{\lambda}\exp\left[-bm_1\lambda_N t -\lambda x\right] +  \left(C_{\lambda}+D_{\lambda}\frac{2Nm_1}{m_2}\right)\exp\left[- bm_1\lambda_N t -\frac{2Nm_1}{m_2}\frac{\lambda_N}{\lambda} x\right].
		\end{equation}
	\end{lemm}

	\subsection{Proof of the auxiliary statements}\label{subsect:proof_auxiliary_statements}
	We provide here the proofs of the statements presented in Section~\ref{subsect:auxiliary_statements_derivatives}, except the one of Proposition~\ref{prop:uniqueness_dirichlet} which is given in Appendix~\ref{appendix:proof_uniqueness_dirichlet}. In all this section, to simplify notations, we~write 
	\begin{equation}\label{eq:mean_sigma_dft}
		\mu := bm_1,\hspace{5.1mm} \text{and}\hspace{5.1mm} \sigma_N^2 := \frac{bm_2}{N}.
	\end{equation}
	In particular, we have by~\eqref{eq:approximation_eigenvalues_chaplaplace} that
	\begin{equation}\label{eq:approx_eigen_with_mu_sigma}
		\lambda_N = \lambda\left(1 - \frac{\lambda\sigma_N^2}{2\mu}\right).
	\end{equation}
	

	\subsubsection{Proof of Proposition~\ref{prop:second_derivative_explicit}}\label{subsubsect:proof_second_derivative_explicit}
	
	Let us fix $N > 0$. We consider for all $(t,x)\in\mathbb{R}_+^2$ 	
	\begin{equation}\label{eq:_proof_prop_second_derivative_explicit_intermediate_first}
		\tilde{v}^{(N)}(t,x) =  \begin{cases}
			n_0''(x), & \text{ if }t = 0,\\
			\left(n_0''*\Psi_{N}(t,.)\right)(x) - \exp\left[-\frac{2\mu}{\sigma_N^2}x\right]\left(n_0''*\Psi_{N}(t,.)\right)(-x), & \text{ otherwise.}
		\end{cases}
	\end{equation}	
	In view of~Proposition~\ref{prop:uniqueness_dirichlet}, our aim is to prove that $\tilde{v}^{(N)}$ is a mild solution of~\eqref{eq:PDE_second_derivative_advection_diffusion} in the space~$C^0\left(\mathbb{R}_+,L^2\left(\mathbb{R}_+\right)\right)$. To do so, first notice that as for any $(f,h)\in L^1\left(\mathbb{R}_+\right)\times W^{1,1}\left(\mathbb{R}\right)$ it holds $(f*h)' = f*h'$, we have for all $(t,x)\in\mathbb{R}_+^*\times\mathbb{R}_+$ that
	$$
	\begin{aligned}
		\partial_x \tilde{v}^{(N)}(t,x) &=   \left(n_0''*\partial_x\Psi_{N}(t,.)\right)(x) + \exp\left[-\frac{2\mu}{\sigma_N^2}x\right]\bigg[ \left(n_0''*\partial_x\Psi_{N}(t,.)\right)(-x)  \\ 
		&\hspace{90mm}+\frac{2\mu}{\sigma_N^2}\left(n_0''*\Psi_{N}(t,.)\right)(-x)\bigg],\\
		\partial_{x}^2 \tilde{v}^{(N)}(t,x) &= \left(n_0''*\partial_{x}^2\Psi_{N}(t,.)\right)(x) - \exp\left[-\frac{2\mu}{\sigma_N^2}x\right]\bigg[ \left(n_0''*\partial_{x}^2\Psi_{N}(t,.)\right)(-x)\\ 
		&\hspace{34.85mm} +\frac{4\mu}{\sigma_N^2}\left(n_0''*\partial_x\Psi_{N}(t,.)\right)(-x)  +\left(\frac{2\mu}{\sigma_N^2}\right)^2\left(n_0''*\Psi_{N}(t,.)\right)(-x)\bigg] .
	\end{aligned}
	$$
	Then, by doing a linear combination of the two above terms, we obtain for all $(t,x)\in\mathbb{R}_+^*\times\mathbb{R}_+$ 
	\begin{equation}\label{eq:_proof_prop_second_derivative_explicit_intermediate_second}
		\begin{aligned}
			\mu \partial_x \tilde{v}^{(N)}(t,x) + \frac{\sigma_N^2}{2} \partial_{x}^2 \tilde{v}^{(N)}(t,x) &= \left(n_0''*\left[\mu \partial_x\Psi_N(t,.) + \frac{\sigma_N^2}{2}\partial_{x}^2\Psi_N(t,.)\right]\right)(x) \\ 
			&- \exp\left[-\frac{2\mu}{\sigma_N^2}x\right]\left(n_0''*\left[\mu \partial_x\Psi_N(t,.) + \frac{\sigma_N^2}{2}\partial_{x}^2\Psi_N(t,.)\right]\right)(-x).
		\end{aligned}
	\end{equation}
	On the other hand, we have by~\eqref{eq:_proof_prop_second_derivative_explicit_intermediate_first} that the following holds, for all $(t,x)\in\mathbb{R}_+^*\times\mathbb{R}_+$, 
	\begin{equation}\label{eq:_proof_prop_second_derivative_explicit_intermediate_third}
		\partial_t \tilde{v}^{(N)}(t,x) = \left(n_0''*\partial_t\Psi_N(t,.)\right)(x) - \exp\left[-\frac{2\mu}{\sigma_N^2}x\right]\left(n_0''*\partial_t\Psi_N(t,.)\right)(-x).
	\end{equation}
	We also have since $\varphi(s,y)\mapsto \frac{1}{\sqrt{4\pi s}}\exp\left(-\frac{y^2}{4s}\right)$ is the fundamental solution to the heat equation, and since $\Psi_N(s,y) = \varphi\left(\frac{\sigma^2_N}{2} s,y+\mu s\right)$ for all $(s,y)\in\mathbb{R}_+^*\times\mathbb{R}$ (see~\eqref{eq:definition_fundamental_advection_diffusion} and~\eqref{eq:mean_sigma_dft}), that~$\Psi_N$ verifies the following, for all \hbox{$(t,x)\in\mathbb{R}_+^*\times\mathbb{R}$,} 
	\begin{equation}\label{eq:_proof_prop_second_derivative_explicit_intermediate_fourth}
		\partial_t \Psi_{N}(t,x) = \mu\partial_x \Psi_{N}(t,x) + \frac{\sigma_N^2}{2}\partial_{x}^2 \Psi_{N}(t,x).
	\end{equation}
	Then, by plugging~\eqref{eq:_proof_prop_second_derivative_explicit_intermediate_fourth} in~\eqref{eq:_proof_prop_second_derivative_explicit_intermediate_third}, we obtain that the left-hand sides of~\eqref{eq:_proof_prop_second_derivative_explicit_intermediate_second} and~\eqref{eq:_proof_prop_second_derivative_explicit_intermediate_third} are equal, so that $\tilde{v}^{(N)}(t,.)$ verifies the first line of~\eqref{eq:PDE_second_derivative_advection_diffusion} when $t>0$. It thus only remains to prove that~$v^{(N)}(t,.)$ is continuous at $t= 0$ with respect to the topology of $L^2\left(\mathbb{R}_+\right)$, and we will have that $v^{(N)}$ is a mild solution to~\eqref{eq:PDE_second_derivative_advection_diffusion} in~$C^0\left(\mathbb{R}_+,L^2\left(\mathbb{R}_+\right)\right)$. To do so, notice that as $\Psi_N(t,x) = \varphi\left(\frac{\sigma^2_N}{2} t,x+\mu t\right)$ for all \hbox{$(t,x)\in\mathbb{R}_+\times\mathbb{R}$}, and as $\varphi$ is a classical mollifier kernel on $\mathbb{R}$, we have  $\left(h*\Psi_N(t,.)\right)(y) \underset{t\rightarrow0^+}{\rightarrow} h(y)$ for all~$h\in \mathcal{C}_c\left(\mathbb{R}\right)$,~$y\in\mathbb{R}$. This yields that for all~$f\in\mathcal{C}_c\left(\mathbb{R}_+\right)$,~$x\geq0$, it holds (we take $h = f1_{\mathbb{R}_+}$)
	$$
	\lim_{t\rightarrow 0^+}\left[\left(f*\Psi_{N}(t,.)\right)(x) - \exp\left[-\frac{2\mu}{\sigma_N^2}x\right]\left(f*\Psi_{N}(t,.)\right)(-x)\right] = f(x).
	$$
	Thanks to the dominated convergence theorem, one can obtain that the above convergence result also holds in the norm $||.||_{L^2\left(\mathbb{R}_+\right)}$. Then, as the set $\mathcal{C}_c\left(\mathbb{R}_+\right)$ is dense in the space $L^2\left(\mathbb{R}_+\right)$, we obtain from this last result and a density argument that $v^{(N)}(t,.)$ is continuous at $t=0$, so that $v^{(N)}$ is a mild solution to~\eqref{eq:PDE_second_derivative_advection_diffusion} in~$C^0\left(\mathbb{R}_+,L^2\left(\mathbb{R}_+\right)\right)$.  \qed 
	

	\subsubsection{Proof of Lemma~\ref{lemm:bound_second_derivative}}\label{subsubsect:proof_bound_second_derivative}
	
	Let us fix $N > \frac{m_2}{2m_1}\left(2\lambda + \lambda'\right)$. We consider the functions~$E^{(N)}$ and $\overline{v}^{(N)}$, defined such that for all $(t,x)\in\mathbb{R}_+^2$
	$$
	\begin{aligned}
		E^{(N)}(t,x) &= D_{\lambda}\exp\left[-\mu\lambda_Nt\right]\left[\exp\left(- \lambda x\right)- \exp\left(-\frac{2\mu}{\sigma_N^2}\frac{\lambda_N}{\lambda}x\right)\right], \\
		\overline{v}^{(N)}(t,x) &=  E^{(N)}(t,x) - v^{(N)}(t,x).
	\end{aligned}
	$$
	Our aim is to prove that~$\overline{v}^{(N)}$ is non-negative. Then,~\eqref{eq:bound_second_derivative} will directly come from this last property.
	
	It is well-known that the solutions to the differential equation~\hbox{$\frac{\sigma_N^2}{2}y'' + \mu y' + \mu\lambda_N y = 0$} are the functions of the form
	\begin{equation}\label{eq:proof_bound_second_derivative_intermediate_first}
		y(x) = Ae^{r_1 x} + Be^{r_2x}, \hspace{4mm} \forall x\in\mathbb{R},
	\end{equation}
	where $(A,B)\in\mathbb{R}^2$, and $(r_1,r_2)\in\mathbb{R}^2$ are the roots of the polynomial \hbox{$P := \frac{\sigma_N^2}{2}\text{Id}^2 + \mu\text{Id} + \mu\lambda_N$}. In addition, one can easily check in view of~\eqref{eq:approx_eigen_with_mu_sigma} that it holds
	$$
	\frac{\sigma_N^2}{2}\left(\text{Id}+\lambda\right)\left(\text{Id}+\frac{2\mu}{\sigma_N^2}\frac{\lambda_N}{\lambda}\right) = \frac{\sigma_N^2}{2}\text{Id}^2 + \left(\mu\frac{\lambda_N}{\lambda} + \frac{\sigma_N^2}{2}\lambda\right)\text{Id} + \mu\lambda_N = \frac{\sigma_N^2}{2}\text{Id}^2 + \mu\text{Id} + \mu\lambda_N = P,
	$$
	so that $r_1 = -\lambda$ and $r_2 = -\frac{2\mu}{\sigma_N^2}\frac{\lambda_N}{\lambda}$. Therefore, by combining these equalities with the fact that~$x\mapsto E^{(N)}(t,x)$ is of the form presented in~\eqref{eq:proof_bound_second_derivative_intermediate_first} for all $t\geq0$, we obtain that $E^{(N)}$~verifies 
	$$
	\mu \partial_{x}E^{(N)} + \frac{\sigma_N^2}{2}\partial_{x}^2E^{(N)}  = -\mu\lambda_N E^{(N)} = \partial_t E^{(N)},
	$$
	so that this is an eigenfunction of the operator $\mu \partial_{x} + \frac{\sigma_N^2}{2}\partial_{x}^2$ with Dirichlet conditions (notice that $E^{(N)}(.,0) \equiv 0$). This yields, in view of~\eqref{eq:PDE_second_derivative_advection_diffusion}, that~$\overline{v}^{(N)}$ is a solution of
	\begin{equation}\label{eq:proof_bound_second_derivative_intermediate_second}
		\begin{cases}
			\partial_t \overline{v}^{(N)}(t,x) = \mu\partial_{x} \overline{v}^{(N)}(t,x) + \frac{\sigma_N^2}{2}\partial_{x}^2 \overline{v}^{(N)}(t,x), & \forall t\geq0,\,x\geq0,\\
			\overline{v}^{(N)}(t,0) = 0, & \forall t\geq0,\\
			\overline{v}^{(N)}(0,x) = D_{\lambda}\left[\exp\left(- \lambda x\right)- \exp\left(-\frac{2\mu}{\sigma_N^2}\frac{\lambda_N}{\lambda}x\right)\right] - n''_0(x), & \forall x\geq0.
		\end{cases}
	\end{equation}
	We thus only have to prove that a solution of~\eqref{eq:proof_bound_second_derivative_intermediate_second} is non-negative and our lemma will be proved. The fact that $\overline{v}(0,.)$ is non-negative is trivial from the right-hand side of~\eqref{eq:assumptions_main_result_chaplaplace} and the following inequality, consequence of the fact that $\lambda + \lambda' < \frac{2\mu}{\sigma_N^2} - \lambda$ (since $N > \frac{m_2}{2m_1}\left(2\lambda + \lambda'\right)$) and~\eqref{eq:approx_eigen_with_mu_sigma}
	$$
	\forall x\geq0: \hspace{2mm}\exp\left(-(\lambda+\lambda')x\right) \leq \exp\left[-\left(\frac{2\mu}{\sigma_N^2} - \lambda\right)x\right] = \exp\left[-\frac{2\mu}{\sigma_N^2}\frac{\lambda_N}{\lambda}x\right].
	$$
	Hence, we focus on proving that $\overline{v}(t,.)$ is non-negative for all $t >0$. Notice that by adapting the proof of Proposition~\ref{prop:second_derivative_explicit} to $\overline{v}^{(N)}$ instead of~$v^{(N)}$, we have for all $(t,x)\in\mathbb{R}_+^*\times\mathbb{R}$ 
	\begin{equation}\label{eq:proof_bound_second_derivative_intermediate_third}
		\begin{aligned}
			&\overline{v}^{(N)}(t,x) = \left(\overline{v}^{(N)}(0,.)*\Psi_{N}\left(t,.\right)\right)(x) - \exp\left[-\frac{2\mu}{\sigma_N^2}x\right]\left(\overline{v}^{(N)}(0,.)*\Psi_{N}\left(t,.\right)\right)(-x) \\
			&= \frac{1}{\sigma_N\sqrt{2\pi t}}\int_0^{+\infty} \left(\exp\left[-\frac{\left(x-y+\mu t\right)^2}{2\sigma_N^2 t}\right] - \exp\left[-\frac{\left(-x-y+\mu t\right)^2 + 4\mu t x}{2\sigma_N^2 t}\right]\right)\overline{v}^{(N)}(0,y)\dd y.
		\end{aligned}
	\end{equation}
Notice also that for all $(t,x,y)\in\mathbb{R}_+^3$,  it holds 
$$
\left(-x-y+\mu t\right)^2 + 4\mu t x = x^2 +y^2 +\left(\mu t\right)^2 + 2xy + 2\mu tx - 2\mu ty = \left(x-y+\mu t\right)^2 + 4xy \geq \left(x-y+\mu t\right)^2.
$$
By combining the above and the fact that $\overline{v}(0,.)$ is non-negative with~\eqref{eq:proof_bound_second_derivative_intermediate_third}, we have that~$\overline{v}^{(N)}(t,.)$ is the integral of a non-negative function for all $t>0$. Hence, $\overline{v}^{(N)}(t,.)$ is non-negative for all $t>0$, and so~\eqref{eq:bound_second_derivative} is true.  \qed

\subsubsection{Proof of Lemma~\ref{lemm:bound_third_derivative}}\label{subsubsect:proof_bound_third_derivative}


The fact that~\eqref{eq:bound_third_derivative} is true when $t = 0$ is trivial from the left-hand side of~\eqref{eq:assumptions_main_result_chaplaplace}. Hence, we focus in this proof in proving~\eqref{eq:bound_third_derivative} when $t>0$. Notice that as $n_0''\in H_0^1\left(\mathbb{R}_+\right)\cap W^{1,1}\left(\mathbb{R}_+\right)$ by~\eqref{eq:assumptions_main_result_chaplaplace}, we have by using~Proposition~\ref{prop:second_derivative_explicit} (we recall that~$v^{(N)} = \partial_x^2 u^{(N)}$) and the fact that~\hbox{$(f*h)' = f'*h$} when $(f,h)\in  W^{1,1}\left(\mathbb{R}_+\right)\times L^1\left(\mathbb{R}\right)$ and $f(0) = 0$ that for all~$(t,x)\in\mathbb{R}_+^*\times\mathbb{R}_+$
\begin{equation}\label{eq:proof_bound_third_derivative_intermediate_first}
	\begin{aligned}
		\partial_{x}^3u^{(N)}(t,x) &= \left(n'''_0*\Psi_{N}\left(t,.\right)\right)(x) \\ 
		&+ \exp\left(-\frac{2\mu}{\sigma_N^2}x\right)\left[\left(n'''_0*\Psi_{N}\left(t,.\right)\right)(-x) + \frac{2\mu}{\sigma_N^2}\left(n''_0*\Psi_{N}\left(t,.\right)\right)(-x)\right].
	\end{aligned}
\end{equation}
Thus, to obtain~\eqref{eq:bound_third_derivative}, we only have to bound from above each of the three above convolutions, and then to combine these bounds. Let us do it.

We begin by bounding the first term of~\eqref{eq:proof_bound_third_derivative_intermediate_first}. First, we develop the convolution by using the first equality in~\eqref{eq:definition_convolution}, and apply the left-hand side of~\eqref{eq:assumptions_main_result_chaplaplace} to bound $n_0'''$. Thereafter, we develop the integrand in view of the equality $-\frac{1}{2}(c-a)^2 -ad= -\frac{1}{2}(c-a-d)^2 + \frac{1}{2}d^2-cd$ for~$c = \frac{x+\mu t}{\sigma_N \sqrt{t}}$, $a = \frac{y}{\sigma_N \sqrt{t}}$ and $d = \lambda \sigma_N\sqrt{t}$. Finally, we apply the following inequality, for all~$z\in\mathbb{R}$, 
$$
\frac{1}{\sigma_N \sqrt{2\pi t}}\int_{0}^{+\infty} \exp\left[-\frac{1}{2}\left(\frac{z - y}{\sigma_N\sqrt{t}}\right)^2\right]\dd y \leq \frac{1}{\sigma_N \sqrt{2\pi t}}\int_{\mathbb{R}} \exp\left[-\frac{1}{2}\left(\frac{z - y}{\sigma_N\sqrt{t}}\right)^2\right]\dd y = 1, 
$$
and use~\eqref{eq:approx_eigen_with_mu_sigma} to make appear $\lambda_N$. We obtain that for all $(t,x)\in\mathbb{R}_+^*\times\mathbb{R}_+$
\begin{equation}\label{eq:proof_bound_third_derivative_intermediate_second}
	\begin{aligned}
		\Big|\!\left(n'''_0*\Psi_{N}\left(t,.\right)\right)(x&)\Big| \leq \frac{C_{\lambda}}{\sigma_N \sqrt{2\pi t}} \int_{0}^{+\infty}\! \exp\left[-\frac{1}{2}\left(\frac{x - y + \mu t}{\sigma_N\sqrt{t}}\right)^2-\lambda y\right]  \dd y \\
		&= \frac{C_{\lambda}}{\sigma_N \sqrt{2\pi t}} \int_{0}^{+\infty}\! \exp\left[-\frac{1}{2}\left(\frac{x - y + \mu t - \lambda\sigma_N^2t}{\sigma_N\sqrt{t}}\right)^2+ \frac{\lambda^2\sigma_N^2t}{2} - \lambda\left(\mu t + x\right)\right]  \dd y \\ 
		&\leq C_{\lambda}\exp\left[\frac{\lambda^2\sigma_N^2t}{2} - \lambda\left(\mu t + x\right)\right] = C_{\lambda}\exp\left[-\mu\lambda_N t -\lambda x\right].
	\end{aligned}
\end{equation}
We now bound from above the two other convolutions in~\eqref{eq:proof_bound_third_derivative_intermediate_first}. To do so, we only have to do the same steps as those done to obtain~\eqref{eq:proof_bound_third_derivative_intermediate_second}, replacing $\mu t + x$ with $\mu t-x$, and using the right-hand side of~\eqref{eq:assumptions_main_result_chaplaplace} for the last convolution instead of the left-hand side. We obtain that for all~$(t,x)\in\mathbb{R}_+^*\times\mathbb{R}_+$
\begin{equation}\label{eq:proof_bound_third_derivative_intermediate_third}
	\begin{aligned}
		\left|\left(n'''_0*\Psi_{N}\left(t,.\right)\right)(-x)\right| &\leq C_{\lambda}\exp\left[\frac{\lambda^2\sigma_N^2t}{2} - \lambda\left(\mu t - x\right)\right]
		= C_{\lambda}\exp\left[-\mu\lambda_N t +\lambda x\right],\\
		\left|\left(n''_0*\Psi_{N}\left(t,.\right)\right)(-x)\right| &\leq D_{\lambda}\exp\left[-\mu\lambda_N t +\lambda x\right].
	\end{aligned}
\end{equation}
It thus only remains to combine the three bounds obtained in~\eqref{eq:proof_bound_third_derivative_intermediate_second} and~\eqref{eq:proof_bound_third_derivative_intermediate_third}, and the lemma will be proved. To do this, we first use the triangular inequality, and then the following equality, that comes from~\eqref{eq:approx_eigen_with_mu_sigma}, to simplify the two last terms, for all $(t,x)\in\mathbb{R}_+^2$,

$$
\begin{aligned}
	\exp\left(-\frac{2\mu}{\sigma_N^2}x\right)\exp\left[-\mu\lambda_N t +\lambda x\right] &= \exp\left[ - \mu\lambda_N t -\frac{2\mu}{\sigma_N^2} \frac{\lambda_N}{\lambda} x\right].
\end{aligned}
$$
We obtain the following inequality, which concludes the proof, for all~$(t,x)\in\mathbb{R}_+^*\times\mathbb{R}$, 
$$
\left|\partial_{x}^3u^{(N)}(t,x)\right| \leq C_{\lambda}\exp\left[-\mu\lambda_N t -\lambda x\right] +  \left(C_{\lambda}+D_{\lambda}\frac{2\mu}{\sigma_N^2}\right)\exp\left[ - \mu\lambda_N t -\frac{2\mu}{\sigma_N^2}\frac{\lambda_N}{\lambda} x\right].
$$  
\qed

\subsection{Proof of Theorem~\ref{te:model_approximation}}\label{subsect:proof_full_approximation}

In this section, we present the proof of Theorem~\ref{te:model_approximation}. First, we present the proof of the first statement of the theorem in Section~\ref{subsubsect:proof_approximation_density_lengths}. Then, we prove the second statement in Section~\ref{subsubsect:proof_approximation_cemetery}.
\subsubsection{Proof of Theorem~\ref{te:model_approximation}-\ref{te:model_approximation_first}}\label{subsubsect:proof_approximation_density_lengths}

This proof is inspired by the proof of~\cite[Proposition $3.2$-$(a)$]{olaye_transport_2026}. Let us fix $N > \frac{\lambda m_2}{2m_1}$, and let us write to simplify notations $g(\dd v) = g\left(v\right) \dd v$. Our aim is to control the absolute value of~\hbox{$\overline{u}^{(N)} := n^{(N)} - u^{(N)}$}. We thus need to find the equation verified by $\overline{u}^{(N)}$ to obtain this bound. To do so, we first take the difference between the first lines of~\eqref{eq:scaled_model} and~\eqref{eq:approximation_transport_diffusion}. Then, we use that~\hbox{$n^{(N)} = \overline{u}^{(N)} + u^{(N)}$} to develop the term $n^{(N)}$ in the integral. Finally, we use the two following equalities to put the terms with derivatives in the integral, for all $(t,x)\in\mathbb{R}_+^2$,
$$
\begin{aligned}
	bm_1\partial_{x} u^{(N)}(t,x) &= bN \int_0^{\delta} \frac{v}{N}g(v) \dd v\partial_{x} u^{(N)}(t,x), \\
	\frac{bm_2}{2N}\partial_{x}^2 u^{(N)}(t,x) &= bN \int_0^{\delta} \frac{v^2}{2N^2} g(v) \dd v\partial_{x}^2 u^{(N)}(t,x). 
\end{aligned}
$$
We obtain that $\overline{u}^{(N)}$ verifies for all $(t,x)\in\mathbb{R}_+^2$
\begin{align}
	\partial_t \overline{u}^{(N)}(t,x)  &= \! bN\int_0^{\delta}\Big[n^{(N)}\left(t,x + \frac{v}{N}\right) - n^{(N)}(t,x)\Big]g\left(\dd v\right) - \! bm_1\partial_{x} u^{(N)}(t,x) -\! \frac{bm_2}{2N}\partial_{x}^2 u^{(N)}(t,x) \nonumber \\
	&= bN\int_0^{\delta}\Big[\overline{u}^{(N)}\left(t,x + \frac{v}{N}\right)- \overline{u}^{(N)}(t,x)\Big]g\left(\dd v\right) + bN\int_0^{\delta}\Big[u^{(N)}\left(t,x + \frac{v}{N}\right) \nonumber  \\
	&- u^{(N)}(t,x) - \frac{v}{N}\partial_{x} u^{(N)}(t,x) - \frac{v^2}{2N^2}\partial_{x}^2 u^{(N)}(t,x)\Big]g\left(\dd v\right). \label{eq:proof_approximation_lengths_intermediate_first}
\end{align}
Thanks to this equation, we are now able to obtain an analytical expression for $\overline{u}^{(N)}$. In order to write this expression, we consider for all $(s,y)\in\mathbb{R}_+^2$
\begin{align}
	H_N(s,y) &:= bN\int_0^{\delta}\Big[u^{(N)}\left(s,y + \frac{v}{N}\right)  - u^{(N)}(s,y) - \frac{v}{N}\partial_{x} u^{(N)}(s,y) - \frac{v^2}{2N^2}\partial_{x}^2 u^{(N)}(s,y)\Big]g\left(\dd v\right) \nonumber \\ 
	&= \frac{b}{2N^2}\int_0^{\delta} \left[\int_0^1(1 - w)^2 \partial_{x}^3 u^{(N)}\left(s,y+w\frac{v}{N}\right) \dd w\right] v^3  g\left(\dd v\right), \label{eq:proof_approximation_lengths_intermediate_second}
\end{align}
where the last equality comes from a Taylor expansion with remainder in integral form. By Lemma~\ref{lemm:bound_third_derivative}, we have for all $(s,y)\in\mathbb{R}_+^2$ that 
\begin{equation}\label{eq:proof_approximation_lengths_intermediate_second_bis}
	\begin{aligned}
		\big|H_N(s,y)\big| &\leq \frac{b}{2N^2}\int_0^{\delta}\left[\int_0^1(1 - w)^2 \bigg(C_{\lambda}\exp\left[-bm_1\lambda_N s -\lambda y\right] +  \left[C_{\lambda}+D_{\lambda}\frac{2Nm_1}{m_2}\right]\right.\\
		&\left.\times \exp\left[ - bm_1\lambda_N s -\frac{2Nm_1}{m_2}\frac{\lambda_N}{\lambda} y\right]\bigg)\dd w\right] v^3 g\left(\dd v\right)\\
		&= \frac{bm_3}{6N^2}\exp\left[-bm_1\lambda_Ns\right]\left(C_{\lambda}\exp\left(-\lambda y\right) +   \left[C_{\lambda}+\frac{2Nm_1}{m_2}D_{\lambda}\right]\exp\left[-\frac{2Nm_1}{m_2}\frac{\lambda_N}{\lambda}  y\right]\right).
	\end{aligned}
\end{equation}
Therefore, by plugging~\eqref{eq:proof_approximation_lengths_intermediate_second} in~\eqref{eq:proof_approximation_lengths_intermediate_first} to obtain an equation of the form given in~\eqref{eq:PDE_to_develop_general_chaplaplace}, and then applying Lemma~\ref{lemm:alternative_representation} for the equation we have obtained, we get that for all $(t,x)\in\mathbb{R}_+^2$
\begin{equation}\label{eq:proof_approximation_lengths_intermediate_third}
	\overline{u}^{(N)}(t,x) = \int_0^t \Phi_N\left(H_N(s,.)\right)\left(t-s,x\right) \dd s .
\end{equation}
Our aim is to use the above analytical expression to bound $\overline{u}^{(N)}$. By applying Lemma~\ref{lemm:maximum_principle} to bound~\eqref{eq:proof_approximation_lengths_intermediate_third} in view of~\eqref{eq:proof_approximation_lengths_intermediate_second_bis}, and then using the left-hand side of~\eqref{eq:approximation_eigenvalues_chaplaplace} to make appear~$-bm_1\lambda_N (t-s)$ in the first exponential, we have for all $(t,x)\in\mathbb{R}_+^2$ that
\begin{align}
	\big|\overline{u}^{(N)}(t,x)\big| &\leq \frac{bm_3}{6 N^2}\int_0^t \exp\left[-bm_1\lambda_Ns\right]\bigg(C_{\lambda}\exp\left[-bm_1\lambda_N (t-s)  -\lambda x\right] +\left[C_{\lambda}+\frac{2Nm_1}{m_2}D_{\lambda}\right] \nonumber\\
	&\times\exp\left[-bN\left[1 - \mathcal{L}(g)\left(\frac{2m_1\lambda_N}{m_2\lambda } \right)\right](t-s) -\frac{2Nm_1}{m_2}\frac{\lambda_N}{\lambda}x\right]\bigg)\dd s.\label{eq:proof_approximation_lengths_intermediate_fourth}
\end{align}
Then, as it holds for all $t\geq0$ 
\begin{equation}\label{eq:proof_approximation_lengths_intermediate_fourth_bis}
	\int_0^t \exp\left[-bm_1\lambda_Ns -bm_1\lambda_N (t-s)   -\lambda x\right]\dd s = t\exp\left[-bm_1\lambda_N t  -\lambda x\right],
\end{equation}
we only have to bound the second term in the sum in the right-hand side of~\eqref{eq:proof_approximation_lengths_intermediate_fourth} to prove~\eqref{eq:main_result_first_statement}. To obtain this bound, notice that for all $(t,a,d)\in\mathbb{R}_+^2\times\mathbb{R}_+^*$, it holds when~$\frac{d}{2} \leq a$
\begin{equation}\label{eq:proof_approximation_lengths_intermediate_fifth}
	\int_0^t \exp\left[- as - d(t-s)\right] \dd s \leq \int_0^t \exp\left[- \frac{d}{2}s - d(t-s)\right] \dd s = \frac{e^{-\frac{d}{2} t}-e^{-d t}}{\frac{d}{2}} \leq \frac{2e^{-\frac{d}{2} t}}{d},
\end{equation}
and it holds when $\frac{d}{2} > a$
\begin{equation}\label{eq:proof_approximation_lengths_intermediate_sixth}
	\int_0^t \exp\left[- as - d(t-s)\right] \dd s =  \frac{e^{-at} - e^{-dt}}{d-a} \leq  \frac{e^{-at} - e^{-dt}}{d-\frac{d}{2}} \leq \frac{2e^{-at}}{d}.
\end{equation}
Notice also that we have $\lambda_N \in\mathbb{R}_+$ and $C_N\in\mathbb{R}_+^*$, from~\eqref{eq:approximation_eigenvalues_chaplaplace} and the fact that $N > \frac{\lambda m_2}{2m_1}$. Then, by applying~\eqref{eq:proof_approximation_lengths_intermediate_fifth} and~\eqref{eq:proof_approximation_lengths_intermediate_sixth} for $a= bm_1\lambda_N $ and $d= bN\left[1 - \mathcal{L}(g)\left(\frac{2m_1\lambda_N}{m_2\lambda } \right)\right]$, and thereafter plugging what we get and~\eqref{eq:proof_approximation_lengths_intermediate_fourth_bis} in~\eqref{eq:proof_approximation_lengths_intermediate_fourth}, we obtain that~\eqref{eq:main_result_first_statement} is true for $c_1  = c_2 = \frac{C_{\lambda}bm_3}{6 }$ and~$c_3 = \frac{D_{\lambda}bm_1m_3}{3 m_2}$. \qed


\subsubsection{Proof of Theorem~\ref{te:model_approximation}-\ref{te:model_approximation_second}}\label{subsubsect:proof_approximation_cemetery}

Let $N > \frac{m_2}{2m_1}\left(2\lambda + \lambda'\right)$. In view of the second line of~\eqref{eq:scaled_model}, the third line of~\eqref{eq:approximation_transport_diffusion}, and the triangular inequality, we have for all $t\geq0$ that
\begin{equation}\label{eq:proof_approximation_cemetery_intermediate_first}
	\begin{aligned}
		\left|n_{\partial}^{(N)}(t) - u_{\partial}^{(N)}(t)\right| 
		&\leq b\left|\int_0^{\delta} \left[n^{(N)}\left(t,\frac{v}{N}\right) - u^{(N)}\left(t,\frac{v}{N}\right)\right] \left(1-G(v)\right)\dd v\right| \\
		&+ \left|b\int_0^{\delta}  u^{(N)}\left(t,\frac{v}{N}\right) \left(1-G(v)\right)\dd v - bm_1u^{(N)}(t,0) - \frac{bm_2}{2N}\partial_xu^{(N)}(t,0)\right|.
	\end{aligned}
\end{equation}
Our aim is to bound each of the two terms that compose the right-hand side of~\eqref{eq:proof_approximation_cemetery_intermediate_first}. To do this, our first step is to bound the first of these terms, and obtaining a better expression for the second one. To control the first term, we first apply Theorem~\ref{te:model_approximation}-\ref{te:model_approximation_first}, and then use the fact that $\int_0^{\delta}(1-G(v)) \dd v = m_1$ (obtained by integration by part) to simplify the bound. To develop the second one, we first put $u^{(N)}(t,0)$ and $\partial_xu^{(N)}(t,0)$ in the integral, by using that $\int_0^{\delta}(1-G(v)) \dd v = m_1$ and $\int_0^{\delta}v(1-G(v)) \dd v = \frac{m_2}{2}$, and then use a Taylor expansion to develop~$u^{(N)}\left(s,\frac{v}{N}\right) - u^{(N)}(t,0) - \frac{v}{N}\partial_xu^{(N)}(t,0)$. We obtain that for all $t\geq0$
\begin{align}
	\left|n_{\partial}^{(N)}(t) - u_{\partial}^{(N)}(t)\right| &\leq 
	b\left(\frac{c_1t}{N^2}\exp\left[-bm_1\lambda_N t\right]   +  \frac{1}{C_N}\left(\frac{c_2}{N^3} + \frac{c_3}{N^2}\right)\exp\left[-\beta_Nt\right]\right)\int_0^{\delta}(1-G(v))\dd v  \nonumber\\
	&+b\left|\int_0^{\delta}  \left[u^{(N)}\left(s,\frac{v}{N}\right) - u^{(N)}(t,0) - \frac{v}{N}\partial_xu^{(N)}(t,0)\right] \left(1-G(v)\right)\dd v\right| \nonumber \\
	&= \frac{bm_1c_1t}{N^2}\exp\left[-bm_1\lambda_N t\right] +  \frac{bm_1}{C_N}\left(\frac{c_2}{N^3} + \frac{c_3}{N^2}\right)\exp\left[-\beta_Nt\right] \label{eq:proof_approximation_cemetery_intermediate_second} \\ 
	&+ b\left|\int_0^{\delta}\left[\int_0^1\left(1-w\right)\partial_{x}^2u^{(N)}\left(t,w\frac{v}{N}\right)\frac{v^2}{N^2}\dd w\right]\left(1-G(v)\right)\dd v \right|. \nonumber
\end{align}
It thus only remains to bound the last term of~\eqref{eq:proof_approximation_cemetery_intermediate_second}, and Theorem~\ref{te:model_approximation}-\ref{te:model_approximation_second} will be proved. To do so, we first apply Lemma~\ref{lemm:bound_second_derivative} to bound $\partial_{x}^2u^{(N)}\left(t,w\frac{v}{N}\right)$ by $D_{\lambda}\exp\left[-bm_1\lambda_Nt\right]$. Then, we use the equalities \hbox{$\int_0^{1} \left(1-w\right) \dd w = \frac{1}{2}$} and~\hbox{$\int_0^{\delta}v^2(1-G(v)) \dd v = \frac{m_3}{3}$}. We obtain that~\eqref{eq:main_result_second_statement} is true for the constants~$c_1$,~$c_2$ and~$c_3$ defined at the end of Section~\ref{subsubsect:proof_approximation_density_lengths}, and~$c_4 = \frac{D_{\lambda}bm_3}{6}$. \qed

\section{Estimation with the Gaver-Stehfest algorithm}\label{sect:gaver_stehfest}
In this section, we use the theoretical results presented in Section~\ref{subsect:main_result_chaplaplace} to solve our inverse problem. First, in Section~\ref{subsect:estimation_noise_free_data}, we present the estimator we have constructed in the case where we have noise-free data, and check its quality of estimation on simulated data. Then, in Section~\ref{subsect:estimation_random_variables}, we do the same in the case where there is noise  related to sampling in our data.


\subsection{Estimation on noise-free data}\label{subsect:estimation_noise_free_data}

In this section, we assume that we observe perfectly the cemetery~$n_{\partial}^{(N)}$, where $N>0$. Our aim is to retrieve $n_0$ from this observation. This setting is not realistic, since in practice, there is always noise on $n_{\partial}^{(N)}$. However, this study has an interest since it allows us to present the best estimation of $n_0$ we can have with our inference method. First, in Section~\ref{subsubsect:estimation_gaver_stehfest_noise_free}, we construct an estimator of $n_0$ from $n_{\partial}^{(N)}$ based on Proposition~\ref{prop:link_laplace_transforms}. Then, in Section~\ref{subsubsect:estimation_results_noise_free}, we show estimations results with this estimator. Finally, in Section~\ref{subsubsect:difficulties_small_variability}, we study the limitations of this estimator.

\subsubsection{Estimator of \texorpdfstring{$n_0$}{n\_0} on noise-free data}\label{subsubsect:estimation_gaver_stehfest_noise_free}
The fact that~$n_{\partial}^{(N)}$ can be approximated by~$u_{\partial}^{(N)}$ when $N$ is large (see~\eqref{eq:main_result_second_statement}), and that the Laplace transform of $n_0$ can be computed thanks to the Laplace transform of $u_{\partial}^{(N)}$ (see Proposition~\ref{prop:link_laplace_transforms}), gives us the intuition that we can estimate $n_0$ from $n_{\partial}^{(N)}$ by doing a Laplace transform inversion. Precisely, in view of~\eqref{eq:link_laplace_transforms}, our aim is to invert the Laplace transform of the following function to estimate~$n_0$
$$
p\in\mathcal{P}_N \mapsto \left[1 + p \frac{bm_1}{(bm_1)^2\frac{2N}{bm_2} +q_N(p)}\right]\mathcal{L}\left(n_{\partial}^{(N)}\right)\left(q_N(p)\right).
$$
Different numerical methods exist to invert Laplace transforms~\cite{cohen_numerical_2007}. A large part of them are inspired by either the Bromwich integral formula (see Fourier series method~\cite{abate_numerical_1992}, Talbot's method~\cite{talbot_accurate_1979}), or the Post–Widder formula (see Gaver-Stehfest algorithm~\cite{gaver_observing_1966,stehfest_algorithm_1970,stehfest_remark_1970}, Jagerman's method~\cite{jagerman_inversion_1982}). To recall, in view of~\cite[Theorem~$7.3$, p.66]{widder_laplace_1946}, the Bromwich integral formula states that for all $f\in L^1_{loc}\left(\mathbb{R}_+\right)$ continuous with local bounded variations, it holds for all $\gamma >\mathcal{R}(f)$ and $x>0$
\begin{equation}\label{eq:dft_bromwich_integral}
	f(x) = \frac{1}{2\pi i}\int_{\eta - i \infty}^{\eta + i \infty}\mathcal{L}(f)(p)e^{px} dp. 
\end{equation}
The Post-Widder formula, given in~\cite[Theorem~$6.a$, p.288]{widder_laplace_1946}, states for its part that for all $f\in L^1\left(\mathbb{R}_+\right)$ continuous and $x>0$, it holds 
\begin{equation}\label{eq:post_widder_formula}
	f(x) = \lim_{n \rightarrow +\infty} \frac{(-1)^{n}}{n!} \left(\frac{n}{x}\right)^{n+1} \mathcal{L}^{(n)}(f)\left(\frac{n}{x}\right).
\end{equation}
Due to the fact that it is required to compute or approximate the $n$-th derivative of $\mathcal{L}(f)$ to use a method based on the Post-Widder formula, these methods are often less stable numerically. In our case, we have a link between the Laplace transform on the sets $\left(\mathcal{P}_N\right)_{N>0}$ instead of the full complex plane, see Proposition~\ref{prop:link_laplace_transforms}. These sets do not have an interval of the form~\hbox{$[\gamma - i \infty, \gamma + i\infty]$}, where $\gamma \in \mathbb{R}$, as illustrated in Figure~\ref{fig:laplace_set_chapter}. Thus, in view of~\eqref{eq:dft_bromwich_integral}, methods based on the Bromwich integral formula are more difficult to apply in our context. As a result, we use a method based on the Post-Widder formula, even if it is less stable. 
\needspace{2\baselineskip}
\begin{figure}[!htb]
	\centering
	\begin{subfigure}[t]{0.478\textwidth}
		\centering
		\includegraphics[scale = 0.25]{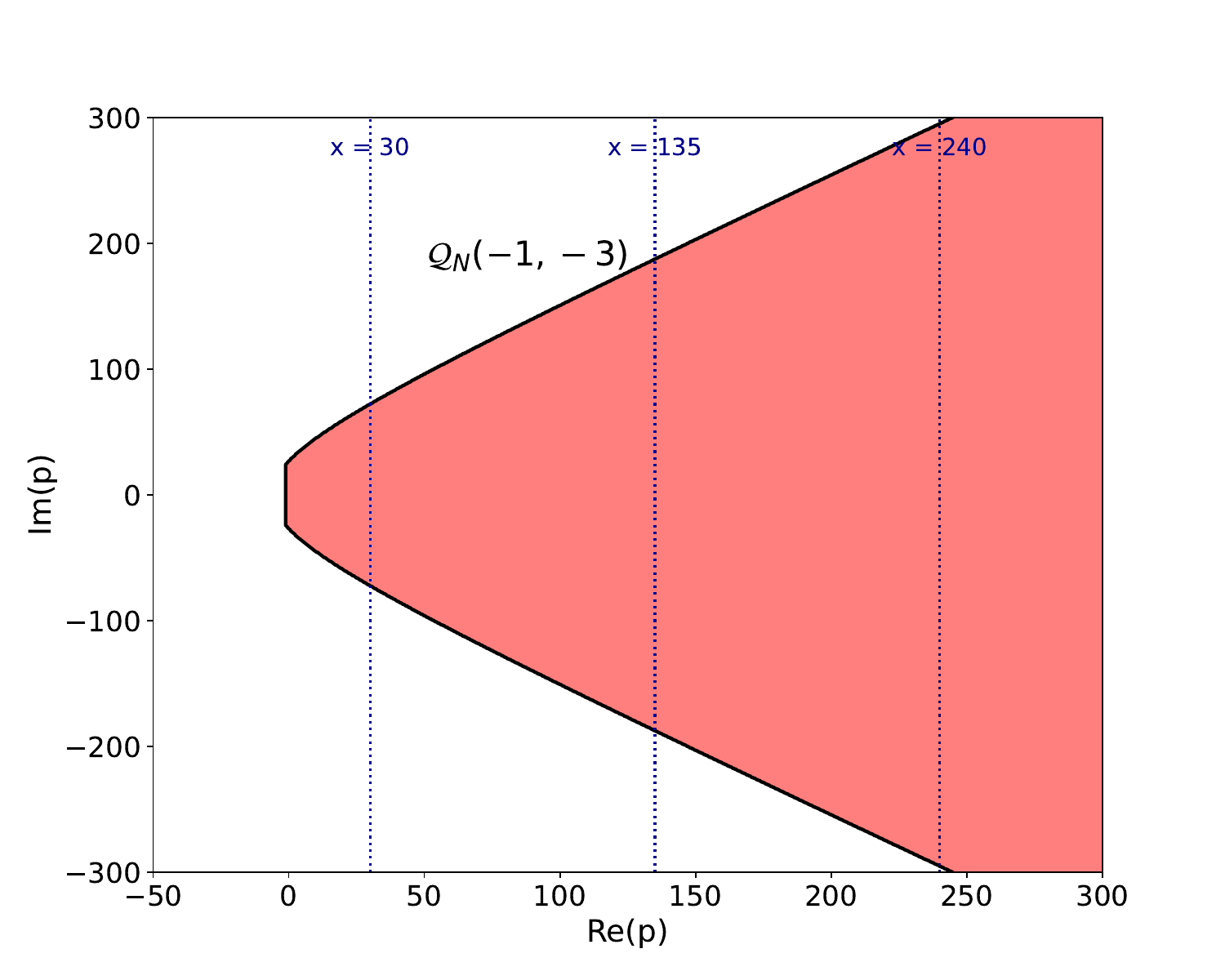}
		\caption{Representation when $(\alpha,\beta) = (-1,-3)$.}
	\end{subfigure}
	\hfill
	\begin{subfigure}[t]{0.478\textwidth}
		\centering
		\includegraphics[scale = 0.25]{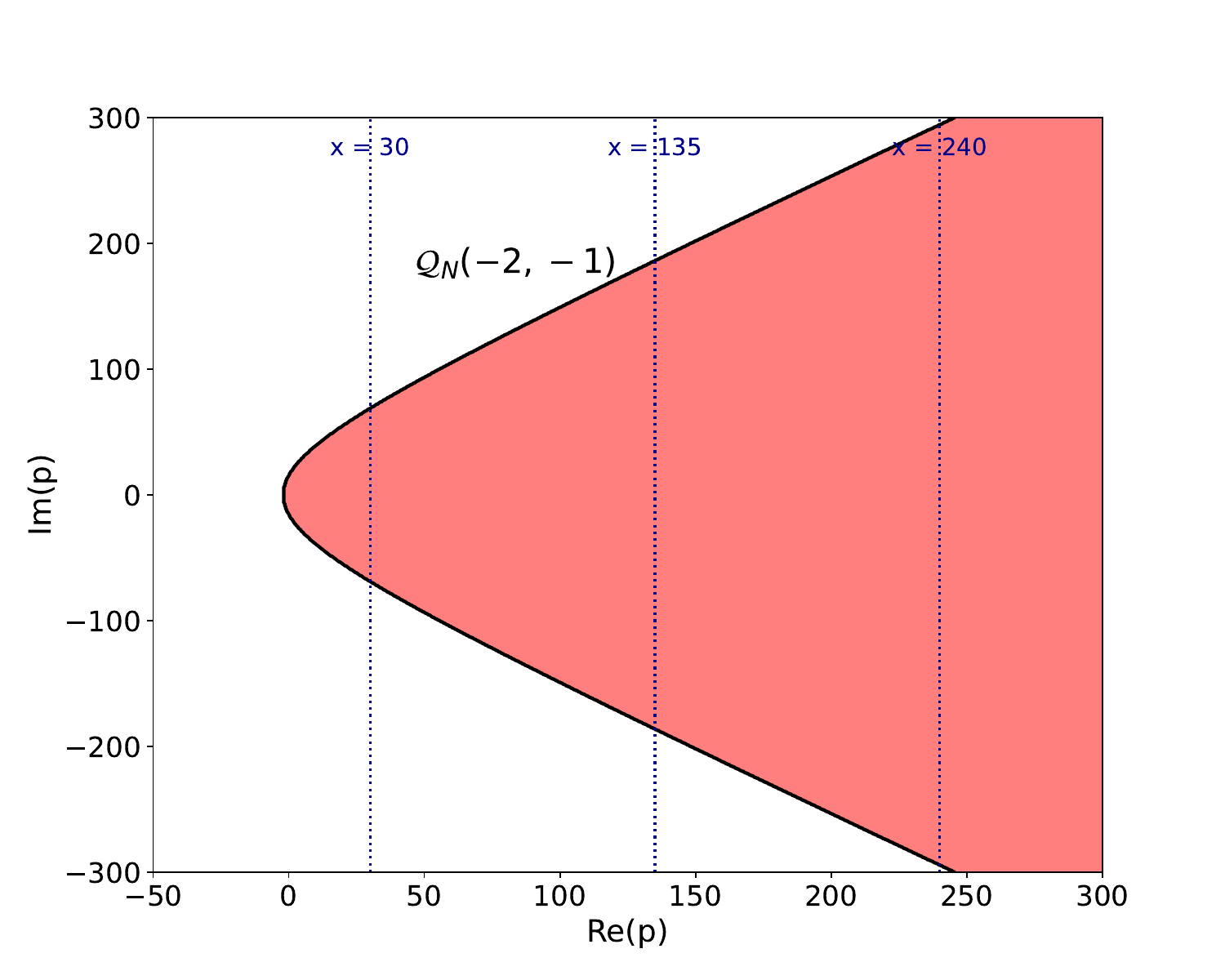}
		\caption{Representation when $(\alpha,\beta) = (-2,-1)$.}
	\end{subfigure}
	\caption{Representation of the set $Q_N(\alpha,\beta) =\left\{p\in\mathbb{C}\,|\,Re(p) > \alpha,\,Re(q_N(p)) > \beta\right\}$ (in red) for different values of $(\alpha, \beta)\in\left(\mathbb{R}_{-}\right)^2$, and comparison with vertical lines (in blue), when $b = 1$, $g = 1_{[0,1]}$ and~$N = 40$.\textit{ We see that as the sets $\left(Q_N(\alpha,\beta)\right)_{(\alpha,\beta)\in\left(\mathbb{R}_{-}\right)^2}$ have a conical shape, they do not contain vertical lines. Since the set $\mathcal{P}_N$ defined in~\eqref{eq:definition_P} has the same form as the sets $\left(Q_N(\alpha,\beta)\right)_{(\alpha,\beta)\in\left(\mathbb{R}_{-}\right)^2}$, this figure illustrates that $\mathcal{P}_N$ does not contain an interval of the form~$[\gamma - i\infty,\gamma + i\infty]$, where $\gamma\in\mathbb{R}$.}}\label{fig:laplace_set_chapter}
\end{figure}

The most famous numerical method to invert Laplace transforms based on the Post-Widder formula is the \textit{Gaver-Stehfest algorithm}~\cite{gaver_observing_1966,stehfest_algorithm_1970,stehfest_remark_1970}. This algorithm has the advantage of being easy to implement, and very accurate when the Laplace transform is explicit. In addition, theoretical results exist to justify the quality of this method~\cite{kuznetsov_convergence_2013,kuznetsov_rate_2022}. This method states that for all $f \in L^1\left(\mathbb{R}_+\right)$ continuous with local bounded variations, we can invert $\mathcal{L}(f)$ by using the following formula (see~\cite[Theorem~$1.1$]{kuznetsov_convergence_2013})
\begin{equation}\label{eq:gaver_stehfest_formula}
	\lim_{K\rightarrow +\infty} \frac{\log(2)}{x}\sum_{n = 1}^{2K} \frac{(-1)^{n+K}}{K!} \sum_{j = \left\lfloor\frac{n+1}{2}\right\rfloor}^{\min(n,K)} j^{K+1} \binom{K}{j}\binom{2j}{j}\binom{j}{n-j}\mathcal{L}(f)\left(n\frac{\log(2)}{x}\right) = f(x).
\end{equation}
The main idea behind this algorithm is to compute finite differences instead of the derivatives stated in~\eqref{eq:post_widder_formula}, and to use a convergence acceleration method to speed up the inversion. We refer to~\cite{kuznetsov_convergence_2013} for more information. In the current article, we propose to use this algorithm to estimate~$n_0$. Specifically, in view of~\eqref{eq:link_laplace_transforms} and~\eqref{eq:gaver_stehfest_formula}, we would like to use the following estimator to estimate $n_0$ from~$n_{\partial}^{(N)}$, that depends on a parameter $K\in\mathbb{N}^*$ (chosen as large as numerical stability allows, see Appendix~\ref{sect:round_off_errors}),
\begin{equation}\label{eq:estimator_gaver_stehfest_noise_free}
	\begin{aligned}
		\forall x>0: \hspace{2mm}\widehat{n}_0^{(N,K)}(x) &= \frac{\log(2)}{x}\sum_{n = 1}^{2K} \frac{(-1)^{n+K}}{K!} \sum_{j = \left\lfloor\frac{n+1}{2}\right\rfloor}^{\min(n,K)} j^{K+1} \binom{K}{j}\binom{2j}{j}\binom{j}{n-j}\\ 
		&\times \left[1 + n\frac{\log(2)}{x} \frac{bm_1}{(bm_1)^2\frac{2N}{bm_2} +q_N\left(n\frac{\log(2)}{x}\right)}\right]\mathcal{L}\left(n_{\partial}^{(N)}\right)\left(q_N\left(n\frac{\log(2)}{x}\right)\right).
	\end{aligned}
\end{equation}
No theoretical result providing a bound on the error between $\widehat{n}_0^{(N,K)}$ and $n_0$ has been obtained for the moment, mainly because the fact that the set $\mathcal{P}_N$ does not contain intervals of the form $[\gamma - i \infty, \gamma + i\infty]$ yields difficulties when proving such a result. This however corresponds to a work in progress. Despite the fact we do not have this result, we can still check on simulations if this estimator has good estimation results or not. This is what we do in the next sections.

\subsubsection{Estimation results on noise-free data}\label{subsubsect:estimation_results_noise_free}

We begin by introducing for which functions $n_0$ we check the quality of our estimator. We define for all $(\ell,\beta)\in\left(\mathbb{R}_+^*\right)^2$ the following function, for all $x\geq0$,
\begin{equation}\label{eq:density_gamma_distribution}
h_{\ell,\beta}(x) := \frac{\beta^{\ell}}{\Gamma(\ell)}x^{\ell-1}\exp\left(-\beta x\right).
\end{equation}
This function corresponds to the probability density function of a Gamma distribution with parameter $(\ell,\beta)$. In the estimations presented here, we choose $n_0$ belonging to the set~$ \left\{h_{\ell,\beta}\,|\,(\ell,\beta)\in\mathbb{N}^*\times\mathbb{R}_+^*,\,\ell\geq4\right\}$ for three reasons. The first reason is that the following proposition holds for such initial distributions. It allows us to compute~\eqref{eq:estimator_gaver_stehfest_noise_free} without having to do a numerical approximation of the Laplace transform and it is proved in Appendix~\ref{appendix:explicit_laplace_erlang}.
\begin{prop}\label{prop:explicit_laplace_erlang}
	Assume that there exists $\left(\ell,\beta\right)\in\mathbb{N}^*\times\mathbb{R}_+^*$ such that $n_0 = h_{\ell,\beta}$. Then, for all $N>0$, there exist $\left(a_{0,N},\hdots,a_{\ell-1,N}\right)\in\left(\mathbb{R}_+^*\right)^{\ell }$ and $\tilde{\beta}_N >0$, such that for all $p\in\mathbb{C}$ verifying~$\text{Re}\left(p\right) > -bm_1 \tilde{\beta}_N$
	\begin{equation}\label{eq:explicit_laplace_erlang}
 	\mathcal{L}\left(n_{\partial}^{(N)}\right)(p) = \sum_{i = 0}^{\ell -1} a_{i,N}\frac{i!}{\left(p + bm_1 \tilde{\beta}_N\right)^{i+1}}.
	\end{equation}
	Moreover, $\left(a_{0,N},\hdots,a_{\ell-1,N}\right)$ and $\tilde{\beta}_N$ can be explicitly computed thanks to $b$, $N$, and the sets of functions~$\left(\mathcal{L}(g)\text{Id}^i\right)_{j\in\llbracket0,\ell-1\rrbracket}$ and $\left(\mathcal{L}\left(\left(1-G\right)\text{Id}^i\right)\right)_{j\in\llbracket0,\ell-1\rrbracket}$ restricted on $\mathbb{R}_+^{*}$.
\end{prop}
\begin{rem}
	When $g = \frac{1}{\delta}1_{[0,\delta]}$, we can obtain an explicit form for the functions in the sets~$\left(\mathcal{L}\left(g\text{Id}^j\right)\right)_{j\in\llbracket0,\ell-1\rrbracket}$ and $\left(\mathcal{L}\left(\left(1-G\right)\text{Id}^j\right)\right)_{j\in\llbracket0,\ell-1\rrbracket}$. For example, in view of~\cite[p.$137$]{Ibe_2005}, one has that for all $j\in\llbracket0,\ell-1\rrbracket$, $\beta >0$,
	$$
	\mathcal{L}\left(g\text{Id}^j\right)(\beta) = \frac{1}{\delta}\int_0^{\delta} x^j\exp\left(-\beta x\right) \dd x = \frac{j!}{\delta\beta^{j+1}}\left[1 - \sum_{i = 0}^{j}\frac{\left(\beta x\right)^i}{i!}\exp\left(-\beta x\right)\right].
	$$
	For $\left(\mathcal{L}\left(\left(1-G\right)\text{Id}^j\right)\right)_{j\in\llbracket0,\ell-1\rrbracket}$, a similar argument can be done since $G = \frac{1}{\delta}\text{Id}1_{[0,\delta]}$. However, the formula is slightly more complicated.
\end{rem}
\noindent The second reason we use the functions in the set $\left\{h_{\ell,\beta}\,|\,(\ell,\beta)\in\mathbb{N}^*\times\mathbb{R}_+^*,\,\ell\geq4\right\}$ is that they are unimodal, so biologically relevant, see~\cite{Xu2013}. The last reason is that the assumptions of Theorem~\ref{te:model_approximation} and Proposition~\ref{prop:link_laplace_transforms} are verified for $n_0 = h_{\ell,\beta}$ when $\ell \geq 4$ and $\beta > 0$. This can be easily obtained from the fact that by Leibniz's formula and the change of variable $j' = \ell-1 -j$, it holds for all $n\in\llbracket1,\ell-1\rrbracket$ and~$x\geq0$
\begin{equation}\label{eq:n_th_derivative_h_ell_beta}
\begin{aligned}
\frac{\dd^n}{\dd x^n} h_{\ell,\beta}(x) &= \frac{\beta^{\ell}}{(\ell-1)!} \sum_{j = 0}^{n}\binom{n}{j} \left[\frac{(\ell-1)!}{(\ell-1 -j)!}x^{\ell-1 -j}\right] \left[\left(-\beta\right)^{n-j}\exp\left(-\beta x\right)\right] \\
&= \sum_{j' = \ell-1 -n}^{\ell-1}\left(-1\right)^{n- (\ell-1) + j'}\binom{n}{\ell-1 -j'} \frac{\left(\beta\right)^{n+j'+1}}{j'!} x^{j'} \exp\left(-\beta x\right),
\end{aligned}
\end{equation}
and from the fact that for all $j'\geq1$, $\lambda\in(0,\beta)$ and $\lambda'>0$, there exists $C_{j',\lambda,\lambda'}>0$ such that 
$$
\forall x\geq0: \hspace{2.5mm}x^{j'}\exp(-\beta x) \leq C_{j',\lambda,\lambda'}\exp(-\lambda x)\left(1-\exp(-\lambda'x)\right). 
$$
\begin{rem}\label{rem:coefficient_variation_gamma}
	To recall, in view of~\eqref{eq:density_gamma_distribution} and~\cite[p.$109$]{forbes_statistical_2011}, the coefficient of variation of the distribution associated to $n_0 = h_{\ell,\beta}$, where $(\ell,\beta) \in \mathbb{N}^*\times\mathbb{R}_+^*$, is equal to $\ell^{-\frac{1}{2}}$.
\end{rem}

We now proceed to the estimations. First, we fix for model parameters $N = 40$, $b=1$ and~\hbox{$g = 1_{[0,1]}$}, which are the parameters we  used in~\cite{olaye_transport_2026}. The scaling parameter has been chosen as $N = 40$ because this is the most realistic value for the budding yeast, see~\cite[Section~$2.3$]{olaye_transport_2026}. This species is the one that motivates our study. Then, for \hbox{$n_0\in\left\{h_{9,12},h_{16,16},h_{25,30},h_{49,50}\right\}$}, we compute $\mathcal{L}\left(n_{\partial}^{(N)}\right)$ thanks to~\eqref{eq:explicit_laplace_erlang}, and use it to retrieve $n_0$ by using~$\widehat{n}_0^{(N,K)}$ defined in~\eqref{eq:estimator_gaver_stehfest_noise_free}, for~$K = 50$. We plot in Figure~\ref{fig:estimation_gaver_stehfest_noisefree} the estimated curves of $n_0$ (in red), and compare them with their associated theoretical curve (in black), and the estimator of $n_0$ constructed in~\cite{olaye_transport_2026} (in blue). To recall, this estimator is defined as 
$\widehat{n}^ {(\text{old},N)}_0(x) := \frac{1}{bm_1}n_{\partial}^{(N)}\left(\frac{x}{bm_1}\right)$ for all $x\geq0$. We observe in Figure~\ref{fig:estimation_gaver_stehfest_noisefree} that the estimations obtained with $\widehat{n}^ {(N,K)}_0$ are much more satisfactory than the ones with~$\widehat{n}_0^{(\text{old},N)}$. In particular, the diffusivity of $n_0$ is each time well-estimated with~$\widehat{n}^{(N,K)}_0$, which is not the case for the estimator~$\widehat{n}_0^{(\text{old},N)}$ (we refer to~\cite[Section~$5.1.2$]{olaye_transport_2026} to know why). Our new estimator seems thus to improve the estimation obtained with $\widehat{n}_0^{(\text{old},N)}$.

\begin{figure}[!htb]
	\centering
	\begin{subfigure}[t]{0.485\textwidth}
		\centering
		\includegraphics[scale = 0.35]{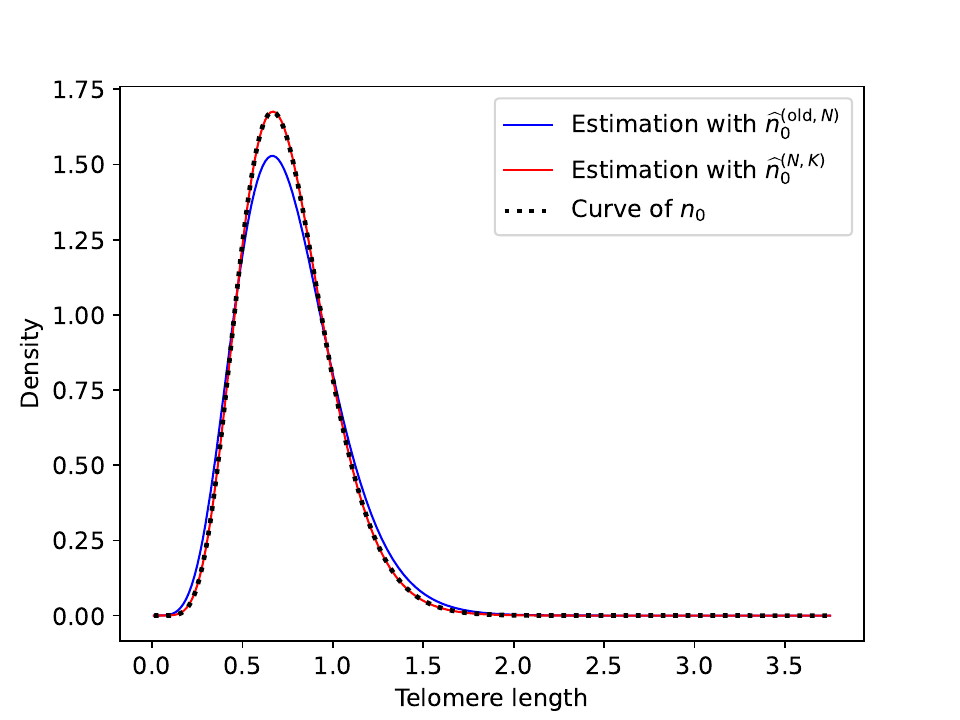}
		\caption{Estimation for $n_0 = h_{9,12}$.}\label{fig:estimation_gaver_stehfest_noisefree_first}
	\end{subfigure}
	\hfill
	\begin{subfigure}[t]{0.485\textwidth}
		\centering
		\includegraphics[scale = 0.35]{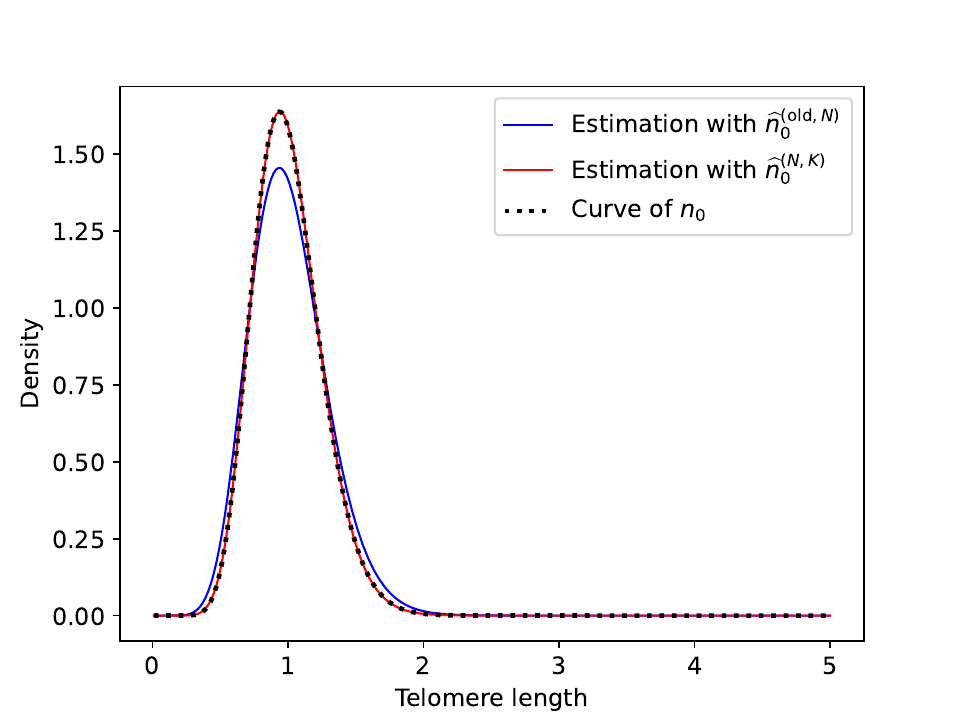}
		\caption{Estimation for $n_0 = h_{16,16}$.}\label{fig:estimation_gaver_stehfest_noisefree_second}
	\end{subfigure}
	\begin{subfigure}[t]{0.485\textwidth}
		\centering
		\includegraphics[scale = 0.35]{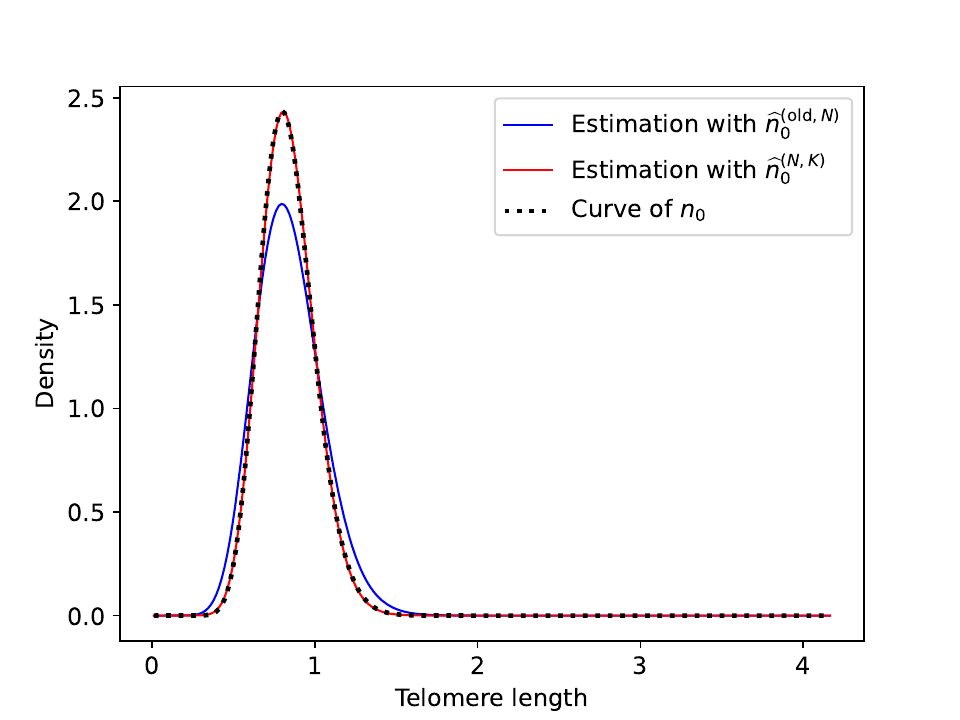}
		\caption{Estimation for $n_0 = h_{25,30}$.}\label{fig:estimation_gaver_stehfest_noisefree_third}
	\end{subfigure}
	\hfill
	\begin{subfigure}[t]{0.485\textwidth}
		\centering
		\includegraphics[scale = 0.35]{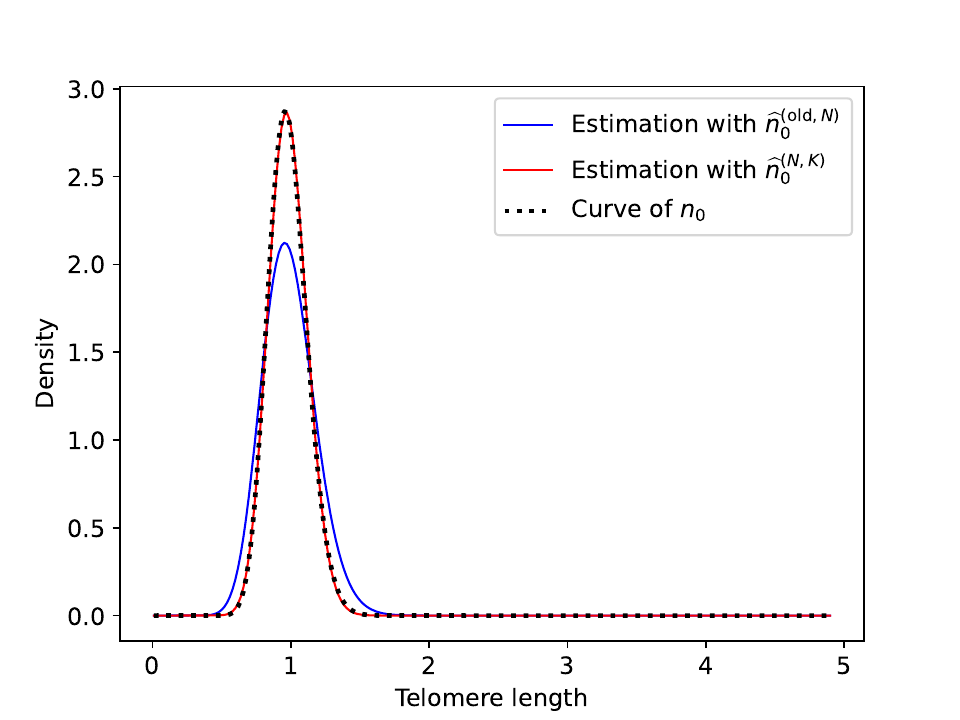}
		\caption{Estimation for $n_0 = h_{49,50}$.}\label{fig:estimation_gaver_stehfest_noisefree_fourth}
	\end{subfigure}
	\caption{Estimation results with the estimator $\widehat{n}_0^{(N,K)}$ defined in~\eqref{eq:estimator_gaver_stehfest_noise_free} when $b = 1$, $g = 1_{[0,1]}$, $N = 40$,~$K=50$ and $n_0\in\left\{h_{9,12},h_{16,16},h_{25,30},h_{49,50}\right\}$. Comparison with the results obtained with the estimator~$\widehat{n}_0^{(\text{old},N)}(x) =\frac{1}{bm_1}n_{\partial}^{(N)}\left(\frac{x}{bm_1}\right)$ for all $x\geq0$, constructed in~\cite{olaye_transport_2026}. \textit{For plotting the curves, the numerical precision of the computations was set to $200$ digits.}}\label{fig:estimation_gaver_stehfest_noisefree}
\end{figure}

To illustrate that the improvement observed in Figure~\ref{fig:estimation_gaver_stehfest_noisefree} is related to the fact that we have constructed our new estimator with a finer model approximation, we now fix for initial distribution~$n_0 = h_{9,12}$. Then, we proceed to estimations with both $\widehat{n}_0^{(\text{old},N)}$ and~$\widehat{n}^ {(N,K)}_0$ for the same parameters $b$, $g$ and $K$ as before, and for $N \in\left\{1 + 4\ell\,|\,\ell\in\llbracket0,25\rrbracket\right\}$. We plot in Figure~\ref{fig:curve_estimation_errors_logscale} the estimation errors in $L^1$-norm we have obtained as a function of $N$, at the logarithm scale, in blue and red respectively. We observe that the estimation error of $\widehat{n}^{(N,K)}_0$ goes more quickly to~$0$ when $N\rightarrow+\infty$. In particular, we show in this figure that there exists~\hbox{$C_1 > 0$} such that the curve of~$\left|\left|\widehat{n}_0^{(\text{old},N)} - n_0\right|\right|_{L^1\left(\mathbb{R}_+\right)}$ as a function of $N$ decreases to~$0$ at a speed similar to the one of the function \hbox{$N\in\mathbb{R}_+^*\mapsto \frac{C_1}{N}$}. Similarly, there exists $C_2 > 0$ such that the curve of~$\left|\left|\widehat{n}_0^{(N,K)} - n_0\right|\right|_{L^1\left(\mathbb{R}_+\right)}$ as a function of $N$ decreases to~$0$ at a speed similar to the curve of the function~$N\in\mathbb{R}_+^*\mapsto \frac{C_2}{N^2}$. This is coherent with the fact that the model approximation used to construct $\widehat{n}_0^{(\text{old},N)}$ decreases to~$0$ at a speed of the order~$\frac{1}{N}$  (see~\cite[Proposition~$3.2$]{olaye_transport_2026}), and that the one to construct $\widehat{n}_0^{(N,K)}$ decreases to~$0$ at a speed~$\frac{1}{N^2}$ (see~Theorem~\ref{te:model_approximation}). For~the estimator~$\widehat{n}_0^{(\text{old},N)}$, we have obtained in~\cite[Theorem~$2.7$]{olaye_transport_2026} a theoretical result justifying that the error between $\widehat{n}_0^{(\text{old},N)}$ and~$n_0$ decreases to $0$ at a speed $\frac{1}{N}$ when $N\rightarrow+\infty$. As said in Section~\ref{subsubsect:estimation_gaver_stehfest_noise_free}, obtaining a theoretical result on the error between $\widehat{n}_0^{(N,K)}$ and $n_0$ is not that easy and corresponds to a work in progress. We can however already conjecture from Figure~\ref{fig:curve_estimation_errors_logscale} that this error decreases  to $0$ at a speed $\frac{1}{N^2}$ when $N\rightarrow+\infty$.


\begin{figure}[!htb]
	\centering
	\includegraphics[scale = 0.35]{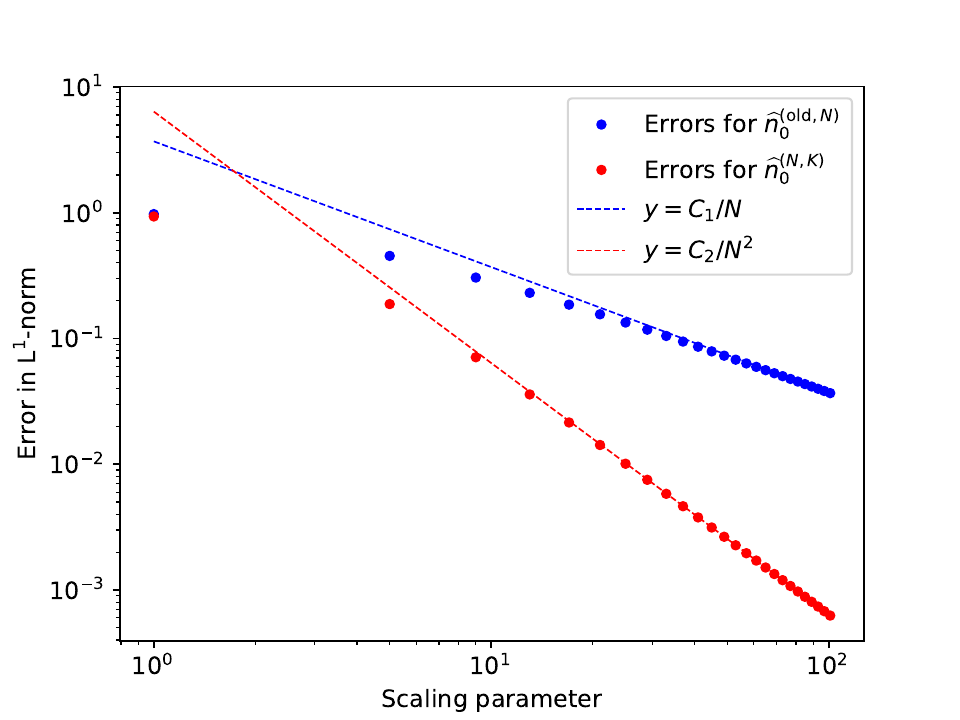}
	\caption{Evolutions of the estimation errors in $L^1$-norm of~$\widehat{n}_0^{(N,K)}$ and $\widehat{n}_0^{(\text{old},N)}$ as a function of~$N\in \left\{1 + 4\ell\,|\,\ell\in\llbracket0,25\rrbracket\right\}$, when $b = 1$, $g = 1_{[0,1]}$, $N = 40$,~$K=50$ and $n_0 = h_{9,12}$. Comparison with the curves of the functions $N\in\mathbb{R}_+^{*} \mapsto \frac{C_1}{N}$ and $N\in\mathbb{R}_+^{*} \mapsto \frac{C_2}{N^2}$, where~$C_1 > 0$ and $C_2 > 0$ are estimated thanks to a constrained least-squares regression on the last error points. \textit{For plotting the curves, the numerical precision of the computations was set to~$200$~digits.}}\label{fig:curve_estimation_errors_logscale}
\end{figure}
\subsubsection{Estimation difficulties for a very small initial variability}\label{subsubsect:difficulties_small_variability}


In~\cite[Section~$5.1.2$]{olaye_transport_2026}, we have shown that the estimation with~$\widehat{n}_0^{(\text{old},N)}$ does not work well when $n_0$ corresponds to the density function of a distribution with a small variability and~\hbox{$N = 40$}. What we called \textit{small variability} in this article was when $n_0$ is associated to a coefficient of variation of the order of $0.15-0.2$. Then, in the current work, we have shown in Figures~\ref{fig:estimation_gaver_stehfest_noisefree_third} and~\ref{fig:estimation_gaver_stehfest_noisefree_fourth} that when $n_0$ has a coefficient of variation of this order, see Remark~\ref{rem:coefficient_variation_gamma}, our new estimator $\widehat{n}_0^{(N,K)}$ allows us to manage this issue when $K\in\mathbb{N}^*$ is large. It thus remains to check how the quality of estimation evolves when $n_0$ has a coefficient of variation of a smaller order than the one studied in  \cite[Section~$5.1.2$]{olaye_transport_2026}, notably to study the limitations of our estimator. To check this, we have done exactly the same protocol as in Section~\ref{subsubsect:estimation_results_noise_free}, this time in the case where $n_0\in\left\{h_{100,120},h_{225,200},h_{324,330},h_{400,500}\right\}$. In view of Remark~\ref{rem:coefficient_variation_gamma}, this case corresponds to having an initial distribution with coefficient of variation~\hbox{$1/10 = 0.1$}, \hbox{$1/15 \approx 0.067$}, \hbox{$1/18 \approx 0.056$} or~\hbox{$1/20 = 0.05$}, so of the order of $0.05-0.1$. The estimation results we have obtained with this protocol are plotted in Figure~\ref{fig:estimation_verysmall_variability}. What we observe is that the estimations obtained with $\widehat{n}^ {(N,K)}_0$ are less satisfactory than in Figure~\ref{fig:estimation_gaver_stehfest_noisefree}, especially in Figures~\ref{fig:estimation_verysmall_variability_third} and~\ref{fig:estimation_verysmall_variability_fourth}. Precisely, the smaller the coefficient of variation of the initial distribution is, the poorer the estimation of the position of the mode and the asymmetry of the curve becomes. In addition, a part of the estimated curve is negative for the estimations with~$\widehat{n}^{(N,K)}_0$ in Figure~\ref{fig:estimation_verysmall_variability}, and the negative part increases when the coefficient of variation of the distribution associated to $n_0$ decreases. In Figures~\ref{fig:estimation_verysmall_variability_first} and~\ref{fig:estimation_verysmall_variability_second}, the error on the estimation of the mode of the curve is not dramatically poor, the negative part of~$\widehat{n}^{(N,K)}_0$ is very small, and the estimations with $\widehat{n}^{(N,K)}_0$ clearly improve the estimations obtained with~$\widehat{n}^{(\text{old},N)}_0$. Our estimator stays therefore useful in this context (coefficients of variation of $0.1$ and~$0.067$), even if it is not perfect. In Figures~\ref{fig:estimation_verysmall_variability_third} and~\ref{fig:estimation_verysmall_variability_fourth}, which correspond to estimations of initial distributions with coefficients of variation $0.056$ and~$0.05$, the estimation of the mode of the curve is not satisfactory at all, and the negative part is quite large. Our new estimator seems thus not very reliable when the coefficient of variation of the initial distribution is very small. We however mention that this limitation of our estimator is not critical. Indeed, having such a small coefficient of variation for the initial distribution is not biologically relevant, see~\cite{Xu2013}. 

\begin{figure}[!htb]
	\centering
	\begin{subfigure}[t]{0.485\textwidth}
		\centering
		\includegraphics[scale = 0.35]{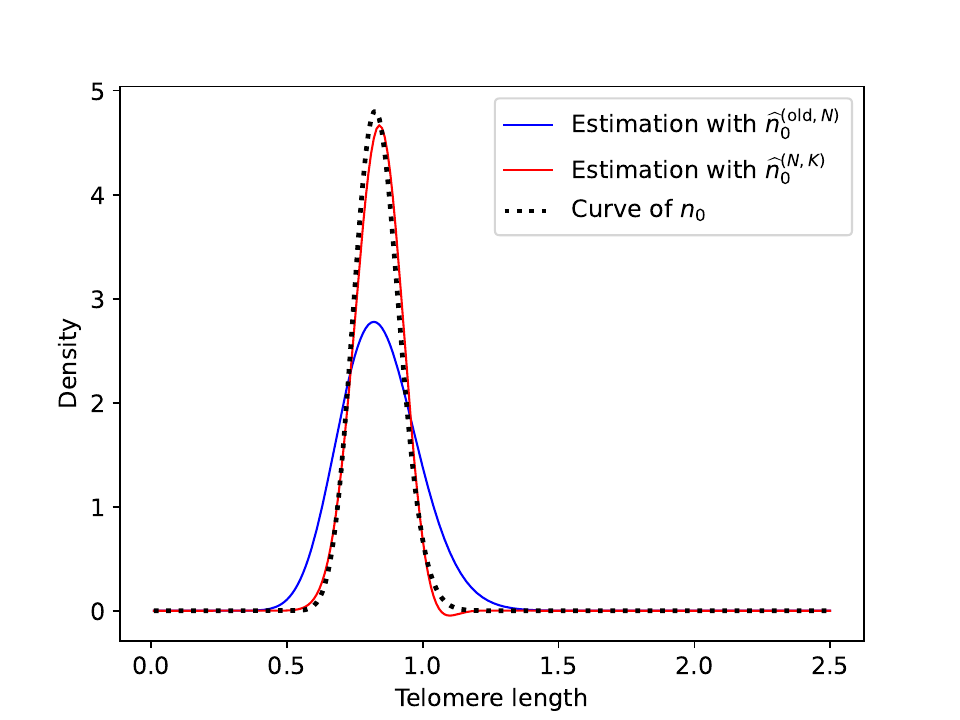}
		\caption{Estimation for $n_0 = h_{100,120}$.}\label{fig:estimation_verysmall_variability_first}
	\end{subfigure}
	\hfill
	\begin{subfigure}[t]{0.485\textwidth}
		\centering
		\includegraphics[scale = 0.35]{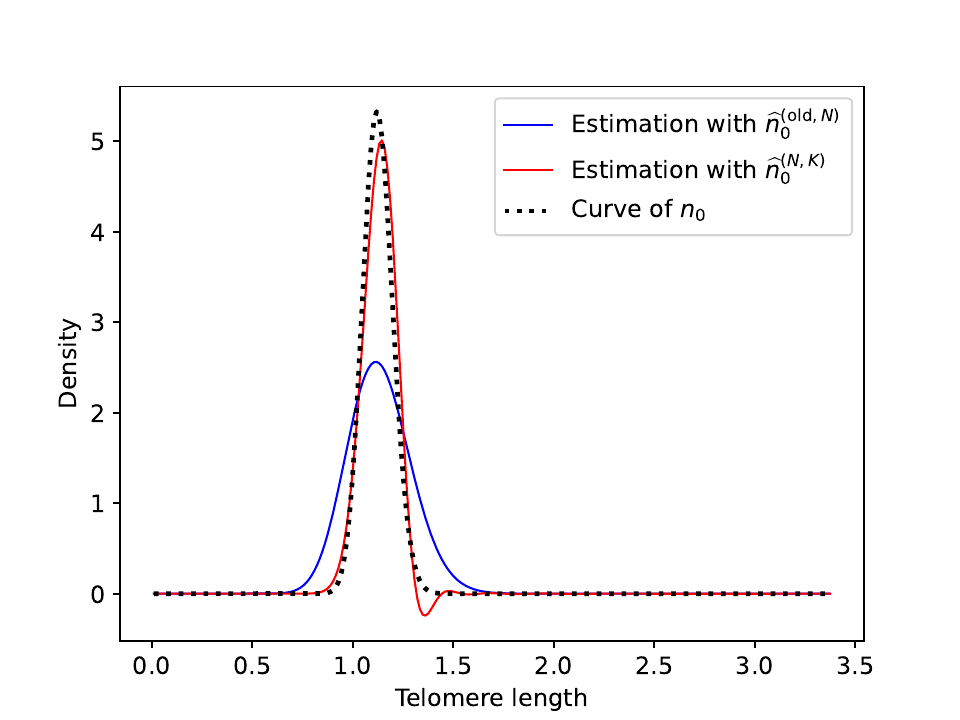}
		\caption{Estimation for $n_0 = h_{225,200}$.}\label{fig:estimation_verysmall_variability_second}
	\end{subfigure}
	\begin{subfigure}[t]{0.485\textwidth}
		\centering
		\includegraphics[scale = 0.35]{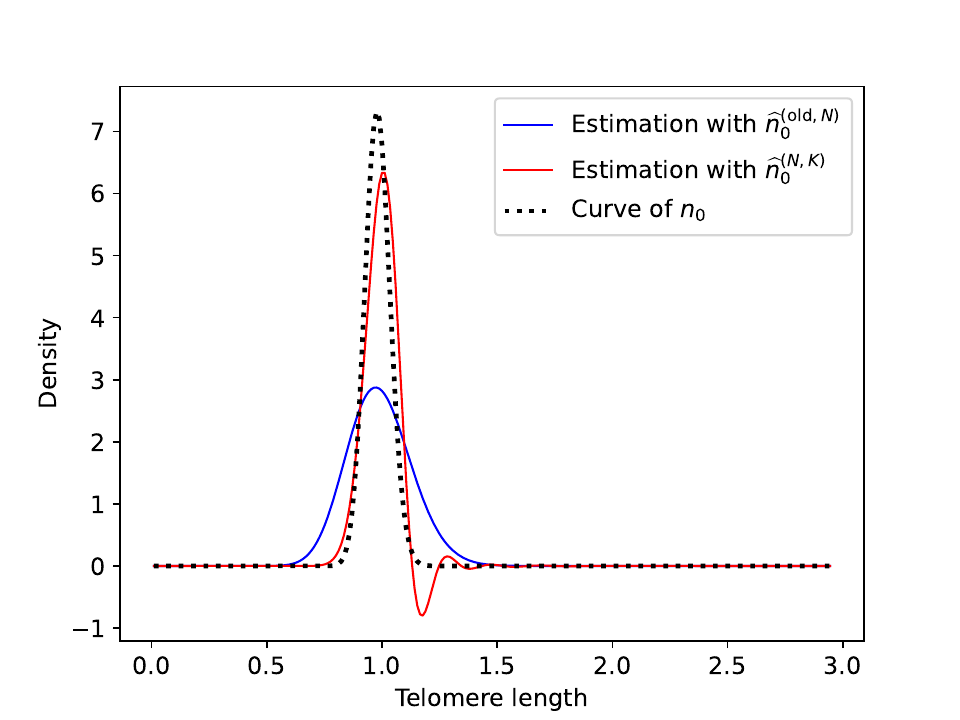}
		\caption{Estimation for $n_0 = h_{324,330}$.}\label{fig:estimation_verysmall_variability_third}
	\end{subfigure}
	\hfill
	\begin{subfigure}[t]{0.485\textwidth}
		\centering
		\includegraphics[scale = 0.35]{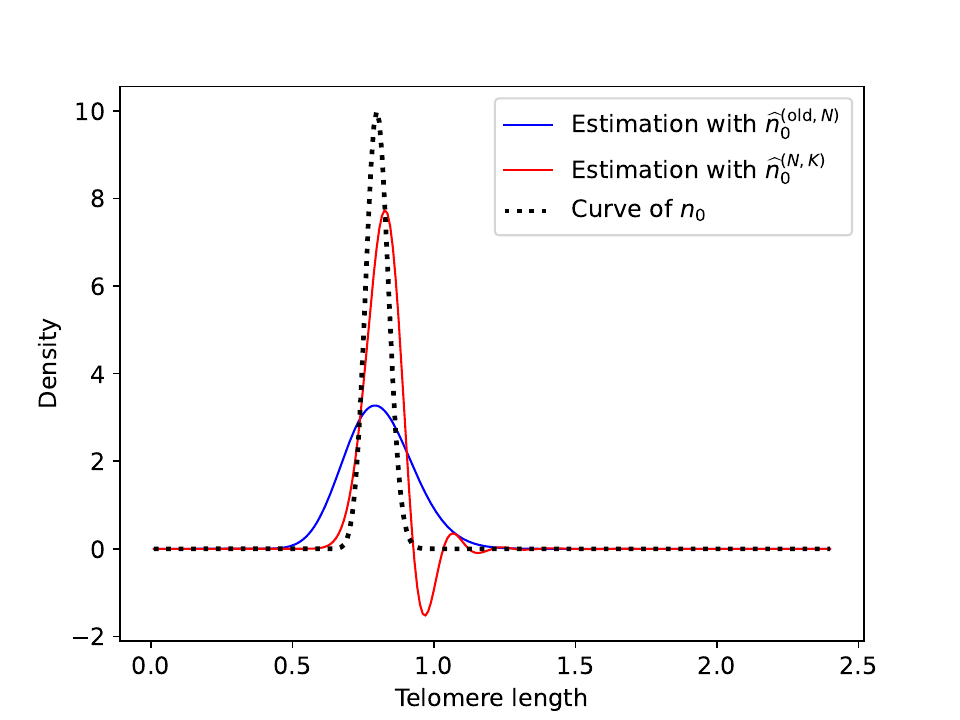}
		\caption{Estimation for $n_0 = h_{400,500}$.}\label{fig:estimation_verysmall_variability_fourth}
	\end{subfigure}
	\caption{Estimation results with the estimator $\widehat{n}_0^{(N,K)}$ defined in~\eqref{eq:estimator_gaver_stehfest_noise_free} when $b = 1$, $g = 1_{[0,1]}$, $N = 40$,~$K=50$ and $n_0\in\left\{h_{100,120},h_{225,200},h_{324,330},h_{400,500}\right\}$. Comparison with the results obtained with the estimator~$\widehat{n}_0^{(\text{old},N)}(x) =\frac{1}{bm_1}n_{\partial}^{(N)}\left(\frac{x}{bm_1}\right)$ for all $x\geq0$, constructed in~\cite{olaye_transport_2026}. \textit{For plotting the curves, the numerical precision of the computations was set to $200$ digits.}}\label{fig:estimation_verysmall_variability}
\end{figure}

\subsection{Estimation when random variables are observed}\label{subsect:estimation_random_variables}

In practice, we never observe $n_{\partial}^{(N)}$. Instead, we observe a sequence of random variables~$\left(T_i^{(N)}\right)_{1\leq i \leq n_d}$, where $n_d\in\mathbb{N}^*$, distributed according to $n_{\partial}^{(N)}$ with possibly noise related to measurement errors. In this section, we check if the estimator defined in~\eqref{eq:estimator_gaver_stehfest_noise_free} can be adapted to this setting, and only take into account the noise related to sampling. First, in Section~\ref{subsubsect:estimation_gaver_stehfest_random_variables}, we present the equivalent of~\eqref{eq:estimator_gaver_stehfest_noise_free} in the case where random variables distributed according to~$n_{\partial}^{(N)}$ are observed. Then, in Section~\ref{subsubsect:choice_smoothing_parameter}, we present how we propose to choose the smoothing parameter of this new estimator. Finally, in Section~\ref{subsubsect:estimation_results_random_variables}, we present estimation~results.
\subsubsection{Estimator of \texorpdfstring{$n_0$}{n\_0} when random variables are observed}\label{subsubsect:estimation_gaver_stehfest_random_variables}


Let $N>0$, $n_d\in\mathbb{N}^*$ and $\left(T_i^{(N)}\right)_{1\leq i \leq n_d}$ a sequence of random variables distributed according to $n_{\partial}^{(N)}$. To find an equivalent to~$\widehat{n}_0^{(N,K)}$ defined in~\eqref{eq:estimator_gaver_stehfest_noise_free} with $\left(T_i^{(N)}\right)_{1\leq i \leq n_d}$, we proceed in two steps. First, we consider an estimator $n_{\partial}^{(N)}$ from $\left(T_i^{(N)}\right)_{1\leq i \leq n_d}$, for which we can compute its Laplace transform easily. Then, the second step is to replace $n_{\partial}^{(N)}$ with its estimator in~\eqref{eq:estimator_gaver_stehfest_noise_free}, to obtain a new estimator of $n_0$.

Following what is done in~\cite[Section~$5.2.1$]{olaye_transport_2026}, the method we use to estimate $n_{\partial}^{(N)}$ from~$\left(T_i^{(N)}\right)_{1\leq i \leq n_d}$ is to use a smoothed version of $\frac{1}{n_{d}}\sum_{i = 1}^{n_d} \delta_{T_i^{(N)}}$, which corresponds to the empirical estimator of $n_{\partial}^{(N)}$. In~\cite{olaye_transport_2026}, the empirical estimator was smoothed by using the fact that for all $i\in\llbracket1,n_d\rrbracket$ and $\alpha >0$, the measure~$\frac{1}{t\alpha}\rho\left(\frac{1}{\alpha}\log\left(\frac{t}{T_i}\right)\right) \dd  t$, where $\rho$ is the density of a standard Gaussian, is a smoothed approximation of $\delta_{T_i}(dx)$. In particular, we have proved in~\cite[Eq.~$5.3$]{olaye_transport_2026} that for all $f\in\mathcal{C}_c\left(\mathbb{R}_+\right)$, $x>0$, it holds 
$$
\lim_{\alpha \rightarrow 0} \int_{t\in\mathbb{R}_+^*} \frac{1}{t\alpha}\rho\left(\frac{1}{\alpha}\log\left(\frac{t}{x}\right)\right)f(t)\dd t = \delta_x(f).
$$
Thus, the following estimator, that depends on a smoothing parameter $\alpha >0$, was what we used to estimate $n_{\partial}^{(N)}$
\begin{equation}\label{eq:estimation_cemetery_lognormal}
\forall t\geq0: \hspace{2mm}\widetilde{n}_{\partial}^{(N,\alpha)}\left(t\right) := \frac{1}{n_d}\sum_{i = 1}^{n_d} \frac{1}{t\alpha}\rho\left(\frac{1}{\alpha}\log\left(\frac{t}{T_{i}^{(N)}}\right)\right).
\end{equation}
This estimator is called log-transform kernel density estimator of $n_{\partial}^{(N)}$~\cite{charpentier_2015,nguyen_positive_2019}.
\begin{rem}\label{rem:approximation_lognormal}
Since $\rho$ is the density of a standard Gaussian, we have for all $\alpha >0$ and~$i\in\llbracket1,n_d\rrbracket$ 
\begin{equation}\label{eq:equality_lognormal}
\frac{1}{t\alpha}\rho\left(\frac{1}{\alpha}\log\left(\frac{t}{T_{i}^{(N)}}\right)\right) \dd t = \frac{1}{\sqrt{2\pi}t\alpha}\exp\left[- \frac{\left(\log(t) - \log\left(T_i^{(N)}\right)\right)^2}{2\alpha^2}\right] \dd t.
\end{equation}
As the right-hand side corresponds to the density of a log-normal distribution with parameters $\left(\log\left(T_i^{(N)}\right),\alpha\right)$,~\eqref{eq:equality_lognormal} means that to construct~$\widetilde{n}_{\partial}^{(N,\alpha)}$, we have approximated Dirac measures in~$\frac{1}{n_d}\sum_{i = 1}^{n_d} \frac{1}{t\alpha}\delta_{T_{i}^{(N)}}$ with the densities of log-normal distributions.
\end{rem}
In this work, we do not use $\widetilde{n}_{\partial}^{(N,\alpha)}$ to estimate $n_{\partial}^{(N)}$. The first reason is that since $\widetilde{n}_{\partial}^{(N,\alpha)}$ is a linear combination of the densities of log-normal distributions (see Remark~\ref{rem:approximation_lognormal}), its Laplace transform is difficult to compute analytically~\cite{asmussen_laplace_2016,miles_2022}. The second one is that we need the existence of $f\in L^1\left(\mathbb{R}_+\right)$ such that  $\mathcal{L}\left(\widetilde{n}_{\partial}^{(N,\alpha)}\right)\left(q_N(p)\right) = \mathcal{L}\left(f\right)(p)$ to have the convergence of the Gaver-Stehfest algorithm, see~\cite[Theorem~$1.1$]{kuznetsov_convergence_2013}, and we do not know if this is the case here. To solve these issues, our idea here is to use the density of Gamma distributions~(see~\eqref{eq:density_gamma_distribution}) to approximate the Dirac measures in $\frac{1}{n_d}\sum_{i = 1}^{n_d} \frac{1}{t\alpha}\delta_{T_{i}^{(N)}}$. Indeed, the Laplace transform of the density of a Gamma distribution is explicit in function of its parameters~$(\ell,\beta)$, and is equal to~$p\in(-\beta,+\infty) \mapsto \left(1+\frac{p}{\beta}\right)^{-k}$ in view of~\cite[p.$109$]{forbes_statistical_2011}. In addition, the composition this Laplace transform with the polynomial~$q_N$ defined in~\eqref{eq:definition_qN} corresponds to the inverse of a polynomial, so admits an inverse Laplace transform, see~\cite[Section~$5.1$]{cohen_numerical_2007}. The parameters we take for the Gamma distributions to approximate the Dirac measures are on the form described below, for all $a >0$ and $\alpha\in\left(0,\log(2)^{\frac{1}{2}}\right)$,
\begin{equation}\label{eq:parameters_gamma_same_mean}
\left(\ell^*(\alpha),\beta^*(a,\alpha)\right) := \left(\left(e^{\alpha^2}-1\right)^{-1},e^{-a-\frac{\alpha^2}{2}}\left(e^{\alpha^2}-1\right)^{-1}\right).
\end{equation}
The reason is that the Laplace transform of a Gamma distribution with such parameters is very close to the one of a log-normal distribution with parameters $\left(a,\alpha\right)$, as illustrated in Figures~\ref{fig:comparaison_Laplace_absolute_errors} and~\ref{fig:comparaison_Laplace_relative_errors}. Intuitively, this comes from the fact that these distributions are unimodal and have the same mean and coefficient of variation, see \cite[p.109]{forbes_statistical_2011} and~\cite[p.132]{forbes_statistical_2011}. By replacing the densities of log-normal distributions in~\eqref{eq:estimation_cemetery_lognormal} with the densities of Gamma distributions in view of~\eqref{eq:density_gamma_distribution}, we obtain the following estimator of $n_{\partial}^{(N)}$
\begin{equation}\label{eq:estimation_cemetery_gamma}
\begin{aligned}
\forall t\geq0:\hspace{2mm}\overline{n}_{\partial}^{(N,\alpha)}\left(t\right) &:= \frac{1}{n_d}\sum_{i = 1}^{n_d} h_{\ell^*\left(\alpha\right),\beta^*\left(\log\left(T_i^{(N)}\right),\alpha\right)}(t)\\
&= \frac{t^{\ell^{*}(\alpha)-1}}{\Gamma\left(\ell^{*}(\alpha)\right)}\frac{1}{n_d}\sum_{i = 1}^{n_d}\left[\frac{\ell^{*}(\alpha)e^{-\frac{\alpha^2}{2}}}{T_i}\right]^{\ell^{*}(\alpha)}\exp\left[-\frac{\ell^{*}(\alpha)e^{-\frac{\alpha^2}{2}}}{T_i}t\right].
\end{aligned}
\end{equation}
\begin{rem}
The restriction $\alpha\in\left(0,\log(2)^{\frac{1}{2}}\right)$ in~\eqref{eq:parameters_gamma_same_mean} is important because if we extend the definition of $\ell^*(\alpha)$ for $\alpha \geq \log(2)^{\frac{1}{2}}$, then we will have $\ell^{*}(\alpha)  \leq 1$ for these values. This is a problem because a Gamma distribution with parameters $\left(\ell,\beta\right)\in(0,1]\times\mathbb{R}_+^*$ is not unimodal (see~\cite[p.$109$]{forbes_statistical_2011}), which is important to be a smooth approximation of a Dirac measure.
\end{rem}

We mention that for all $a>0$ and $\alpha\in\left(0,\log(2)^{\frac{1}{2}}\right)$, this is also possible to use a Gamma distribution with parameters 
\begin{equation}\label{eq:other_possible_parameters_gamma_distribution}
\left(\ell^{**}(\alpha),\beta^{**}(a,\alpha)\right) := \left(\ell^{*}(\alpha),e^{-a}\ell^{*}(\alpha)\right)\text{ or }\left(\ell^{***}(\alpha),\beta^{***}(a,\alpha)\right) := \left(\ell^{*}(\alpha),e^{-a}\left(\ell^{*}(\alpha)-1\right)\right)
\end{equation}
to approximate the Dirac measures in $\frac{1}{n_d}\sum_{i = 1}^{n_d} \frac{1}{t\alpha}\delta_{T_{i}^{(N)}}$. 
Indeed, in view of~\cite[p.109 and p.132]{forbes_statistical_2011}, a Gamma distribution with these parameters has the same coefficient of variation as the one of a Log-normal distribution with parameters~$(a,\alpha)$. In addition, as the Gamma distribution with parameters $\left(\ell^{**}(\alpha),\beta^{**}(a,\alpha)\right)$ has a mean of $e^{a}$, and the one with parameters $\left(\ell^{***}(\alpha),\beta^{***}(a,\alpha)\right)$ has a mode of $e^{a}$, taking one of these two parameters for $a = \log(T_i^{(N)})$ and $i\in\llbracket1,n_d\rrbracket$ will give a distribution concentrated around $T_i^{(N)}$, see Figure~\ref{fig:comparaison_densities}. We prefer to use the parameters $\left(\ell^{*}(\alpha),\beta^{*}(a,\alpha)\right)$ because it seems that the Laplace transform of such a Gamma distribution is closer to the one of the log-normal distribution than when it has one of the parameters in~\eqref{eq:other_possible_parameters_gamma_distribution}, see Figures~\ref{fig:comparaison_Laplace_absolute_errors} and~\ref{fig:comparaison_Laplace_relative_errors}. This is important that the Laplace transform of the chosen Gamma distribution is close to the one of the log-normal distribution because our way to choose~$\alpha$ is based on this, see Section~\ref{subsubsect:choice_smoothing_parameter}.

Now, we replace $n_{\partial}^{(N)}$ with~$\overline{n}_{\partial}^{(N,\alpha)}$ in~\eqref{eq:estimator_gaver_stehfest_noise_free}. We obtain that the equivalent of~\eqref{eq:estimator_gaver_stehfest_noise_free} when random variables are observed is the following
\begin{equation}\label{eq:estimator_gaver_stehfest_random_variables}
	\begin{aligned}
		\forall x>0: \hspace{2mm}\overline{n}_0^{(N,K,\alpha)}(x) &= \frac{\log(2)}{x}\sum_{n = 1}^{2K} \frac{(-1)^{n+K}}{K!} \sum_{j = \left\lfloor\frac{n+1}{2}\right\rfloor}^{\min(n,K)} j^{K+1} \binom{K}{j}\binom{2j}{j}\binom{j}{n-j}\\ 
		&\times\left[1 + n\frac{\log(2)}{x} \frac{bm_1}{(bm_1)^2\frac{2N}{bm_2} +q_N\left(n\frac{\log(2)}{x}\right)}\right]\mathcal{L}\left(\overline{n}_{\partial}^{(N,\alpha)}\right)\left(q_N\left(n\frac{\log(2)}{x}\right)\right).
	\end{aligned}
\end{equation}
As before, the parameter $K$ must be as large as the numerical stability allows (see Appendix~\ref{sect:round_off_errors}). The choice of the smoothing parameter, for its part, is discussed in the next section.

\begin{figure}[!htb]
	\centering
	\begin{subfigure}[t]{0.45\textwidth}
		\centering
		\includegraphics[scale = 0.35]{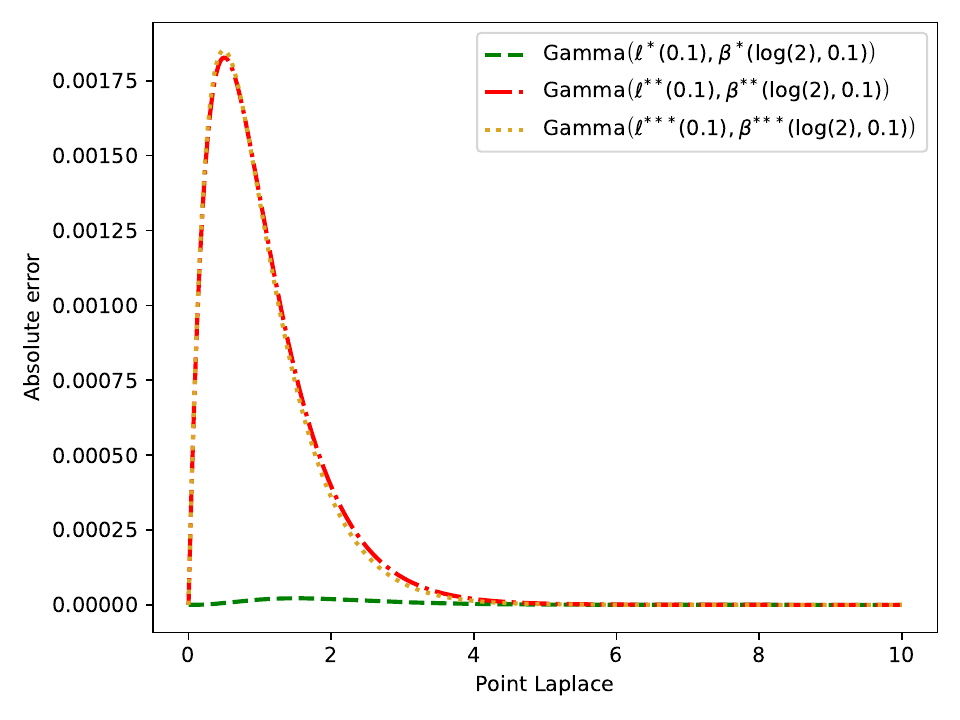}
		\caption{Absolute errors when $\alpha = 0.1$.}
	\end{subfigure}
	\hfill
	\begin{subfigure}[t]{0.45\textwidth}
		\centering
		\includegraphics[scale = 0.35]{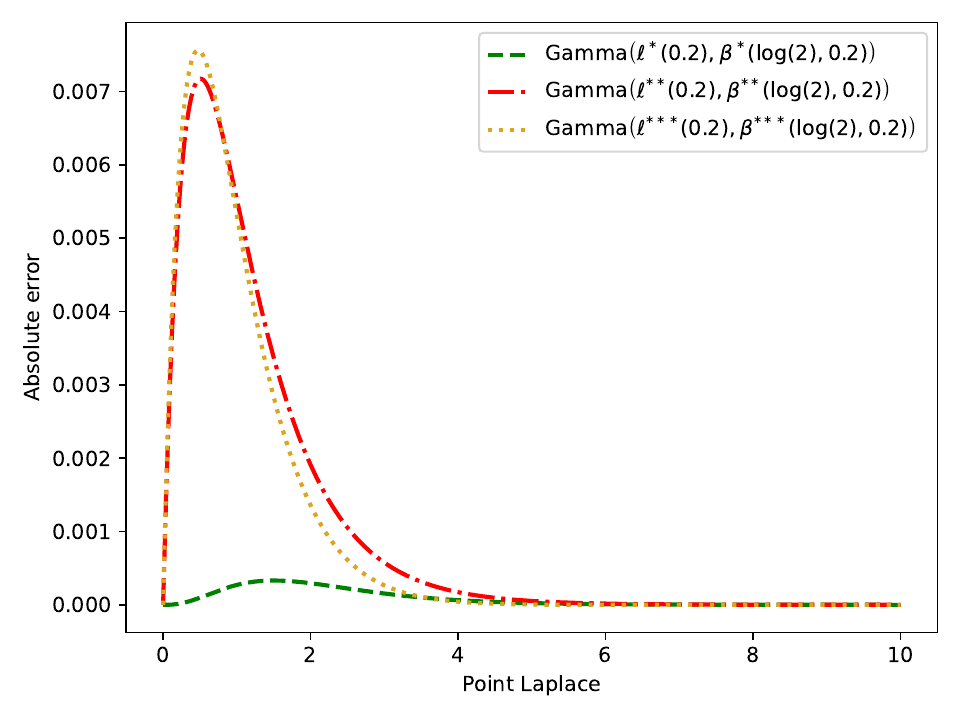}
		\caption{Absolute errors when $\alpha = 0.2$.}
	\end{subfigure}
	\caption{Comparison of the absolute errors between the Laplace transform of a Log-normal distribution with parameters $\left(\log(2),\alpha\right)$ and  the Laplace transform of Gamma distributions with parameters $\left(\ell^{*}(\alpha),\beta^{*}(\log(2),\alpha)\right)$, $\left(\ell^{**}(\alpha),\beta^{**}(\log(2),\alpha)\right)$ and $\left(\ell^{***}(\alpha),\beta^{***}(\log(2),\alpha)\right)$ (see~\eqref{eq:parameters_gamma_same_mean} and~\eqref{eq:other_possible_parameters_gamma_distribution}), when $\alpha\in\{0.1,0.2\}$. }\label{fig:comparaison_Laplace_absolute_errors}
\end{figure}
\begin{figure}[!htb]
	\centering
	\begin{subfigure}[t]{0.49\textwidth}
		\centering
		\includegraphics[scale = 0.35]{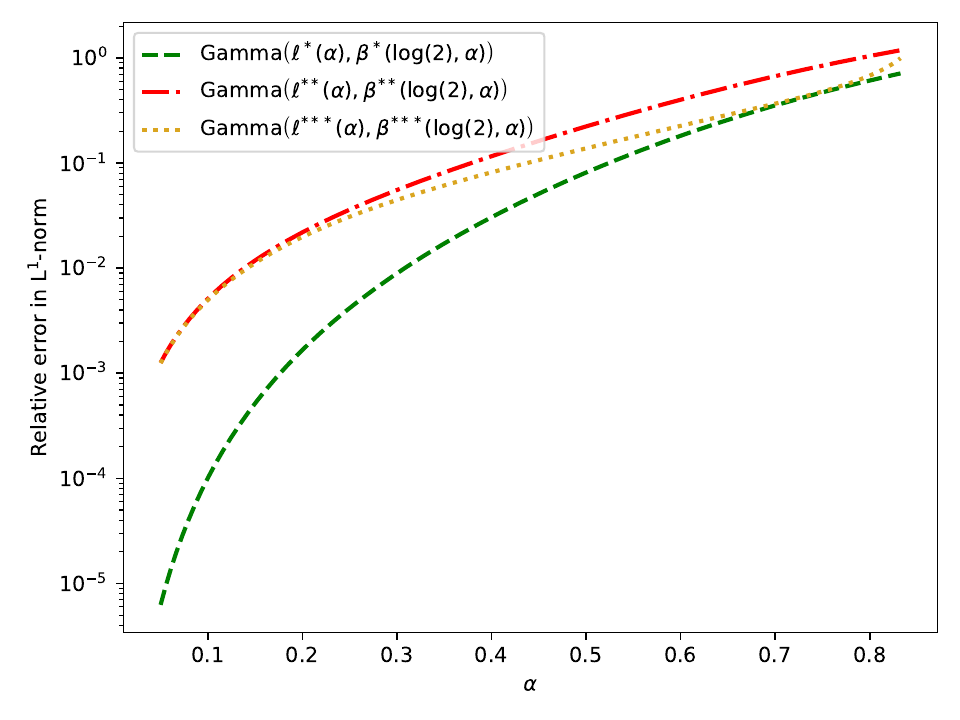}
		\caption{Relative error in $L^1$-norm versus $\alpha$.}
	\end{subfigure}
	\hfill
	\begin{subfigure}[t]{0.49\textwidth}
		\centering
		\includegraphics[scale = 0.35]{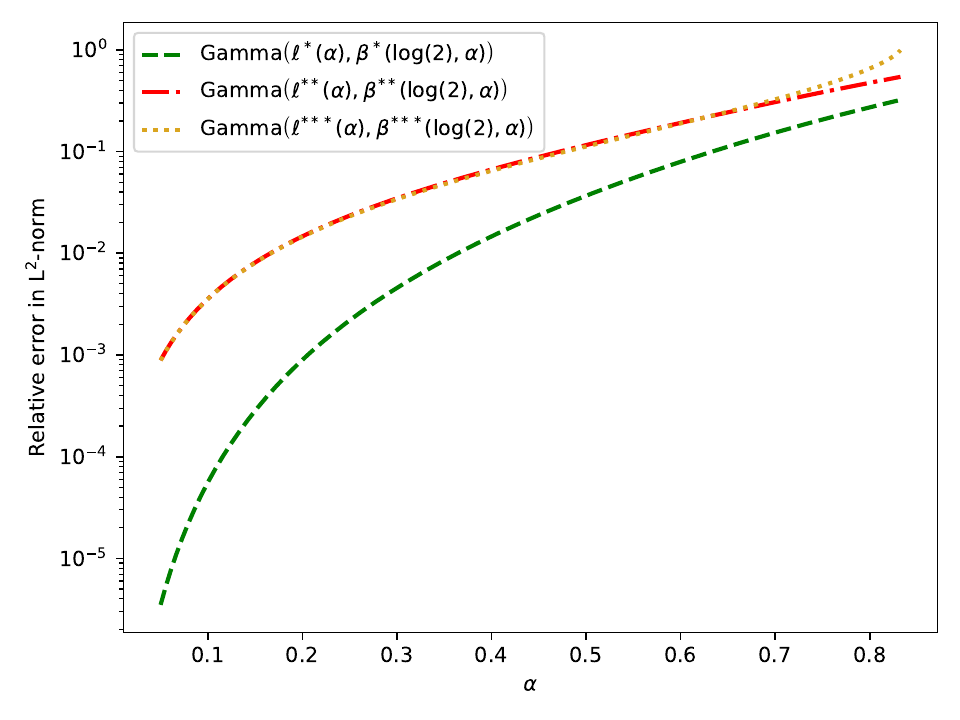}
		\caption{Relative error in $L^2$-norm versus $\alpha$.}
	\end{subfigure}
	\begin{subfigure}[t]{0.49\textwidth}
		\centering
		\includegraphics[scale = 0.35]{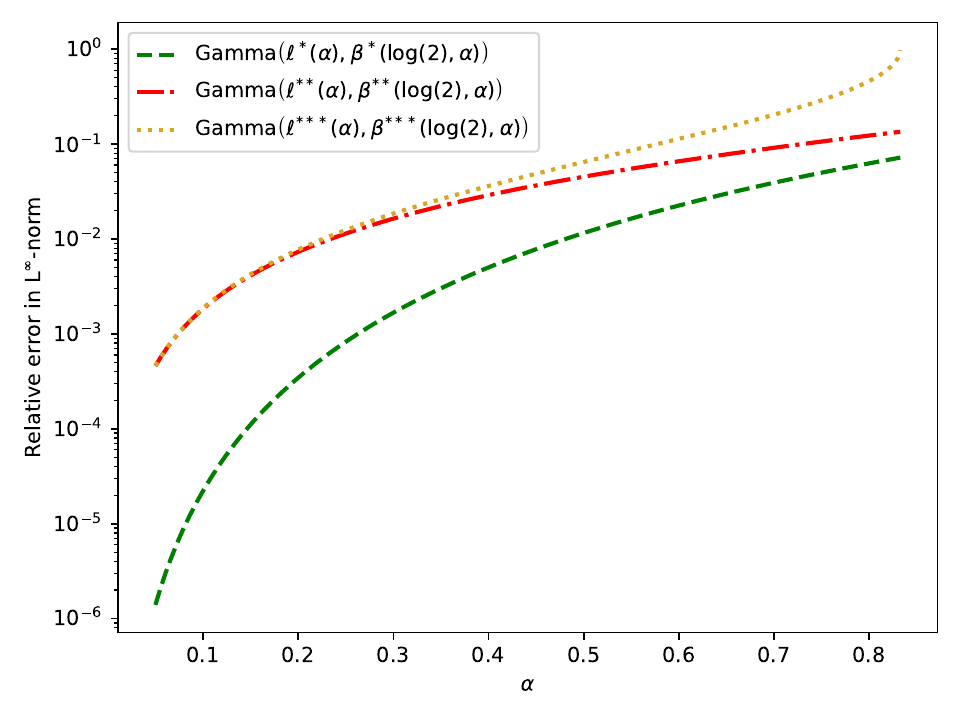}
		\caption{Relative error in $L^{\infty}$-norm versus $\alpha$.}
	\end{subfigure}
	\caption{Relative error between the Laplace transform of a Log-normal distribution with parameters $\left(\log(2),\alpha\right)$ and  the Laplace transform of Gamma distributions with parameters $\left(\ell^{*}(\alpha),\beta^{*}(\log(2),\alpha)\right)$, $\left(\ell^{**}(\alpha),\beta^{**}(\log(2),\alpha)\right)$ and $\left(\ell^{***}(\alpha),\beta^{***}(\log(2),\alpha)\right)$ (see~\eqref{eq:parameters_gamma_same_mean} and~\eqref{eq:other_possible_parameters_gamma_distribution}), for $\alpha\in\left(0,\log(2)^{\frac{1}{2}}\right)$.}\label{fig:comparaison_Laplace_relative_errors}
\end{figure}
\begin{figure}[!htb]
	\centering
	\begin{subfigure}[t]{0.485\textwidth}
		\centering
		\includegraphics[scale = 0.35]{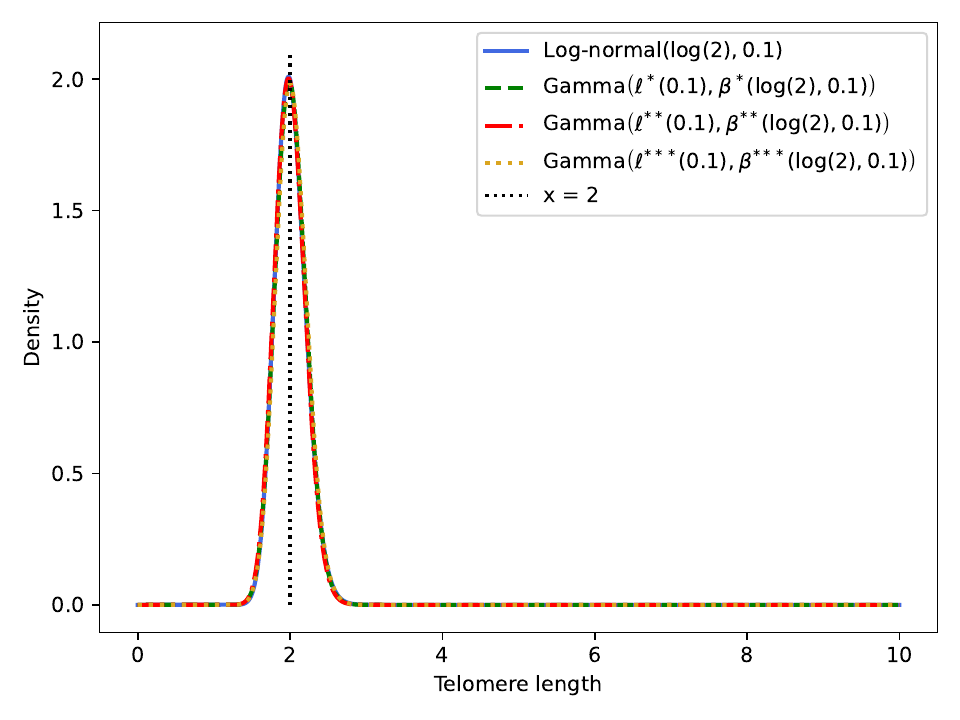}
		\caption{Densities when $\alpha = 0.1$.}
	\end{subfigure}
	\hfill
	\begin{subfigure}[t]{0.485\textwidth}
		\centering
		\includegraphics[scale = 0.35]{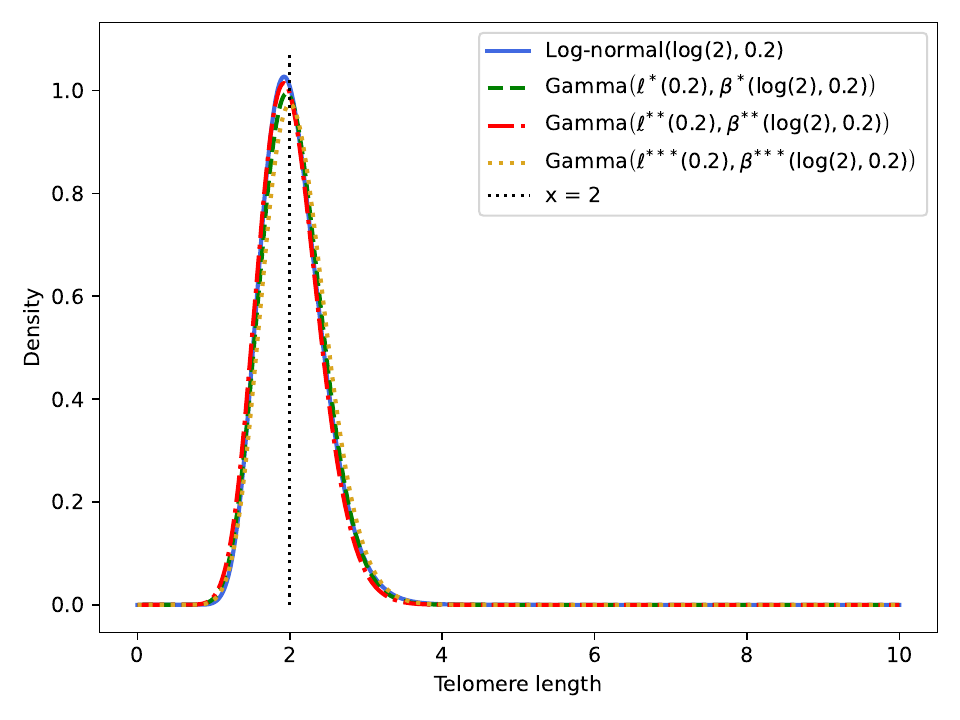}
		\caption{Densities when $\alpha = 0.2$.}
	\end{subfigure}
	\caption{Comparison of the density of a Log-normal distribution with parameters $\left(\log(2),\alpha\right)$ with the densities of Gamma distributions with parameters $\left(\ell^{*}(\alpha),\beta^{*}(\log(2),\alpha)\right)$, $\left(\ell^{**}(\alpha),\beta^{**}(\log(2),\alpha)\right)$ and $\left(\ell^{***}(\alpha),\beta^{***}(\log(2),\alpha)\right)$ (see~\eqref{eq:parameters_gamma_same_mean} and~\eqref{eq:other_possible_parameters_gamma_distribution}), for $\alpha\in\{0.1,0.2\}$.}\label{fig:comparaison_densities}
\end{figure}

\subsubsection{Choice of the smoothing parameter}\label{subsubsect:choice_smoothing_parameter}

To choose the smoothing parameter $\alpha$, we would like to take a parameter that minimises the error between~$\mathcal{L}\left(\overline{n}_{\partial}^{(N)}\right)$ and $\mathcal{L}\left(n_{\partial}^{(N)}\right)$. In view of the following inequality
$$
\begin{aligned}
\forall x\in\mathbb{R}_+:\hspace{2mm}\left|\mathcal{L}\left(\overline{n}_{\partial}^{(N)}\right)(x) - \mathcal{L}\left(n_{\partial}^{(N)}\right)(x)\right| &\leq \int_0^{\infty} \left|\overline{n}_{\partial}^{(N)}(s) - n_{\partial}^{(N)}(s)\right|e^{-xs} \dd s \\
&\leq \left|\left|\overline{n}_{\partial}^{(N)} - n_{\partial}^{(N)}\right|\right|_{L^1\left(\mathbb{R}_+\right)},
\end{aligned}
$$
choosing a parameter which minimises the error in $L^1$-norm between $\overline{n}_{\partial}^{(N)}$ and $n_{\partial}^{(N)}$ seems a good idea to control the error between $\mathcal{L}\left(\overline{n}_{\partial}^{(N)}\right)$ and $\mathcal{L}\left(n_{\partial}^{(N)}\right)$. However, for the moment, we are only able to control $\left|\left|\text{Id}\left[\overline{n}_{\partial}^{(N)}-n_{\partial}^{(N)}\right]\right|\right|_{L^\infty\left(\mathbb{R}_+\right)}$ (by proceeding as in the proof of~\cite[Proposition~$5.2$]{olaye_transport_2026}), which is not sufficient to obtain a precise upper bound on $\left|\left|\overline{n}_{\partial}^{(N)}-n_{\partial}^{(N)}\right|\right|_{L^1\left(\mathbb{R}_+\right)}$. In addition, even if we are able to obtain a result similar to~\cite[Proposition~$5.2$]{olaye_transport_2026}  which provides a parameter allowing us to control~$\left|\left|\overline{n}_{\partial}^{(N)}-n_{\partial}^{(N)}\right|\right|_{L^1\left(\mathbb{R}_+\right)}$, we have a problem because this parameter will certainly depend on $n_{\partial}^{(N)}$, that we do not know in practice when we want to proceed to the estimation (we only have random variables distributed according to it). We thus need to choose~$\alpha$ in another~way.

The estimator of $n_{\partial}^{(N)}$ defined in~\eqref{eq:estimation_cemetery_gamma} has been chosen in order to have a Laplace transform very close to the one of the one defined in~\eqref{eq:estimation_cemetery_lognormal}, see Figures~\ref{fig:comparaison_Laplace_absolute_errors} and~\ref{fig:comparaison_Laplace_relative_errors}, which corresponds to the log-transform density estimator~\cite{charpentier_2015,nguyen_positive_2019}. Thus, we have the intuition that choosing $\alpha >0$ such that the log-transform estimator $\widetilde{n}_{\partial}^{(N,\alpha)}$ in~\eqref{eq:estimation_cemetery_lognormal} estimates well $n_{\partial}^{(N)}$ will allow us to have a good estimation of $\mathcal{L}\left(n_\partial\right)$ from $\mathcal{L}\left(\overline{n}_\partial^{(N,\alpha)}\right)$. This is what we do in Section~\ref{subsubsect:estimation_results_random_variables} to choose~$\alpha$. In each estimation, we choose one of the methods available in the package R \textit{logKDE}~\cite{nguyen_positive_2019} to estimate the smoothing parameter for the log-transform kernel density estimator $\widetilde{n}_{\partial}^{(N,\alpha)}$. The method of this package we take depends on the number of observed data, and is specified at each estimation. Then, we use the smoothing estimator estimated to compute $\overline{n}_{\partial}^{(N,\alpha)}$ and after~$\overline{n}_0^{(N,K,\alpha)}$, see~\eqref{eq:estimator_gaver_stehfest_random_variables}. The estimation results presented in Section~\ref{subsubsect:estimation_results_random_variables} are quite satisfactory, so this way to choose $\alpha$ seems accurate. However, in a future work, a more rigorous study must be done to develop a more robust method to choose $\alpha$.

\subsubsection{Estimation results when random variables are observed}\label{subsubsect:estimation_results_random_variables}
	
We now check that the estimator defined in~\eqref{eq:estimator_gaver_stehfest_random_variables} has good estimation results for the choice of smoothing parameter explained in Section~\ref{subsubsect:choice_smoothing_parameter}. To do so, we fix for model parameters~\hbox{$N = 40$}, $b=1$ and \hbox{$g = 1_{[0,1]}$}, as in Section~\ref{subsect:estimation_noise_free_data}. Then, for $n_0\in\left\{h_{9,12},h_{16,16},h_{25,30},h_{49,50}\right\}$, we simulate a sequence of random variables $\left(T_i\right)_{1\leq i \leq n_d}$ for $n_d = 300$ by simulating a probabilistic telomere shortening model, see~\cite[Section~$3$.C]{olaye_thesis_2025}. Thereafter, we use these random variables to compute~$\mathcal{L}\left(\overline{n}^{(N,\alpha)}_{\partial}\right)$ for $\alpha$ chosen with the method \textit{SJ-ste} of the package \textit{logKDE} of~R, which provides a good bandwidth selection for a moderate or large number of data (see~\cite{sheather_1991} for an explanation of the method). Finally, we estimate $n_0$ with the estimator~$\overline{n}_{0}^{(N,K,\alpha)}$ defined in~\eqref{eq:estimator_gaver_stehfest_random_variables}. The parameter $K$ has been chosen manually, by taking the largest possible value of $K$ for which we do not observe round-off errors. We refer to Appendix~\ref{sect:round_off_errors} to see why these round-off errors can lead to bad estimation, and why the value of $K$ chosen in this section is not of the same order as in Sections~\ref{subsubsect:estimation_results_noise_free} and~\ref{subsubsect:difficulties_small_variability}. We plot in Figure~\ref{fig:estimation_gaver_stehfest_random_variables} the estimated curves (in red), and compare them with their associated theoretical distribution (in black). We also plot on these figures the estimations obtained with the estimator constructed in~\cite[Section~$5.2.1$]{olaye_transport_2026} (in blue), defined for all $\alpha \in \left(0,\log(2)^{\frac{1}{2}}\right)$ and $x>0$ as
\begin{equation}\label{eq:estimator_n0_previous_article}
\overline{n}_{0}^{(\text{old},N,\alpha)}(x) := \frac{1}{bm_1}\widetilde{n}_{\partial}^{(N,\alpha)}\left(\frac{x}{bm_1}\right) = \frac{1}{n_d}\sum_{i = 1}^{n_d}\frac{1}{x}\rho_{\alpha}\left(\log\left(\frac{x}{bm_1T_i}\right)\right).
\end{equation}
The smoothing parameter in the above is chosen with the same method as the one for $\overline{n}_{0}^{(N,K,\alpha)}$. We observe in Figure~\ref{fig:estimation_gaver_stehfest_random_variables} that the estimations done with~$\overline{n}_{0}^{(N,K,\alpha)}$ are better than the ones with~$\overline{n}_{0}^{(\text{old},N,\alpha)}$. Specifically, the diffusivity of the curve of $n_0$, which is not well-estimated with $\overline{n}_{0}^{(\text{old},N,\alpha)}$, is now better estimated with $\overline{n}_{0}^{(N,K,\alpha)}$. We have thus improved our estimation method. We however mention that despite this improvement, the estimations with $\overline{n}_{0}^{(N,K,\alpha)}$ are clearly not as satisfying as the estimations of Figure~\ref{fig:estimation_gaver_stehfest_noisefree}, which were perfect. It thus seems that the estimation error is more impacted by the noise on $n_{\partial}^{(N)}$, than by error between $n_{\partial}^{(N)}$ and~$u_{\partial}^{(N)}$. 

The estimation results presented in Figure~\ref{fig:estimation_gaver_stehfest_random_variables} are done with $n_d = 300$, which corresponds to a moderate number of data. In practice, we do not have necessarily a moderate number of data, so this is interesting to study how our estimator acts when we have a small or high number of data. To do this, we apply the same protocol as in the previous paragraph in the case where~$n_d \in\{30,3000\}$ , and for only $n_0\in\{h_{9,12},h_{25,30},h_{49,50}\}$. The only difference is that we use the method \textit{nrd} of the package \textit{logKDE} of~R to choose $\alpha$ when $n_d = 30$ (corresponds to the method presented in~\cite[Eq.~$\left(3.28\right)$]{silverman_density_2018}), instead of the method \textit{SJ-ste}. The reason is that the \textit{nrd} method is based on simple statistics, so is more robust when the number of observations is limited. We plot in Figure~\ref{fig:estimation_gaver_stehfest_variation_number_simulations} the estimations we have obtained. We observe in Figures~\ref{fig:estimation_gaver_large_number_data_second} and~\ref{fig:estimation_gaver_large_number_data_third} that when $n_d = 30$, the estimations obtained with~$\overline{n}_{0}^{(N,K,\alpha)}$ are slightly better than the ones obtained with $\overline{n}_{0}^{(\text{old},N,\alpha)}$. The improvement is far from exceptional, but this is expected since we have a very small number of data when $n_d = 30$. Hence, our new estimator remains useful for a small number of data, but has no incredible results. When $n_d = 3000$,  which corresponds to the estimations presented in the second line of Figure~\ref{fig:estimation_gaver_stehfest_variation_number_simulations}, the estimations done with $\overline{n}_{0}^{(N,K,\alpha)}$ are better than in Figure~\ref{fig:estimation_gaver_stehfest_random_variables}. Notably, the red and blue curves are almost superposed in Figures~\ref{fig:estimation_gaver_large_number_data_first} and~\ref{fig:estimation_gaver_large_number_data_second}. However, one can observe in Figure~\ref{fig:estimation_gaver_large_number_data_third} that even with a large number of data, the estimations obtained with $\overline{n}_0^{(N,K,\alpha)}$ are not totally perfect, in comparison to the one in Figure~\ref{fig:estimation_gaver_stehfest_noisefree}. Our estimation method can therefore still be improved, despite the good estimation results we have obtained for moderate or large amounts of data. \needspace{2\baselineskip}

\begin{figure}[!htb]
	\centering
	\begin{subfigure}[t]{0.485\textwidth}
		\centering
		\includegraphics[scale = 0.34]{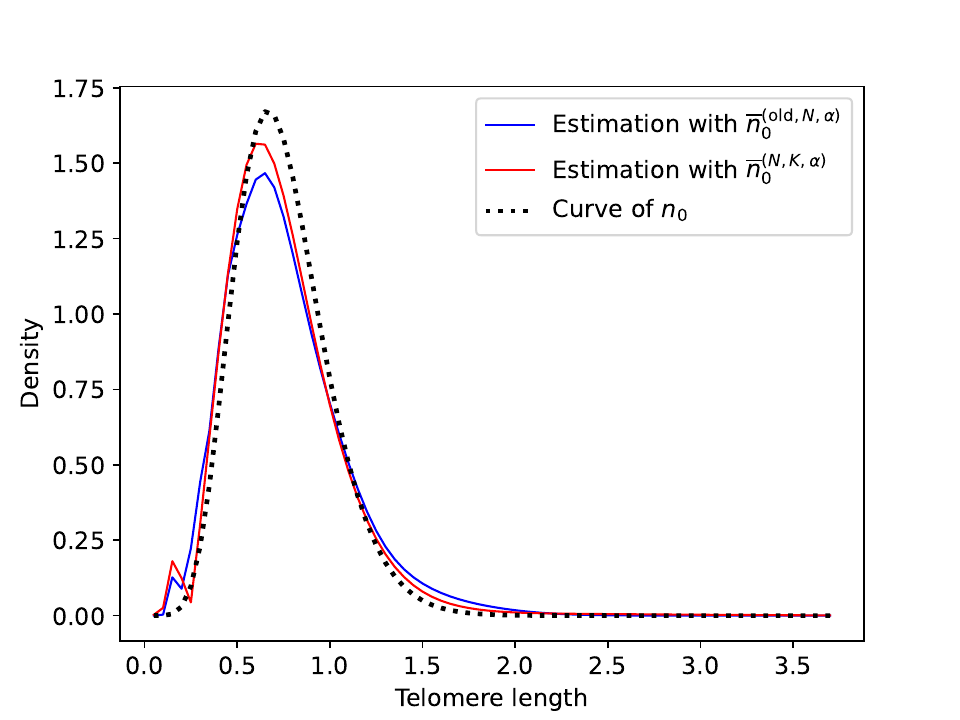}
		\caption{Estimation for $n_0 = h_{9,12}$ and $K = 8$.}\label{fig:estimation_gaver_stehfest_random_variables_intermediate_first}
	\end{subfigure}
	\hfill
	\begin{subfigure}[t]{0.485\textwidth}
		\centering
		\includegraphics[scale = 0.34]{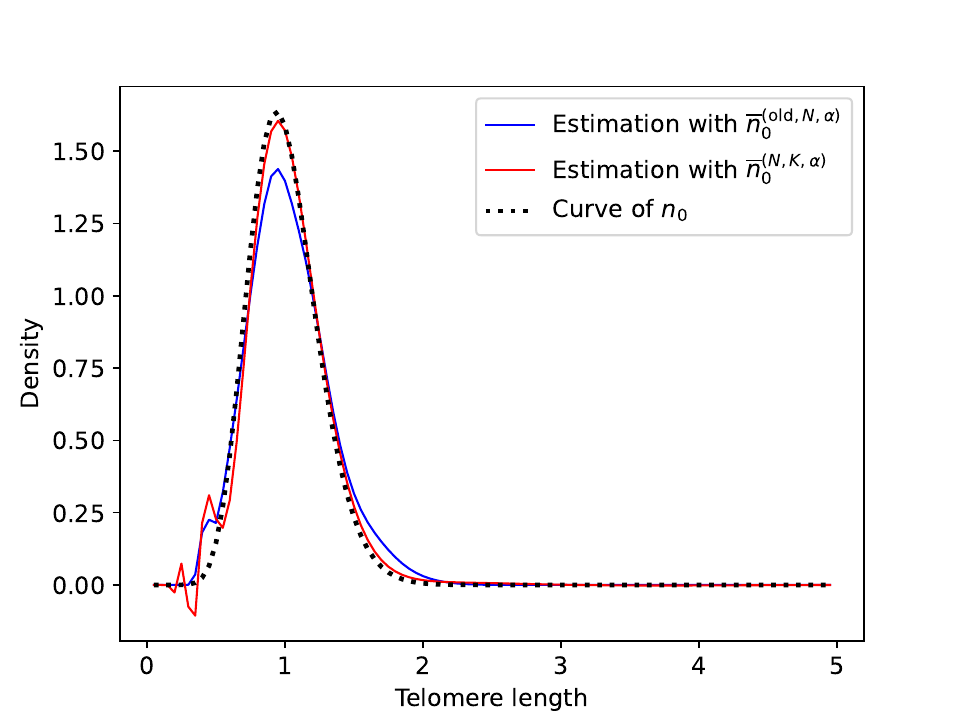}
		\caption{Estimation for $n_0 = h_{16,16}$ and $K = 11$.}\label{fig:estimation_gaver_stehfest_random_variables_intermediate_second}
	\end{subfigure}
	\begin{subfigure}[t]{0.485\textwidth}
		\centering
		\includegraphics[scale = 0.34]{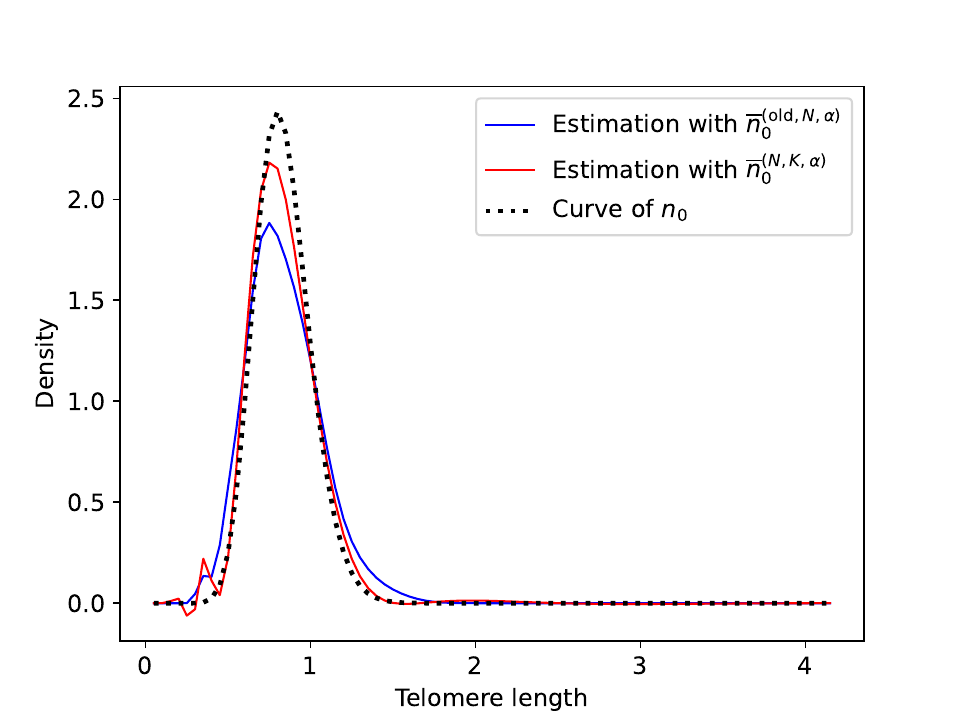}
		\caption{Estimation for $n_0 = h_{25,30}$ and $K = 11$.}\label{fig:estimation_gaver_stehfest_random_variables_intermediate_third}
	\end{subfigure}
	\hfill
	\begin{subfigure}[t]{0.485\textwidth}
		\centering
		\includegraphics[scale = 0.34]{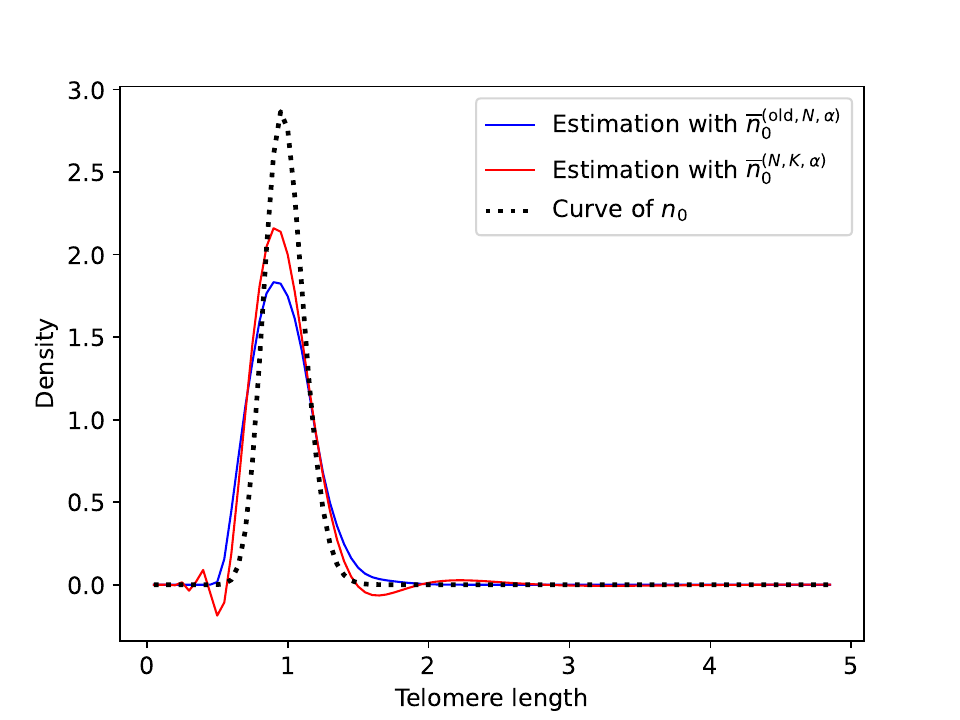}
		\caption{Estimation for $n_0 = h_{49,50}$ and $K = 11$.}\label{fig:estimation_gaver_stehfest_random_variables_intermediate_fourth}
	\end{subfigure}
	\caption{Estimation results with the estimator $\overline{n}_{0}^{(N,K,\alpha)}$ defined in~\eqref{eq:estimator_gaver_stehfest_random_variables} when $b = 1$, $g = 1_{[0,1]}$, $N = 40$, $n_d = 300$, for $\alpha$ chosen with the method \textit{SJ-ste} of the package \textit{logKDE} of R, and for Gamma initial distributions. Comparison with the results obtained with the estimator~$\overline{n}_0^{(\text{old},N,\alpha)}$ defined in~\eqref{eq:estimator_n0_previous_article} and constructed in~\cite{olaye_transport_2026}, for the same value of $\alpha$. \textit{For plotting the curves, the numerical precision of the computations was set to $200$ digits.}}\label{fig:estimation_gaver_stehfest_random_variables}
\end{figure}

We finally conclude this numerical study by testing our estimation method on other initial distributions than Gamma distributions. To do this, we apply the same protocol as the one used for the estimations in Figure~\ref{fig:estimation_gaver_stehfest_random_variables}. The only thing that changes is that we consider the following three initial distributions to do our estimation, for all $x\geq0$:
\begin{equation}\label{eq:definition_other_initial_distributions}
\begin{aligned}
&n_{0,1}(x) = \frac{11}{2}\left(\frac{x}{2}\right)^{10}\exp\left[-\left(\frac{x}{2}\right)^{11}\right], \hspace{7.5mm}n_{0,2}(x) = 2\frac{3^6}{5!}x^{11}\exp\left(-\frac{3x^2}{2}\right) , \\
&\hspace{35mm}n_{0,3}(x) = \frac{1}{2}\left[h_{8,8}(x) + h_{11,3}(x)\right].
\end{aligned}
\end{equation}
In the above, $n_{0,1}$ corresponds to the density of a Weibull distribution with parameters $\left(11,2\right)$, and~$n_{0,2}$ corresponds to the density of a Nakagami distribution with parameters $\left(6,2\right)$. These distributions are interesting to consider since they verify the assumptions of Theorem~\ref{te:model_approximation} and Proposition~\ref{prop:link_laplace_transforms}, and since they are unimodal, which is biologically relevant. In addition, $n_{0,3}$ corresponds to the density of a mixture of Gamma distribution with parameters $(8,8)$ with another Gamma distribution with parameters~$(11,3)$, so is the density of a bimodal distribution. This type of distribution is not biologically relevant, but interesting to consider to know if our method can be adapted to other contexts. We plot in Figure~\ref{fig:estimation_gaver_stehfest_other_distributions} the estimation results we have obtained with the distributions defined in~\eqref{eq:definition_other_initial_distributions}. We observe in Figures~\ref{fig:estimation_gaver_stehfest_other_distributions_first} and~\ref{fig:estimation_gaver_stehfest_other_distributions_second} that the estimations obtained with $\overline{n}_0^{(N,K,\alpha)}$ are better than the ones obtained with $\overline{n}_0^{(\text{old},N,\alpha)}$, since the diffusivity of the curve is now better captured. We also observe in Figure~\ref{fig:estimation_gaver_stehfest_other_distributions_third} that the red curve follows the black curve more precisely than the blue curve, particularly in the modes of the distribution. We thus conclude that the estimator we have constructed in this work improves the estimations done with~the estimator constructed in~\cite{olaye_transport_2026}, even in a general context. 

\begin{figure}[!htb]
	\centering
	\begin{subfigure}[t]{0.325\textwidth}
		\centering
		\includegraphics[scale = 0.325]{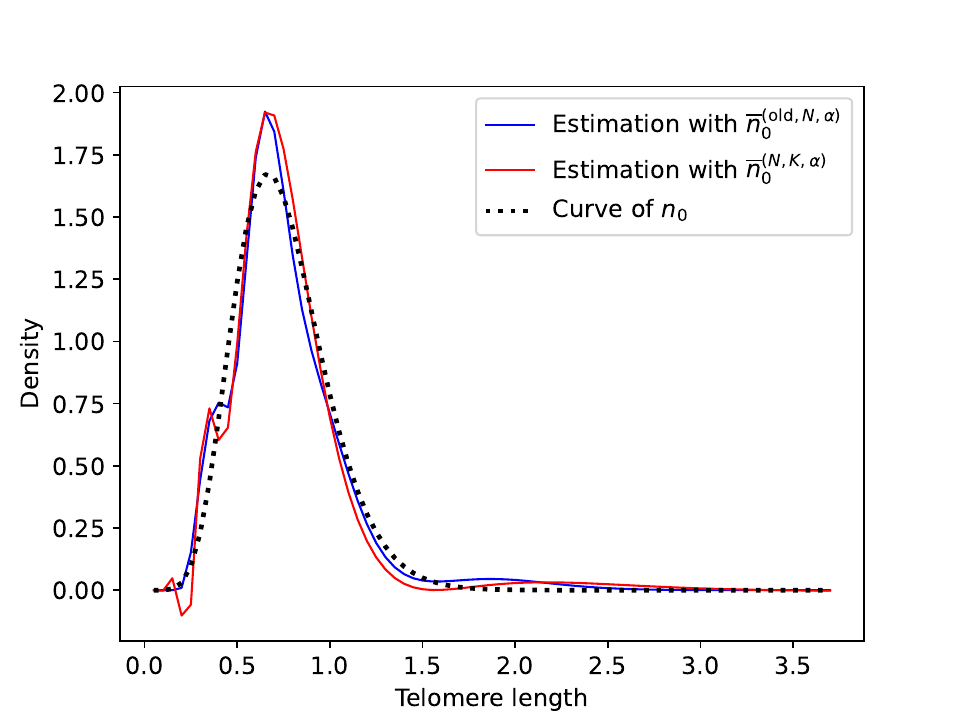}
		\caption{Estimation for $n_d = 30$, $n_0 = h_{9,12}$ and $K = 8$.}\label{fig:estimation_gaver_small_number_data_first}
	\end{subfigure}
	\hfill
	\begin{subfigure}[t]{0.325\textwidth}
		\centering
		\includegraphics[scale = 0.325]{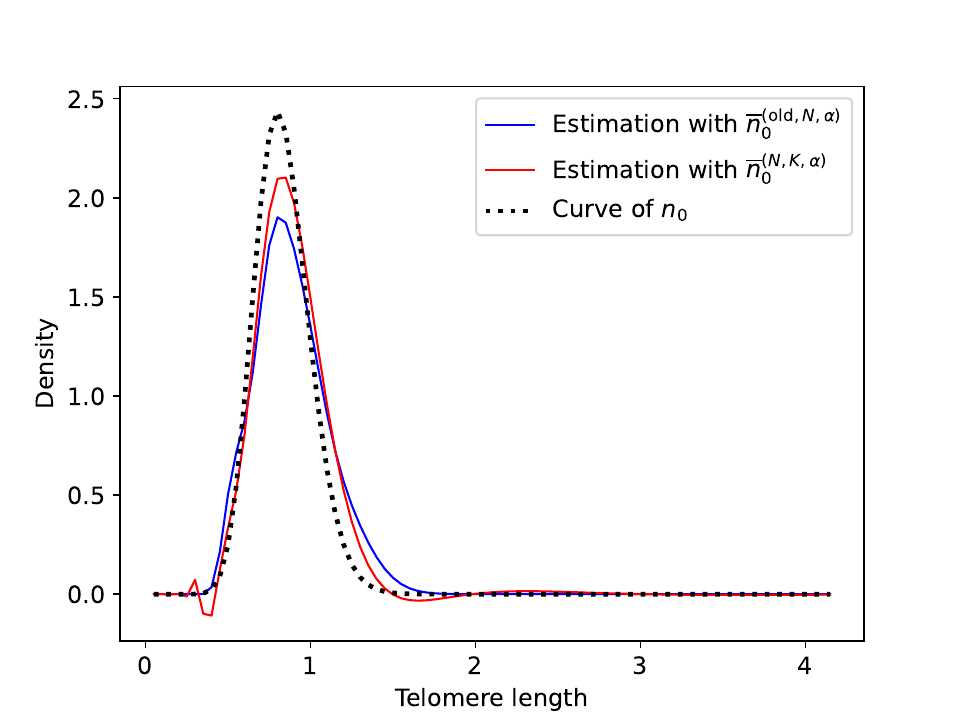}
		\caption{Estimation for $n_d = 30$, $n_0 = h_{25,30}$ and $K = 10$.}\label{fig:estimation_gaver_small_number_data_second}
	\end{subfigure}
	\hfill
	\begin{subfigure}[t]{0.325\textwidth}
		\centering
		\includegraphics[scale = 0.325]{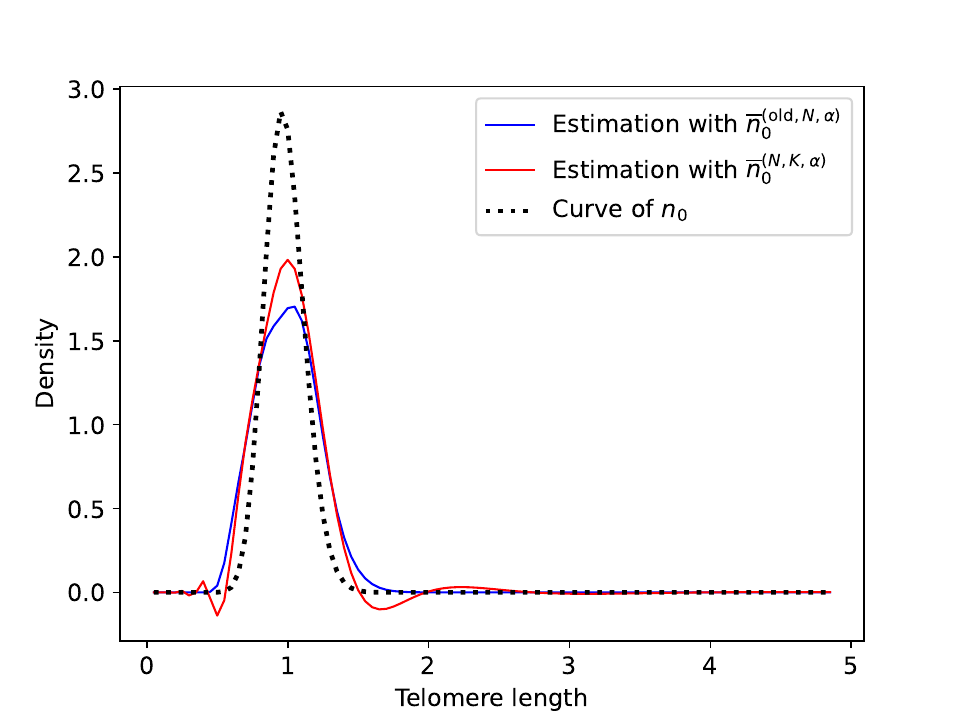}
		\caption{Estimation for $n_d = 30$, $n_0 = h_{49,50}$ and $K = 12$.}\label{fig:estimation_gaver_small_number_data_third}
	\end{subfigure}
	\begin{subfigure}[t]{0.325\textwidth}
		\centering
		\includegraphics[scale = 0.325]{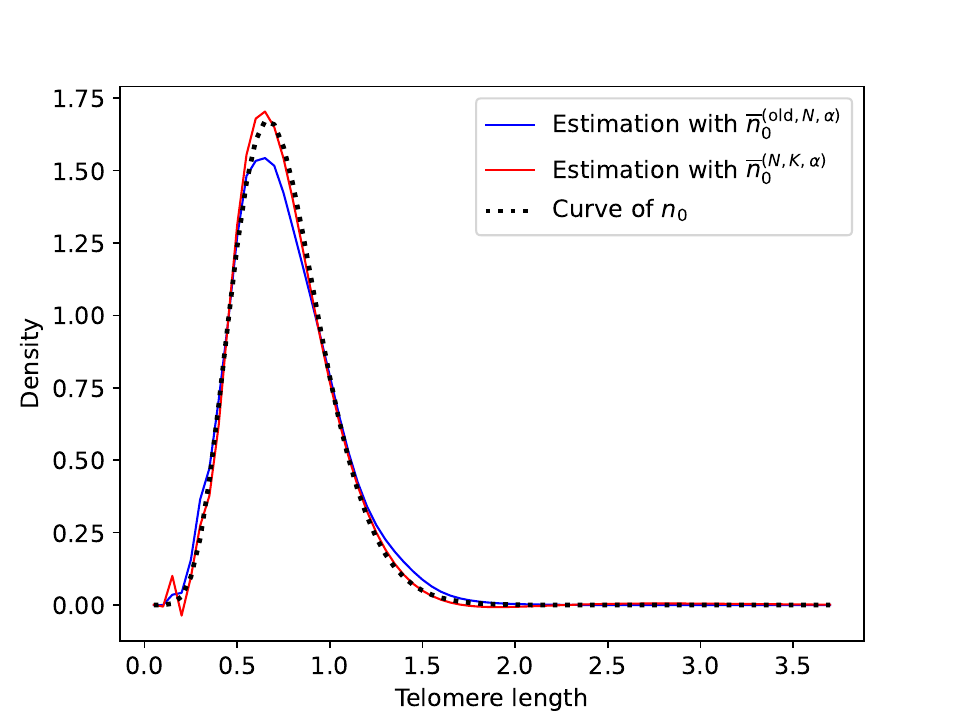}
		\caption{Estimation for $n_d = 3000$, $n_0 = h_{9,12}$ and $K = 8$.}\label{fig:estimation_gaver_large_number_data_first}
	\end{subfigure}
	\hfill
	\begin{subfigure}[t]{0.325\textwidth}
		\centering
		\includegraphics[scale = 0.325]{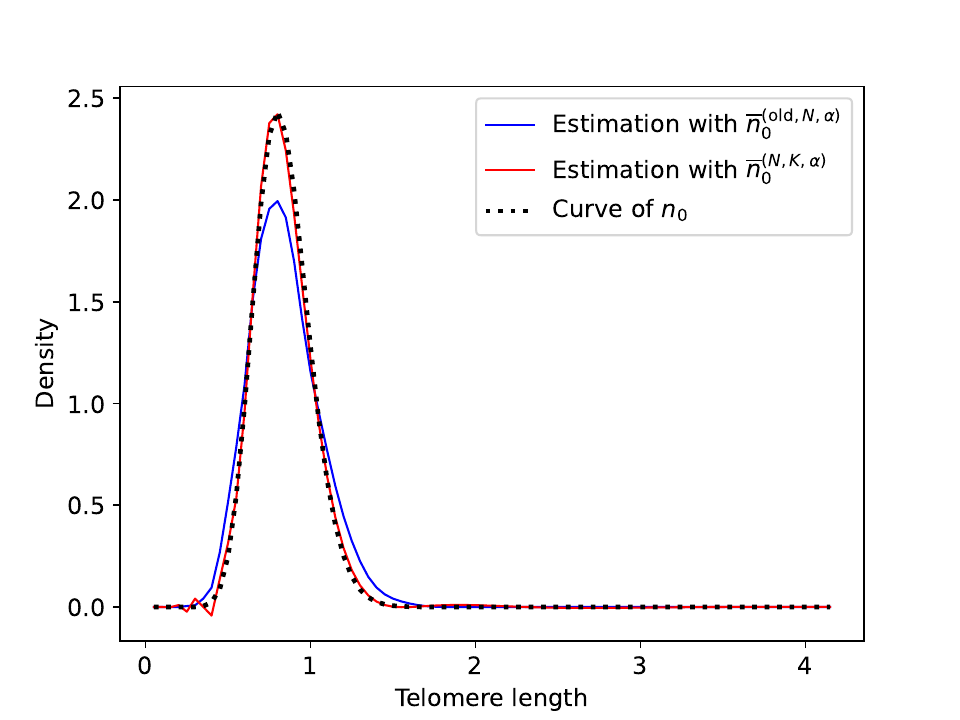}
		\caption{Estimation for $n_d = 3000$, $n_0 = h_{25,30}$ and $K = 12$.}\label{fig:estimation_gaver_large_number_data_second}
	\end{subfigure}
	\hfill
	\begin{subfigure}[t]{0.325\textwidth}
		\centering
		\includegraphics[scale = 0.325]{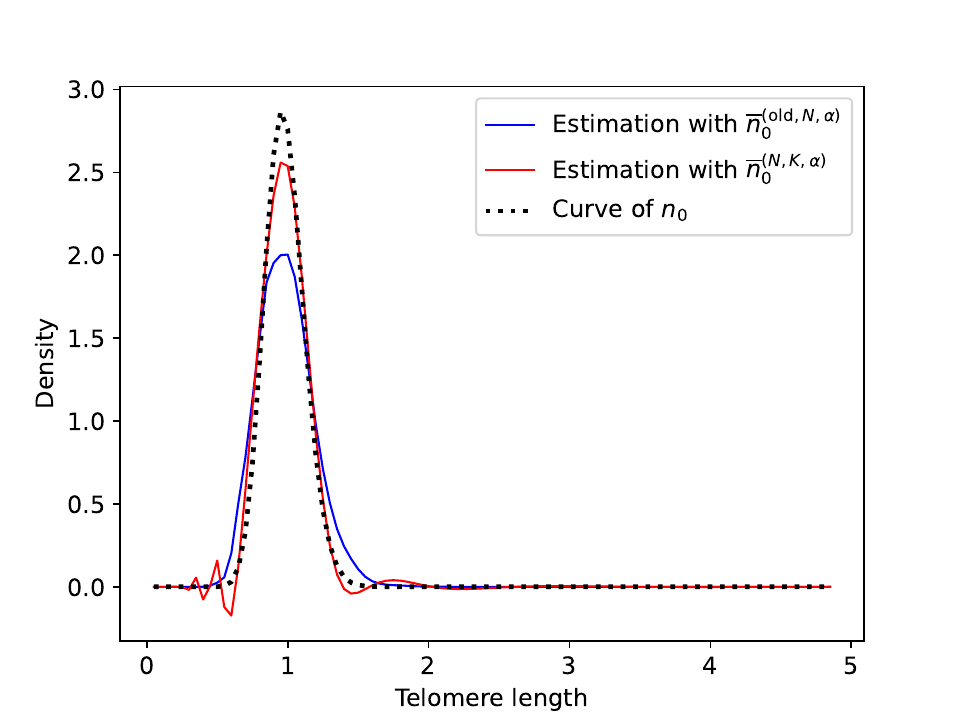}
		\caption{Estimation for $n_d = 3000$, $n_0 = h_{49,50}$ and $K = 16$.}\label{fig:estimation_gaver_large_number_data_third}
	\end{subfigure}
	\caption{Estimation results with the estimator $\overline{n}_0^{(N,K,\alpha)}$ defined in~\eqref{eq:estimator_gaver_stehfest_random_variables} when $b = 1$, $g = 1_{[0,1]}$, $N = 40$, $n_d\in\{30,3000\}$, for $\alpha$ chosen with the method \textit{nrd} (at the top) or \textit{SJ-ste} (at the bottom) of the package \textit{logKDE} of R, and for Gamma initial distributions. Comparison with the results obtained with~$\overline{n}_0^{(\text{old},N,\alpha)}$ defined in~\eqref{eq:estimator_n0_previous_article} and constructed in~\cite{olaye_transport_2026}, for the same value of $\alpha$. \textit{For plotting the curves, the numerical precision of the computations was set to $200$ digits.}}\label{fig:estimation_gaver_stehfest_variation_number_simulations}
\end{figure}
\begin{figure}[!htb]
	\centering
	\begin{subfigure}[t]{0.325\textwidth}
		\centering
		\includegraphics[scale = 0.325]{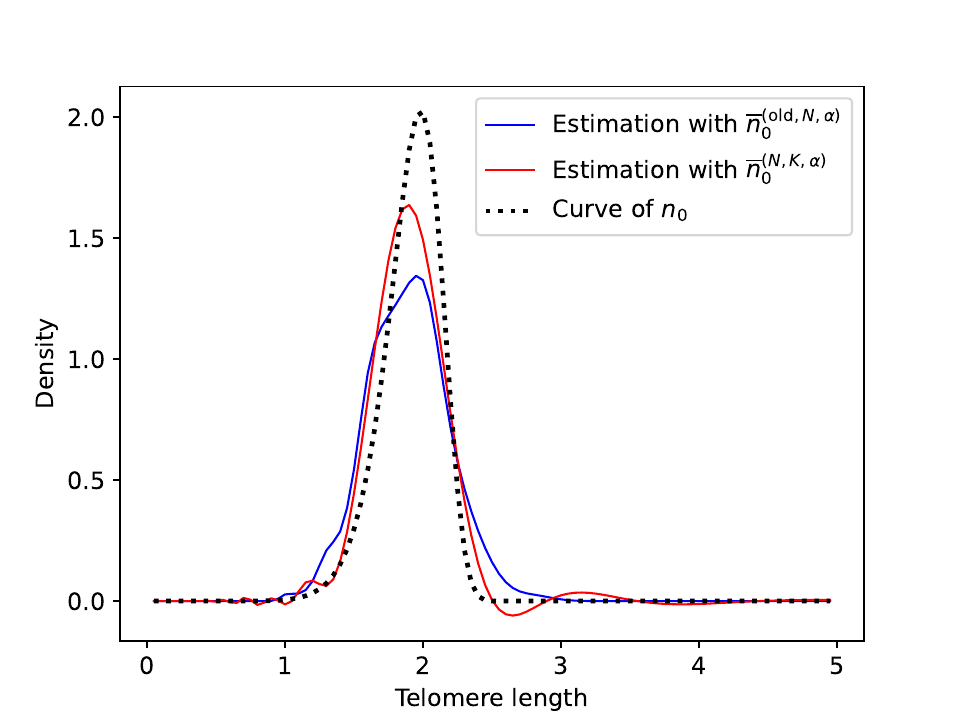}
		\caption{Estimation for $n_0 = n_{0,1}$ and $K = 18$.}\label{fig:estimation_gaver_stehfest_other_distributions_first}
	\end{subfigure}
	\hfill
	\begin{subfigure}[t]{0.325\textwidth}
		\centering
		\includegraphics[scale = 0.325]{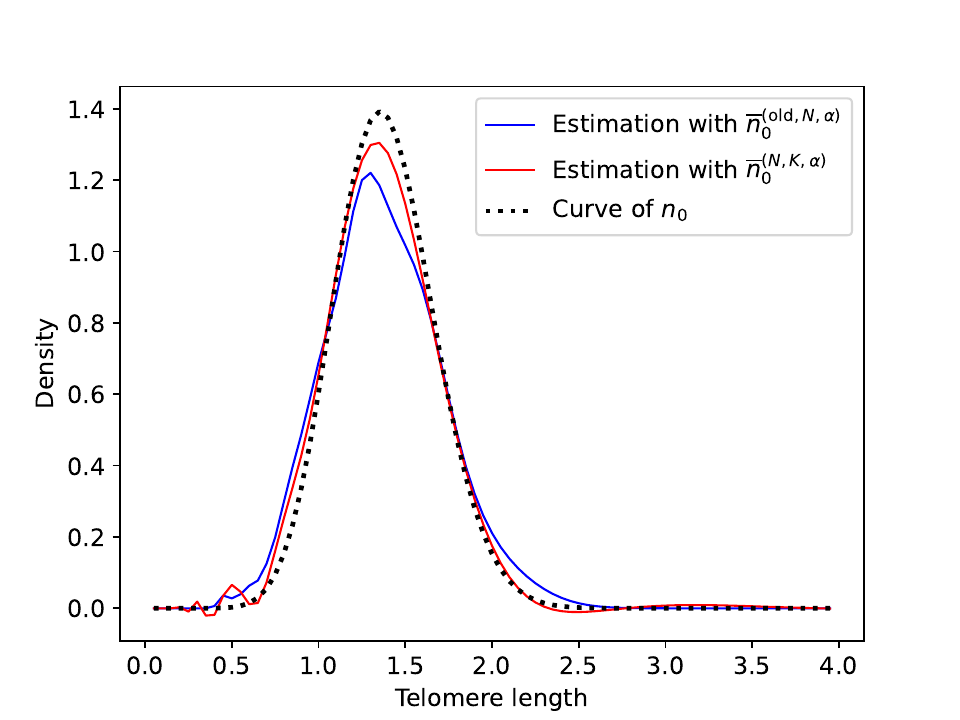}
		\caption{Estimation for $n_0 = n_{0,2}$ and $K = 12$.}\label{fig:estimation_gaver_stehfest_other_distributions_second}
	\end{subfigure}
	\hfill
	\begin{subfigure}[t]{0.325\textwidth}
		\centering
		\includegraphics[scale = 0.325]{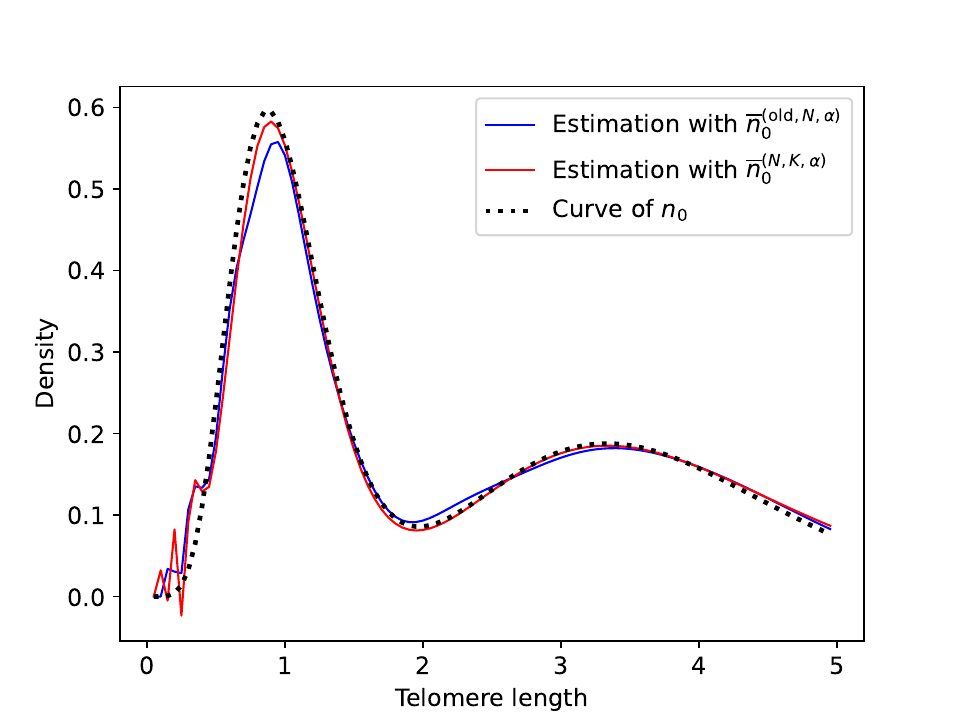}
		\caption{Estimation for $n_0 = n_{0,3}$ and $K = 8$.}\label{fig:estimation_gaver_stehfest_other_distributions_third}
	\end{subfigure}
	\caption{Estimation results with the estimator $\overline{n}_0^{(N,K,\alpha)}$ defined in~\eqref{eq:estimator_gaver_stehfest_random_variables} when $b = 1$, \hbox{$g = 1_{[0,1]}$}, $N = 40$, $n_d = 300$, for $\alpha$ chosen with the method \textit{SJ-ste} of the package \textit{logKDE} of R, for the initial distributions defined in~\eqref{eq:definition_other_initial_distributions}. Comparison with the results obtained with the estimator~$\overline{n}_0^{(\text{old},N,\alpha)}$ defined in~\eqref{eq:estimator_n0_previous_article} and constructed in~\cite{olaye_transport_2026}, for the same value of $\alpha$. \textit{For plotting the curves, the numerical precision of the computations was set to $200$ digits.}}\label{fig:estimation_gaver_stehfest_other_distributions}
\end{figure}

\section{Conclusion and perspectives}\label{sect:discussion} 

The objective of this work was to improve the estimation method presented in~\cite{olaye_transport_2026}. To do this, we first proved that the error between~\eqref{eq:scaled_model} and~\eqref{eq:approximation_transport_diffusion} tends to zero when \hbox{$N\rightarrow+\infty$}, with a higher speed than the one of the error between~\eqref{eq:scaled_model} and the approximated model in~\cite{olaye_transport_2026}, see~Theorem~\ref{te:model_approximation} and~\cite[Proposition~$3.2$]{olaye_transport_2026}. The latter and the fact that the Laplace transform of $n_0$ can be computed using the Laplace transform of $u_{\partial}^{(N)}$ for any $N>0$ suggested that it is possible to perform a Laplace transform inversion to improve our estimation method. Our estimations done in Figure~\ref{fig:estimation_gaver_stehfest_noisefree} showed that, when we observe $n_{\partial}^{(N)}$ without noise, this Laplace transform inversion works perfectly well on most biologically relevant initial distributions. The estimations done in Figures~\ref{fig:estimation_gaver_stehfest_random_variables},~\ref{fig:estimation_gaver_stehfest_variation_number_simulations} and~\ref{fig:estimation_gaver_stehfest_other_distributions} have also shown that when there is noise on $n_{\partial}^{(N)}$ related to sampling, the Laplace transform inversion still improves the estimation obtained with the estimator constructed in~\cite{olaye_transport_2026}. We thus have several satisfactions with this new estimation~method. 

Despite these encouraging results, our estimation method has also several limitations. The first one is that despite the improvement of the estimation when the initial distribution has a small variability observed in Figure~\ref{fig:estimation_gaver_stehfest_noisefree}, there are still issues to estimate the initial distribution when it has an extremely small variability, as shown in Figures~\ref{fig:estimation_verysmall_variability_third} and~\ref{fig:estimation_verysmall_variability_fourth}. The second one is that when there is noise on $n_{\partial}^{(N)}$ related to sampling, the improvement obtained with the Laplace transform inversion is not as satisfying as the one we have when we observe noise-free data, see Figures~\ref{fig:estimation_gaver_stehfest_random_variables},~\ref{fig:estimation_gaver_small_number_data_first}-\ref{fig:estimation_gaver_small_number_data_third} and~\ref{fig:estimation_gaver_stehfest_other_distributions}. However, the type of initial distribution studied in Figures~\ref{fig:estimation_verysmall_variability_third} and~\ref{fig:estimation_verysmall_variability_fourth} is not biologically relevant, since the variability of the initial distribution is too small. In addition, the second limitation is not critical in practice, since this is unavoidable that the estimation loses in quality when there is noise on the observation of $n_{\partial}^{(N)}$. The estimation method developed in this work is therefore already useful in practice, despite these limitations. 

The main results we still need to achieve now correspond to theoretical results providing bounds on the error between $\widehat{n}_0^{(N,K)}$ and $n_0$, and between $\overline{n}_0^{(N,K,\alpha)}$ and $n_0$ for all $N>0$ sufficiently large, $K\in\mathbb{N}^*$ and $\alpha \in\left(0,\log(2)^{\frac{1}{2}}\right)$. Precisely, we would like to prove that these errors converge to $0$ when $N\rightarrow+\infty$, $K\rightarrow+\infty$, and when also~$n_d\rightarrow+\infty$ and $\alpha \rightarrow 0$ for the error between $\overline{n}_0^{(N,K,\alpha)}$ and $n_0$. It has been proved \hbox{in~\cite[Theorem~$2.7$-$(a)$ and Proposition~$5.2$]{olaye_transport_2026}} that the estimation errors converge to~$0$ at a rate~$\frac{1}{N}$ for the estimators constructed in~\cite{olaye_transport_2026}. We would like to prove that the estimation errors for the estimators defined in~\eqref{eq:estimator_gaver_stehfest_noise_free} and~\eqref{eq:estimator_gaver_stehfest_random_variables} converge to~$0$ at a rate~$\frac{1}{N^2}$, as shown in Figure~\ref{fig:curve_estimation_errors_logscale}, to justify rigorously that the Laplace transform inversion improves the estimation obtained with the estimator constructed in~\cite{olaye_transport_2026}. We would also like to obtain information on the impact of the number of random variables $n_{d}$ on the estimation error, to better understand how the noise related to sampling affects the quality of estimation. Obtaining these bounds corresponds to a work in progress.

Other perspectives emerge from this work. The first would be to know whether it is possible to obtain model approximations more precise than those obtained in this article, in order to develop estimators of better quality. This perspective would be very interesting, in particular to improve the estimation results presented in Figure~\ref{fig:estimation_verysmall_variability}. The second perspective would be to know if it is possible to develop a method to estimate the smoothing parameter specific to our problem, rather than using the R package \textit{logKDE}. Developing such a method could perhaps improve the estimation results when observing variables distributed according to~$n_{\partial}^{(N)}$, see Figures~\ref{fig:estimation_gaver_stehfest_random_variables},~\ref{fig:estimation_gaver_stehfest_variation_number_simulations}, and~\ref{fig:estimation_gaver_stehfest_other_distributions}. Finally, it would be interesting to see how the estimation methods developed in this work and in~\cite{olaye_transport_2026} can be adapted to age-structured models such as those in~\hbox{\cite{olaye_long-time_2026,benetos_branching_2025}}, which are more biologically relevant. All these perspectives correspond to future~works.

\appendix

\section{Proof of Proposition~\ref{prop:link_laplace_transforms}}\label{appendix:proof_link_laplace}

Let $N > 0$. We denote the set 
$$
\mathcal{P}'_N := \left\{p\in\mathbb{C}\,|\,\text{Re}(p) > 0 \text{ and } Re(q_N(p)) > \max\left(0,\mathcal{R}\left(u_{\partial}^{(N)} \right), - (bm_1)^2\frac{2N}{bm_2}\right)\right\}. 
$$
The proof of Proposition~\ref{prop:link_laplace_transforms} is done in three steps. First, in Step~\hyperlink{paragraph:step1_proof_link_laplace}{$1$}, we prove the following auxiliary result, necessary to obtain~\eqref{eq:link_laplace_transforms}, for all~$t\geq0$ and~$p \in \mathcal{P}'_N$,
\begin{equation}\label{eq:step1_proof_link_laplace}
	\begin{aligned}
		\frac{\dd}{\dd t}\left(\mathcal{L}(u^{(N)}(t,.))(p)\right) &= q_N(p)\mathcal{L}(u^{(N)}(t,.))(p) - u_{\partial}^{(N)}(t)\\ 
		&- bm_1p\int_0^t \exp\left((bm_1)^2\frac{2N}{bm_2}(s-t)\right)u^{(N)}_{\partial}(s)\dd s.
	\end{aligned}
\end{equation}
Then, in Step~\hyperlink{paragraph:step2_proof_link_laplace}{$2$}, we prove that~\eqref{eq:link_laplace_transforms} is true for all $p\in \mathcal{P}'_N$. Finally, in Step~\hyperlink{paragraph:step3_proof_link_laplace}{$3$}, we extend this result to the complex numbers $p\in \mathcal{P}_N$ thanks to the analytic continuation theorem. 

\paragraph{Step $1$:}\hypertarget{paragraph:step1_proof_link_laplace}{} Let us fix $t\geq0$ and $p\in \mathcal{P}'_N$. By multiplying the first line of \eqref{eq:approximation_transport_diffusion} with $\exp\left(-px\right)$, and then  integrating in~$\dd x$, we have 
\begin{equation}\label{eq:proof_link_laplace_transform_intermediate_first}
	\frac{\dd}{\dd t}\left(\mathcal{L}(u^{(N)}(t,.))(p)\right) = bm_1\mathcal{L}\left(\partial_xu^{(N)}(t,.)\right)(p) + \frac{bm_2}{2N}\mathcal{L}\left(\partial_{x}^2u^{(N)}(t,.)\right)(p).
\end{equation}
Therefore, by first using that for all $f\in H^2\left(\mathbb{R}_+\right)$, it holds $\mathcal{L}(f')(p) = p\mathcal{L}(f)(p) - f(0)$ and~\hbox{$\mathcal{L}(f'')(p) = p^2\mathcal{L}(f)(p) - pf(0) - f'(0)$}, see~\cite[p.$264$]{folland_1992}, and then applying~\eqref{eq:definition_qN} and the third line of~\eqref{eq:approximation_transport_diffusion}, we obtain 
\begin{equation}\label{eq:proof_link_laplace_transform_intermediate_second}
	\begin{aligned}
		\frac{\dd}{\dd t}\left(\mathcal{L}\left(u^{(N)}(t,.)\right)(p)\right) &= \left(bm_1p+ \frac{bm_2}{2N}p^2\right)\mathcal{L}\left(u^{(N)}(t,.)\right)(p) -bm_1u^{(N)}(t,0)\\ 
		&- \frac{bm_2p}{2N}u^{(N)}(t,0) -\frac{bm_2}{2N}\partial_xu^{(N)}(t,0) \\ 
		&= q_N(p)\mathcal{L}\left(u^{(N)}(t,.)\right)(p) - u_{\partial}^{(N)}(t) - \frac{bm_2p}{2N}u^{(N)}(t,0).
	\end{aligned}
\end{equation}
Observe that by plugging the second line of~\eqref{eq:approximation_transport_diffusion} in its third line, and multiplying everything by $bm_1\frac{2N}{bm_2}$, we have 
$$
bm_1\frac{2N}{bm_2}u_{\partial}^{(N)}(t) = \left(bm_1\right)^2\frac{2N}{bm_2}u^{(N)}(t,0) + \partial_tu^{(N)}(t,0).
$$
This is a classical first order linear differential equation. We can solve it using the usual formula for this type of equation, see for example~\cite[Theorem~$2$, p.$41$]{coddington_introduction_1989}, to obtain that (we recall that~$u^{(N)}(0,0) = n_0(0) = 0$)
$$
u^{(N)}(t,0) = bm_1\frac{2N}{bm_2}\int_0^t \exp\left((bm_1)^2\frac{2N}{bm_2}(s-t)\right)u^{(N)}_{\partial}(s)\dd s.
$$
Then, by plugging the latter in \eqref{eq:proof_link_laplace_transform_intermediate_second}, we obtain~\eqref{eq:step1_proof_link_laplace}.

\paragraph{Step $2$:}\hypertarget{paragraph:step2_proof_link_laplace}{}  To prove~\eqref{eq:link_laplace_transforms} for all $p\in \mathcal{P}'_N$, we begin by obtaining the equation verified by the function $\left(t,p\right) \mapsto \exp\left(-q_N(p)t\right)\mathcal{L}(u^{(N)}(t,.))(p)$, which is equal to $p\mapsto \mathcal{L}(n_0)(p)$ when $t=0$. First, we multiply both sides of~\eqref{eq:step1_proof_link_laplace} by~$\exp\left(-q_N(p)t\right)$. Then, we put the first term of the right-hand side to the left-hand side, in view of the following equality, for all $t\geq0$, $p\in \mathcal{P}'_N$, and~$f\in C^1\left(\mathbb{R}_+\right)$,
$$
\exp\left(-q_N(p)t\right)\frac{\dd }{\dd t}f(t) - q_N(p)\exp\left(-q_N(p)t\right)f(t) = \frac{\dd }{\dd t}\left(\exp\left(-q_N(p)t\right)f(t)\right). 
$$
We obtain that for all $t\geq0$ and $p\in \mathcal{P}'_N$
$$
\begin{aligned}
	\frac{\dd }{\dd t}\left[\exp\left(-q_N(p)t\right)\mathcal{L}\left(u^{(N)}(t,.)\right)(p)\right] &=  -\exp\left(-q_N(p)t\right)u_{\partial}^{(N)}(t) \\ 
	&\hspace{-9mm}-bm_1p\exp\left(-q_N(p)t\right)\left[\int_0^t \exp\left((bm_1)^2\frac{2N}{bm_2}(s-t)\right) u_{\partial}^{(N)}(s) \dd s\right].
\end{aligned}
$$
From the above equality, we conclude the step. To do so, we first integrate in $\dd t$, in view of the fact that for all $p\in\mathbb{C}$ such that $Re(p)> 0$ and $Re(q_N(p))>0$, it holds by Theorem~\ref{te:model_approximation}-\ref{te:model_approximation_first} and~Remark~\ref{rem:mass_conservation} (the condition $N > \frac{\lambda m_2}{2m_1}$ is not an issue, since~\eqref{eq:assumptions_main_result_chaplaplace} is also verified for a smaller coefficient in the exponentials than $\lambda$)
$$
\begin{aligned}
	\underset{t\rightarrow+\infty}{\limsup} \left|\exp\left(-q_N(p)t\right)\mathcal{L}\left(u^{(N)}(t,.)\right)(p)\right| &\leq \underset{t\rightarrow+\infty}{\limsup} \left[\exp\left(-q_N(p)t\right)\left|\left|\left(u^{(N)}-n^{(N)}\right)(t,.)\right|\right|_{L^1\left(\mathbb{R}_+\right)}\right] \\
	&+ \underset{t\rightarrow+\infty}{\limsup} \left[\exp\left(-q_N(p)t\right)\left|\left| n^{(N)}(t,.)\right|\right|_{L^1\left(\mathbb{R}_+\right)}\right] = 0. 
\end{aligned}
$$
We obtain that for all $p\in\mathcal{P}'_N$
\begin{equation}\label{eq:proof_link_laplace_transform_intermediate_third}
	\begin{aligned}
		-\mathcal{L}\left(n_0\right)(p) &=  - \mathcal{L}\left(u_{\partial}^{(N)}\right)\left(q_N(p)\right) \\ 
		&-bm_1p\int_0^{+\infty}\exp\left(-q_N(p)t\right)\left[\int_0^t \exp\left((bm_1)^2\frac{2N}{bm_2}(s-t)\right)u_{\partial}^{(N)}(s)\dd s\right] \dd t.
	\end{aligned}
\end{equation}
Then, we observe that by changing the bound of the integral below, we have for all $p\in\mathcal{P}'_N$
\begin{equation}\label{eq:proof_link_laplace_transform_intermediate_fourth}
	\begin{aligned}
		&\int_0^{+\infty}\exp\left(-q_N(p)t\right)\left[\int_0^t \exp\left((bm_1)^2\frac{2N}{bm_2}(s-t)\right)u_{\partial}^{(N)}(s)\dd s\right] \dd t \\ 
		&= \int_0^{+\infty}\exp\left((bm_1)^2\frac{2N}{bm_2}s\right)\left[\int_s^{+\infty} \exp\left[-\left((bm_1)^2\frac{2N}{bm_2} + q_N(p)\right)t\right]\dd t\right]u_{\partial}^{(N)}(s) \dd s \\
		&= \frac{1}{(bm_1)^2\frac{2N}{bm_2} + q_N(p)}\int_0^{+\infty}\exp\left(-q_N(p)s\right)u_{\partial}^{(N)}(s) \dd s = \frac{\mathcal{L}\left(u_{\partial}^{(N)}\right)\left(q_N(p)\right)}{(bm_1)^2\frac{2N}{bm_2} + q_N(p)}.
	\end{aligned}
\end{equation}
Finally, we combine~\eqref{eq:proof_link_laplace_transform_intermediate_third} and~\eqref{eq:proof_link_laplace_transform_intermediate_fourth}. We obtain that~\eqref{eq:link_laplace_transforms} is true for all $p\in\mathcal{P}'_N$. 

\paragraph{Step $3$:}\hypertarget{paragraph:step3_proof_link_laplace}{} Now, we aim to extend this equality for the largest possible set, namely the set~$\mathcal{P}_N$ defined in~\eqref{eq:definition_P}. First, notice that $\mathcal{L}(n_0)$ and $\varphi:p \mapsto \left(1 + p \frac{bm_1}{(bm_1)^2\frac{2N}{bm_2} + q_N(p)}\right)\mathcal{L}\left(u_{\partial}^{(N)}\right)(q_N(p))$ are both analytic on $\mathcal{P}_N$. Then, notice that $z \in \mathcal{P}_N\mapsto \mathcal{L}(n_0)(z)$ is an analytic continuation of 
$$
z \in \mathcal{P}'_N \mapsto \mathcal{L}(n_0)(z),
$$
and that $z \in \mathcal{P}_N \mapsto \varphi(z)$ is an analytic continuation of 
$$
z \in \mathcal{P}'_N \mapsto \varphi(z) = \mathcal{L}(n_0)(z).
$$
By uniqueness of the analytic continuation, $z \in \mathcal{P}_N \mapsto \mathcal{L}(n_0)(z)$ and $z \in \mathcal{P}_N \mapsto \varphi(z)$ are equal. Then, from this last result, we have that~\eqref{eq:link_laplace_transforms} holds for all $p \in \mathcal{P}_N$. \qed

\section{Proof of Proposition~\ref{prop:uniqueness_dirichlet}}\label{appendix:proof_uniqueness_dirichlet}

This proof is inspired by what is done in the proof of~\cite[Proposition~$2.2$]{doumic_asymptotic_2026}, and its main argument is the Lumer-Phillips theorem, see~\cite[Theorem $2.6$, p.$106$]{bensoussan_representation_2007}. Let us denote the operator~\hbox{$A : H_0^1\left(\mathbb{R}_+\right)\cap H^2\left(\mathbb{R}_+\right) \rightarrow L^2\left(\mathbb{R}_+\right)$}, defined for all $u\in H_0^1\left(\mathbb{R}_+\right)\cap H^2\left(\mathbb{R}_+\right)$ as 
$$
A\left(u\right) = bm_1u' + \frac{bm_2}{2N}u''.
$$
For all $u\in H_0^1\left(\mathbb{R}_+\right)\cap H^2\left(\mathbb{R}_+\right)$, as $\left(u'u\right)' = u''u + \left(u'\right)^2 \geq u''u$, and as $u(0) = 0$, we have that 
$$
\begin{aligned}
	\langle A\left(u\right),u\rangle_{L^2\left(\mathbb{R}_+\right)} &= bm_1 \int_0^{+\infty} u'(x)u(x) \dd x +\frac{bm_2}{2N}\int_0^{+\infty} u''(x)u(x) \dd x \\
	&\leq \frac{bm_1}{2} \int_0^{+\infty} \left(u^2\right)'(x) \dd x + \frac{bm_2}{2N}\int_0^{+\infty}\left(u'u\right)'(x) \dd x = 0.
\end{aligned}
$$
\noindent Then, $A$ is a diffusive operator. We thus only have to prove that $A$ is maximal, and the assumptions of the Lumer-Phillips theorem will be verified. To prove this, we introduce for all~$\lambda >0$ the bilinear form $a_{\lambda}: \left(H_0^1\left(\mathbb{R}_+\right)\right)^2\rightarrow \mathbb{R}$, defined for all $(u,v)\in \left(H_0^1\left(\mathbb{R}_+\right)\right)^2$ as 
$$
a_{\lambda}(u,v) = \lambda \langle u,v\rangle_{L^2\left(\mathbb{R}_+\right)} - bm_1 \langle u',v\rangle_{L^2\left(\mathbb{R}_+\right)} + \frac{bm_2}{2N} \langle u',v'\rangle_{L^2\left(\mathbb{R}_+\right)}.
$$
We have that $a_{\lambda}(u,u) = \lambda \left|\left| u\right|\right|^2_{L^2\left(\mathbb{R}_+\right)} + \frac{bm_2}{2N} \left|\left|u'\right|\right|^2_{L^2\left(\mathbb{R}_+\right)}$ for all $\lambda > 0$ and $u\in H_0^1\left(\mathbb{R}_+\right)$, so that $a_{\lambda}$ is coercive. This yields by the Lax-Milgram theorem that for all $\lambda > 0$, $f\in L^2\left(\mathbb{R}_+\right)$, there exists a unique $u_{\lambda,f}\in H_0^1\left(\mathbb{R}_+\right)$ such that $a_{\lambda}(u_{\lambda,f},v) = \langle f,v\rangle_{L^2\left(\mathbb{R}_+\right)}$ for all $v\in H_0^1\left(\mathbb{R}_+\right)$. In addition, this function $u_{\lambda,f}$ belongs to $H^2\left(\mathbb{R}_+\right)$, as it holds for all $h\in \mathcal{C}_c\left(\mathbb{R}_+^*\right)$
$$
\frac{bm_2}{2N} \langle u''_{\lambda,f}, h\rangle_{L^2\left(\mathbb{R}_+\right)} = \lambda \langle u_{\lambda,f},h\rangle_{L^2\left(\mathbb{R}_+\right)} - bm_1 \langle u_{\lambda,f}',h\rangle_{L^2\left(\mathbb{R}_+\right)} -  \langle f,h\rangle_{L^2\left(\mathbb{R}_+\right)},
$$
with $\lambda u_{\lambda,f} - bm_1 u_{\lambda,f}' - f \in L^2\left(\mathbb{R}_+\right)$. Then, from these two points, we have that for all $\lambda > 0$ and $f\in L^2\left(\mathbb{R}_+\right)$, there exists $u_{\lambda,f}\in H_0^1\left(\mathbb{R}_+\right)\cap H^2\left(\mathbb{R}_+\right)$ such that $\lambda u_{f,\lambda}- A\left(u_{f,\lambda}\right)= f$, so that $A$ is maximal. As $A$ is both maximal and dissipative, and as $H_0^1\left(\mathbb{R}_+\right)\cap H^2\left(\mathbb{R}_+\right)$ is dense in~$L^2\left(\mathbb{R}_+\right)$, we have by the Lumer-Phillips theorem that the proposition is proved. \qed

\section{Proof of Proposition~\ref{prop:explicit_laplace_erlang}}\label{appendix:explicit_laplace_erlang}

In all this proof, we denote for all $n\in\mathbb{N}$, $k\in\llbracket 0, n\rrbracket$, $x\in\mathbb{R}^{n-k+1}$, the polynomial 
\begin{equation}\label{eq:dft_incomplete_Bell_polynomial}
	B_{n,k}(x_1, \dots, x_{n-k+1}) = \sum_{\substack{j_1,\hdots, j_{n-k+1}\geq 0  \\ j_1 + j_2 + \dots + j_{n-k+1} = k \\ j_1 + 2j_2+ \hdots + (n-k+1)j_{n-k+1} = n}} \frac{n!}{j_1! \dots j_{n-k+1}!} \prod_{i = 1}^{n-k+1}\left(\frac{x_{i}}{i!}\right)^{j_i},
\end{equation}	
which is called in the literature \textit{partial Bell polynomial} of weight $n$ and degree $k$, see~\cite[p.$134$, Eq.~$3$d]{comtet_advanced_1974}. We also denote the following polynomial, for all $n\in\mathbb{N}$, $x\in\mathbb{R}^n$,
\begin{equation}\label{eq:dft_complete_Bell_polynomial}
	B_n(x_1, \dots, x_n) = \sum_{k=0}^{n} B_{n,k}(x_1, \dots, x_{n-k+1}),
\end{equation}
which is called \textit{complete Bell polynomial} of weight $n$, see~\cite[p.$134$, Eq.~$3$c]{comtet_advanced_1974}. We finally define the following function and constant, that allow us to simplify notations, for all $N > 0$ and $\left(t,x\right)\in\mathbb{R}_+\times\mathbb{R}_+$, 
$$
\begin{aligned}
	\varphi_{\ell}\left(t,x\right)  := B_{\ell-1}\bigg[\!-b\mathcal{L}\left(g\text{Id}\right)\left(\frac{\beta}{N}\right)t-x,&\left(-1\right)^{2}\frac{b}{N}\mathcal{L}\left(g\text{Id}^2\right)\left(\frac{\beta}{N}\right)t,\hdots,\\
	&\left(-1\right)^{\ell-1}\frac{b}{N^{\ell - 2}}\mathcal{L}\left(g\text{Id}^{\ell - 1}\right)\left(\frac{\beta}{N}\right)t\bigg],\\
	\tilde{\beta}_N := \frac{N}{m_1}	\left[1 - \mathcal{L}(g)\left(\frac{\beta}{N}\right)\right].\hspace{11.5mm}&
\end{aligned}
$$
By combining~\cite[Proposition~A.$9$]{olaye_transport_2026} with the fact that $\frac{\dd^{n}}{\dd \alpha^{n}}\left(\mathcal{L}(g)\right) = \left(-1\right)^n\mathcal{L}\left(g\text{Id}^n\right)$ for all $n\in\mathbb{N}$, see~\hbox{\cite[p.$264$, $8.$]{folland_1992}}, we have for all $N>0$ and $(t,x)\in\mathbb{R}_+\times\mathbb{R}_+$ that
\begin{equation}\label{eq:proof_prop_explicit_erlang_first}
	n^{(N)}(t,x)= (-1)^{\ell -1}\frac{\beta^\ell}{\left(\ell -1\right)!}\varphi_{\ell}\left(t,x\right)\exp\left(-bm_1 \tilde{\beta}_Nt-\beta x\right). 
\end{equation}
In addition, in view of~\cite[Theorem $2$-$(i)$]{wheeler_bell_1987},~\eqref{eq:dft_complete_Bell_polynomial} and~\cite[p.$135$]{comtet_advanced_1974}, the two following equalities hold, for all $ (y,z)\in\left(\mathbb{R}^{\ell - 1}\right)^2$, $\gamma\in\mathbb{R}$, $\omega\in\mathbb{R}$, 
\begin{equation}\label{eq:proof_prop_explicit_erlang_first_bis}
	\begin{aligned}
		B_{\ell-1}\left[y_1 + \gamma \omega z_1, \dots, y_{\ell-1} +  \gamma \omega^{\ell-1} z_{\ell-1}\right] &= \sum_{j=0}^{\ell-1} \binom{\ell-1}{j} B_{\ell-1-j}\left[y_1, \dots, y_{\ell-1-j}\right] \\
		&\times B_{j}\left[\gamma \omega z_1, \dots, \gamma \omega^j z_{j}\right],\\
		\forall j\in\llbracket1,\ell-1\rrbracket:\hspace{3mm} B_{j}\left[\gamma \omega z_1, \dots, \gamma \omega^j z_{j}\right] &= \omega^j\sum_{i=0}^{j} \gamma^iB_{j,i}\left[z_1, \dots, z_{i}\right].
	\end{aligned}
\end{equation}
Then, by using~\eqref{eq:proof_prop_explicit_erlang_first_bis} for $y = (-x,0,\hdots,0)$, $z = \left(b\mathcal{L}\left(g\text{Id}\right)\left(\frac{\beta}{N}\right),\hdots,\frac{b}{N^{\ell-2}}\mathcal{L}\left(g\text{Id}^{\ell-1}\right)\left(\frac{\beta}{N}\right)\right)$, $\gamma = t$ and~\hbox{$\omega = -1$} to develop the function $\varphi_\ell$ in~\eqref{eq:proof_prop_explicit_erlang_first}, we obtain that for all $N>0$, \hbox{$(t,x)\in\mathbb{R}_+\times\mathbb{R}_+$},
\begin{equation}\label{eq:proof_prop_explicit_erlang_second}
	\begin{aligned}
		n^{(N)}(t,x) &= (-1)^{\ell -1}\frac{\beta^\ell}{\left(\ell -1\right)!}\sum_{j=0}^{\ell-1}\sum_{i=0}^{j} \binom{\ell-1}{j} B_{\ell-1-j}\left[-x,0, \dots, 0\right] \left(-1\right)^{j}t^{i} \\
		&\times B_{j,i}\left[b\mathcal{L}\left(g\text{Id}\right)\left(\frac{\beta}{N}\right),\hdots,\frac{b}{N^{j-1}}\mathcal{L}\left(g\text{Id}^{j}\right)\left(\frac{\beta}{N}\right)\right]\exp\left(-bm_1 \tilde{\beta}_Nt-\beta x\right). 
	\end{aligned}
\end{equation}
In view of~\cite[p.$136$, Eq.~$3n'$]{comtet_advanced_1974} and~\eqref{eq:dft_complete_Bell_polynomial}, we have that $B_{\ell-1-j}\left[-x,0, \dots, 0\right] = \left(-1\right)^{\ell-1-j}x^{\ell-1-j}$ for all~\hbox{$x\in\mathbb{R}_+$}, $j\in\llbracket0,\ell-1\rrbracket$. Therefore, by plugging this last equality in~\eqref{eq:proof_prop_explicit_erlang_second}, and then using the second line of~\eqref{eq:scaled_model} to compute $n_{\partial}^{(N)}$, we obtain that for all $N>0$, $t >0$,
$$
\begin{aligned}
	n_{\partial}^{(N)}(t) &= b\frac{\beta^\ell}{\left(\ell -1\right)!}\sum_{j=0}^{\ell-1}\sum_{i=0}^{j} \binom{\ell-1}{j}\frac{1}{N^{\ell-1-j}} \mathcal{L}\left((1-G)\text{Id}^{\ell-1-j}\right)\left(\frac{\beta}{N}\right)  \\
	&\times B_{j,i}\left[b\mathcal{L}\left(g\text{Id}\right)\left(\frac{\beta}{N}\right),\hdots,\frac{b}{N^{j-1}}\mathcal{L}\left(g\text{Id}^{j}\right)\left(\frac{\beta}{N}\right)\right]t^{i}\exp\left(-bm_1 \tilde{\beta}_Nt\right).
\end{aligned}
$$
The proposition is thus proved by switching the two sums in the above, and then using the fact that in view of~\cite[p.$264$, $12.$]{folland_1992}, it holds for all~$i\in\llbracket0,\ell-1\rrbracket$ and $p\in\mathbb{C}$ verifying $\text{Re}(p) > -bm_1 \tilde{\beta}_N$
$$
\mathcal{L}\left(\text{Id}^i \exp\left(-bm_1 \tilde{\beta}_N\text{Id}\right)\right)(p) = \frac{i!}{\left(p +bm_1 \tilde{\beta}_N\right)^{i+1}}.
$$
\qed

\section{Why we do not take \texorpdfstring{$K$}{K} very large in Section~\ref{subsubsect:estimation_results_random_variables}?}\label{sect:round_off_errors}

When we use~\eqref{eq:estimator_gaver_stehfest_random_variables} to estimate $n_0$ in Figures~\ref{fig:estimation_gaver_stehfest_random_variables},~\ref{fig:estimation_gaver_stehfest_variation_number_simulations} and~\ref{fig:estimation_gaver_stehfest_other_distributions}, we do not take $K$ very large. This may seem contradictory with~\eqref{eq:gaver_stehfest_formula} at first sight, which states that the Gaver-Stehfest algorithm converges as $K\rightarrow +\infty$. In fact, it has been proved in~\cite[Theorem~$1.1$]{kuznetsov_rate_2022} that the Gaver-Stehfest algorithm converges exponentially fast as $K\rightarrow +\infty$ when the inverse Laplace transform is regular enough, so that the speed of convergence is very quick. Hence, taking $K$ around $10$ or~$15$ in Figures~\ref{fig:estimation_gaver_stehfest_random_variables},~\ref{fig:estimation_gaver_stehfest_variation_number_simulations} and~\ref{fig:estimation_gaver_stehfest_other_distributions} is sufficient to obtain a good approximation.

Taking a larger parameter $K$ for the estimations done in Section~\ref{subsubsect:estimation_results_random_variables} can even create big estimation errors related to numerical instability. That is why, we do not take it larger in these estimations. To illustrate the issues related to taking $K$ very large, we have represented in Figure~\ref{fig:illustration_round_off_errors} exactly the same estimations as those done in Figures~\ref{fig:estimation_gaver_stehfest_random_variables_intermediate_first} and~\ref{fig:estimation_gaver_stehfest_random_variables_intermediate_third}, except that we now take $K = 18$ in both estimations. We observe in Figure~\ref{fig:illustration_round_off_errors} that for $K = 18$, both estimations are now very poor. Specifically, oscillations with large amplitude appear at the estimated curve. These oscillations correspond in fact to numerical errors. Indeed, as the Gaver-Stehfest algorithm corresponds to the sum of large terms with alternating sign, see~\eqref{eq:gaver_stehfest_formula}, round-off errors appear when we use this algorithm. The absolute value of the terms we sum in the Gaver-Stehfest algorithm increases when its parameter $K$ increases. Hence, increasing $K$ also increases the impact of the round-off errors on the estimation, and that is why we do not take it larger in Figures~\ref{fig:estimation_gaver_stehfest_random_variables},~\ref{fig:estimation_gaver_stehfest_variation_number_simulations}, and~\ref{fig:estimation_gaver_stehfest_other_distributions}. 

\begin{figure}[!htb]
	\centering
	\begin{subfigure}[t]{0.485\textwidth}
		\centering
		\includegraphics[scale = 0.41]{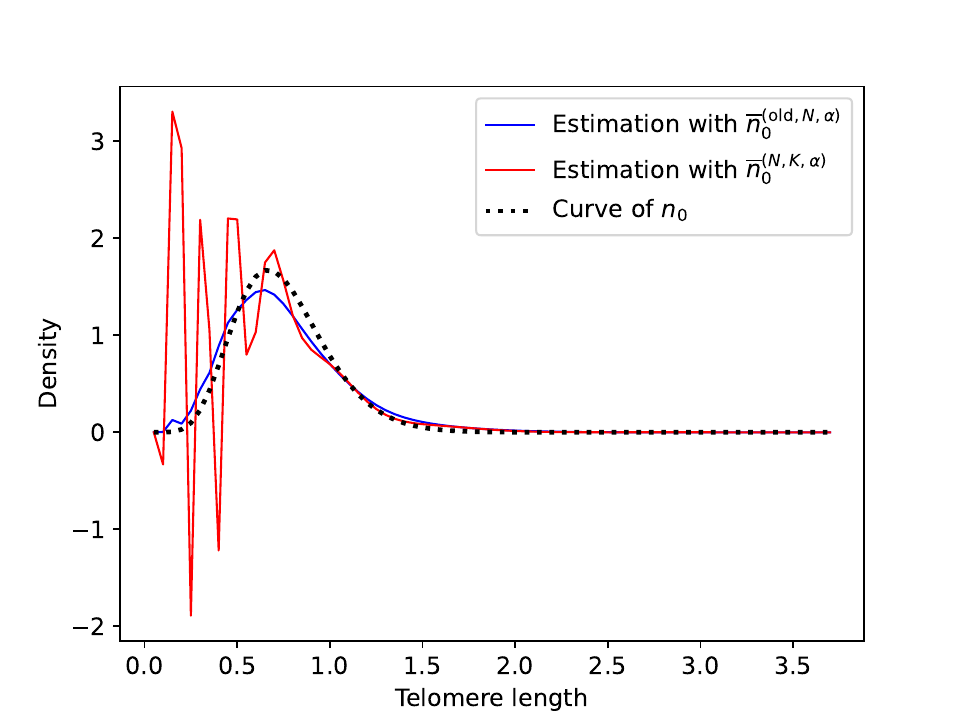}
		\caption{Estimation for $n_0 = h_{9,12}$ and $K = 18$.}
	\end{subfigure}
	\hfill
	\begin{subfigure}[t]{0.485\textwidth}
		\centering
		\includegraphics[scale = 0.41]{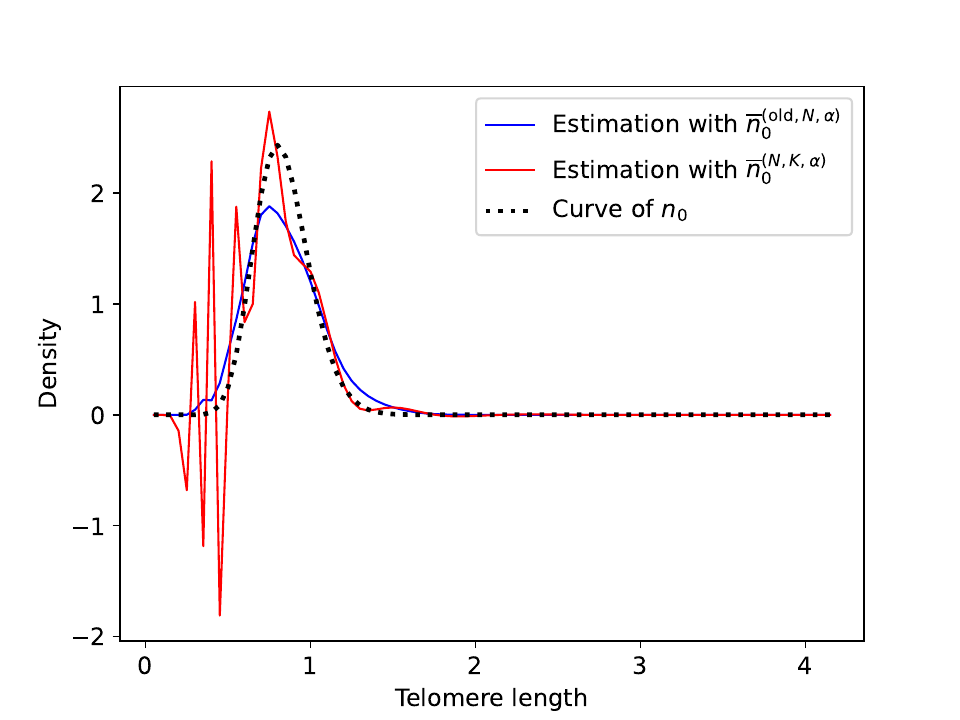}
		\caption{Estimation for $n_0 = h_{25,30}$ and $K = 18$.}
	\end{subfigure}
	\caption{Estimation results in the same setting as in Figures~\ref{fig:estimation_gaver_stehfest_random_variables_intermediate_first} and~\ref{fig:estimation_gaver_stehfest_random_variables_intermediate_second}, except that $K = 18$. \textit{For plotting the curves, the numerical precision of the computations was set to $200$ digits.}}\label{fig:illustration_round_off_errors}
\end{figure}

The problem of round-off errors is in fact very common when working with the Gaver-Stehfest algorithm. The main way to manage this issue is to increase the numerical precision of the computations. In our case, we have set the precision to $200$ decimal digits. We do not use a higher precision because it would be too computationally expensive. One can notice that in Figures~\ref{fig:estimation_gaver_stehfest_noisefree} and~\ref{fig:estimation_verysmall_variability}, we have succeeded to take a large parameter $K$. In the following work of the author~\cite{olaye_estimation_2025} in which the Gaver-Stehfest algorithm is used, see in particular~\cite[Appendix~$6.1.2$]{olaye_estimation_2025}, we have also succeeded in taking a large parameter~$K$. The reason is that the function $\overline{n}_{\partial}^{(N,\alpha)}$, present in the definition of $\overline{n}_0^{(N,K,\alpha)}$ in~\eqref{eq:estimator_gaver_stehfest_random_variables}, and introduced in~\eqref{eq:estimation_cemetery_gamma} corresponds to the sum of the smoothed approximation of Dirac measures, see Section~\ref{subsubsect:estimation_gaver_stehfest_random_variables}. Thus, even if its curve is smooth, its Laplace transform has a more sophisticated form than the ones of the functions in Figures~\ref{fig:estimation_gaver_stehfest_noisefree} and~\ref{fig:estimation_verysmall_variability}, or in~\cite[Appendix~$6.1.2$]{olaye_estimation_2025}. The issues of round-off errors for the Gaver-Stehfest algorithm thus strongly depend on the function for which we apply this algorithm. %

\paragraph{Acknowledgements.} This work was partially funded by the Fondation Mathématique Jacques Hadamard and by the European Union ERC-2024-COG MUSEUM-101170884. Additional support was provided by the France 2030 program, administered by the French National Research Agency (ANR), under grant ANR-23-EXMA-0005. Views and opinions expressed are however those of the author(s) only and do not necessarily reflect those of the European Union or the European Research Council Executive Agency (ERCEA). Neither the European Union nor the granting authority can be held responsible for them. The author warmly thanks Marie Doumic for her valuable guidance during the realisation of this work.

\printbibliography
\end{document}